\documentclass[a4paper,12pt]{article}

\usepackage{mathtools, bbm, bm,array,amsmath,amsthm}
\usepackage{tikz}
\usepackage{caption}
\usepackage{environ}
\usepackage{graphicx}
\usetikzlibrary{decorations.markings, calc, angles, quotes, arrows.meta, positioning, shapes, patterns, patterns.meta}

\usepackage[ragged, hang]{footmisc}

\makeatletter
\newcommand\footnoteref[1]{\protected@xdef\@thefnmark{\ref{#1}}\@footnotemark}
\makeatother

\newtheorem{theorem}{Theorem}%  
\newtheorem{lemma}[theorem]{Lemma}% 
\newtheorem{corollary}[theorem]{Corollary}% 
\newtheorem{definition}[theorem]{Definition}%

\let\originalleft\left
\let\originalright\right
\renewcommand{\left}{\mathopen{}\mathclose\bgroup\originalleft}
\renewcommand{\right}{\aftergroup\egroup\originalright}

\mathcode`\;="303B
\newcommand{\semcol}{\hspace{-.08cm} ; \!}

\newcommand{\vol}{\text{\normalfont  vol}}
\newcommand{\leffdim}{\textit{\underbar n}}
\newcommand{\ueffdim}{\bar n}

\newcommand{\jrdn}{\bar d_k}
\newcommand{\newldeff}{\mathfrak{d}}
\newcommand{\newudeff}{\mathfrak{d}_+}

\usepackage{pdfsync}
\usepackage{amsmath}
\usepackage{amsfonts}
\usepackage{amssymb}
\usepackage{amsthm}
\usepackage{mathrsfs}
\newtheorem*{lemma*}{Lemma}
\newtheorem*{proposition*}{Proposition}
\newtheorem*{theorem*}{Theorem}
\usepackage[normalem]{ulem}
\usepackage{extarrows}
\usepackage{xcolor}
\usepackage[margin=1cm]{caption}
\usepackage{subcaption}
\usepackage[inline]{enumitem}

\usepackage[numbers, sort, compress]{natbib}
\makeatletter
\renewcommand*\env@matrix[1][\arraystretch]{%
  \edef\arraystretch{#1}%
  \hskip -\arraycolsep
  \let\@ifnextchar\new@ifnextchar
  \array{*\c@MaxMatrixCols c}}
\makeatother

\usepackage{framed}

\newcommand{\TODO}[1][0]{%
  \ifx#10
    $\square$
  \else
    $\boxtimes$
  \fi
}
\newcommand\blfootnote[1]{%
  \begingroup
  \renewcommand\thefootnote{}\footnote{#1}%
  \addtocounter{footnote}{-1}%
  \endgroup
}

\usepackage{float}
\usepackage{tikz}
\tikzstyle{inarrow} = [<-,very thick]
\tikzstyle{bcircle} = [circle,draw = blue]
\usetikzlibrary{decorations.pathmorphing}
\usepackage[linkcolor=blue,colorlinks=true,citecolor=blue,filecolor=black]{hyperref}
\usepackage{geometry}
\usepackage{setspace}

\DeclareUnicodeCharacter{025B}{\ensuremath{\varepsilon}}

\begin{document}

\title{Metric Entropy of Ellipsoids in Banach Spaces: Techniques and Precise Asymptotics}
\date{}
\author{Thomas Allard \\ tallard@ethz.ch  \and Helmut Bölcskei \\ hboelcskei@ethz.ch}

\maketitle

\abstract{
\noindent
We develop new techniques for computing the metric entropy of ellipsoids---with polynomially decaying semi-axes---in Banach spaces. 
Besides leading to a unified and comprehensive framework, these tools deliver numerous novel results as well as substantial improvements and generalizations of classical results. Specifically, we characterize the constant in the leading term in the asymptotic expansion of the metric entropy of $p$-ellipsoids with respect to $q$-norm, for arbitrary $p,q \in [1, \infty]$, to date known only in the case $p=q=2$. 
Moreover, for $p=q=2$, we improve upon classical results 
by specifying the second-order term in the asymptotic expansion. 
In the case $p=q=\infty$, we obtain a complete, as opposed to asymptotic, characterization of metric entropy and explicitly construct optimal coverings.
To the best of our knowledge, this is the first exact characterization of the metric entropy of an infinite-dimensional body. Application of our general results to function classes yields an improvement of the asymptotic expansion of the metric entropy of unit balls in Sobolev spaces and identifies the dependency 
of the metric entropy of unit balls in Besov spaces on the domain of the functions in the class.
Sharp results on the metric entropy of function classes find application, e.g., in machine learning, where they allow to specify the minimum required size of deep neural networks for function approximation, nonparametric regression, and classification over these function classes.
}

\blfootnote{\label{ack}The authors gratefully acknowledge support by the Lagrange Mathematics and Computing Research Center, Paris, France.}

\section{Introduction}\label{introduction}

We introduce a suite of new techniques for the characterization of the metric entropy of infinite-dimensional ellipsoids with polynomially decaying semi-axes.
Our approach builds on that underlying the results by the authors of the present paper in \cite{firstpaper} for ellipsoids of exponentially decaying semi-axes.
The case of polynomial decay is, however, significantly more challenging.
Specifically, the pillars of the techniques in \cite{firstpaper}, namely thresholding of semi-axes and the use of volume arguments, need to be replaced by the concept of block decomposition of semi-axes and density arguments, respectively.

The metric entropy of infinite-dimensional ellipsoids with polynomially decaying semi-axes has been studied quite widely in the literature, 
see e.g. \cite{marcus1974varepsilon,lorentzMetricEntropyApproximation1966,pinkus1985,donohoCountingBitsKolmogorov2000,grafSharpAsymptoticsMetric2004,luschgySharpAsymptoticsKolmogorov2004} and work on the entropy numbers of diagonal operators between sequence spaces \cite{carlInequalitiesEigenvaluesEntropy1980,carl1981entropy,carlEntropyCompactnessApproximation1990,edmunds1998entropy,kuhnENTROPYNUMBERSDIAGONAL2001,kuhnEntropyNumbersEmbeddings2005,kuhn2005entropy,kuhnEntropyNumbersSequence2008,KOSSACZKA2020319}. 
The new tools developed here yield substantial improvements and generalizations of the best known results.
Specifically, we characterize the constant in the leading term in the asymptotic expansion of the metric entropy of $p$-ellipsoids with respect to $q$-norm, for arbitrary $p,q \in [1, \infty]$, to date known only in the Hilbertian case $p=q=2$  \cite{donohoCountingBitsKolmogorov2000,luschgySharpAsymptoticsKolmogorov2004}. 
Moreover, for $p=q=2$, we also improve upon the classical results \cite{donohoCountingBitsKolmogorov2000,luschgySharpAsymptoticsKolmogorov2004} by specifying the second order term in the asymptotic expansion of metric entropy.
This result was applied in \cite{thirdpaper} by the authors of the present paper to improve upon the best known characterization
of the asymptotic expansion of the metric entropy of unit balls in Sobolev spaces.

In the case $p=q=\infty$, we even get exact expressions for metric entropy and explicitly construct optimal coverings.
To the best of our knowledge, this is the first exact, as opposed to asymptotic only, characterization of the metric entropy of an infinite-dimensional body.
Finally, for $p=q=2$ and $p=q=\infty$ the lower and the upper bounds, respectively, in our result for general values of $p$ and $q$ are shown to be tight.

We will exclusively be concerned with ellipsoids in real infinite-dimensional sequence spaces. 
Extensions to the complex case do not pose any technical difficulties and mostly entail inserting a multiplicative factor of $2$ only.
Concretely, we consider the following subsets of sequence spaces.
(We refer the reader to 
%Appendix~\ref{NotationandTerminology} 
the end of this section for the notation conventions employed in the paper.)

\begin{definition}[Infinite-dimensional ellipsoids]\label{def: Infinite dimensional ellipsoid}
Let $p \in [1, \infty]$.
To a given sequence $\{\mu_n\}_{n \in \mathbb{N}^*}$ of positive real numbers, we associate the {ellipsoid norm} $\|\cdot\|_{p, \mu}$ on $\ell^p(\mathbb{N^*} \semcol \mathbb{R})$, defined as 
\begin{equation*}
\|\cdot\|_{p, \mu} \colon x \in \ell^p(\mathbb{N^*} \semcol \mathbb{R}) \longmapsto 
\begin{cases}
\left(\sum_{n=1}^\infty  \left\lvert x_n/\mu_n \right\rvert^p\right)^{1/p}, &\text{if } 1 \leq p < \infty, \\
\sup_{n \in \mathbb{N}^*} \, \left\lvert x_n/\mu_n \right\rvert,  &\text{if } p= \infty.
\end{cases}
\end{equation*}
The {infinite-dimensional $p$-ellipsoid} with respect to the norm $\|\cdot\|_{p, \mu}$ is
\begin{equation*}\label{eq: Infinite dimensional ellipsoid definition}
\mathcal{E}_{p}(\{\mu_n\}_{n\in \mathbb{N}^*}) 
\coloneqq 
\left\{ x \in \ell^p(\mathbb{N^*}\semcol\mathbb{R}) \mid \|x\|_{p, \mu} \leq 1 \right\}.
\end{equation*}
The elements of $\{\mu_n\}_{n \in \mathbb{N}^*}$ are referred to as the {semi-axes} of the ellipsoid $\mathcal{E}_{p}(\{\mu_n\}_{n\in \mathbb{N}^*})$.
\end{definition}

\noindent
Throughout the paper, unless explicitly stated otherwise, we assume that the semi-axes $\{\mu_n\}_{n\in \mathbb{N}^*}$ are arranged in non-increasing order and tend to zero, i.e., 
\begin{equation*}
\mu_1 \geq \dots \geq \mu_n \geq \dots \to 0.
\end{equation*}
As mentioned above, the case of exponentially decaying semi-axes has been studied in \cite{firstpaper}.
In the present paper, we consider semi-axes that exhibit polynomial decay, with
the canonical example we shall often refer to given by
\begin{equation}\label{eq:canonicaldecayregvar}
 \mu_n = \frac{c}{n^b}, \quad \text{for } n \in \mathbb{N}^* \text{ and fixed } b, c>0. \tag{CEx}
 \end{equation} 
The general semi-axis decay behavior we are concerned with is formalized through the notion of ``regular variation'' 
(see, e.g., \cite{binghamRegularVariation1987}).

\begin{definition}[Regularly varying sequence]\label{def:regular-variation}
A sequence $\{u_n\}_{n\in \mathbb{N}^*}$ of positive real numbers is regularly varying with index $\beta \in \mathbb{R}$ if
\begin{equation*}
\lim_{n \to \infty} \frac{u_{\lfloor \alpha n \rfloor}}{u_n} = \alpha^\beta,
\quad \text{for all }\,  \alpha>0.
\end{equation*}
A slowly varying sequence is a regularly varying sequence with index $\beta=0$.
\end{definition}

\noindent
The canonical example (\ref{eq:canonicaldecayregvar}) corresponds to a regularly varying sequence with index $-b$.
The concept of regularly varying sequences is more general though as it incorporates polynomial decay tempered by e.g. logarithmic or small oscillatory terms.
We shall also need the notion of regularly varying functions as defined in \cite{binghamRegularVariation1987}. Concretely,
$\phi \colon \mathbb{R}^*_+ \to \mathbb{R}^*_+$ is regularly varying with index $\beta \in \mathbb{R}$ if  
$\phi(\alpha t)/\phi(t) \to_{t\to\infty} \alpha^{\beta}$, for all $\alpha>0$. 
When $\beta=0$, we say that the function is slowly varying.

We now recall the definition of covering numbers and metric entropy.

\begin{definition}[Covering numbers and metric entropy]
Let $\varepsilon >0$, let $(\mathcal{X}, d)$ be a metric space, and let $\mathcal{K}\subseteq \mathcal{X}$ be a compact set. 
An \emph{$\varepsilon$-covering} of $\mathcal{K}$, with respect to the metric $d$, is a set $\{x_1,...\, ,x_N\} \subseteq \mathcal{X}$ such that for each $x \in \mathcal{K}$, there exists an $i\in \{1, \dots, N\}$ so that $d(x,x_i)\leq \varepsilon$. 
The \emph{$\varepsilon$-covering number} $N(\varepsilon \semcol \mathcal{K}, d)$ is the smallest cardinality of an $\varepsilon$-covering of $\mathcal{K}$.
The \emph{metric entropy} of the set $\mathcal{K}$ is given by
\begin{equation*}
H \left(\varepsilon \semcol \mathcal{K}, d\right)
\coloneqq \log N\left(\varepsilon \semcol \mathcal{K}, d\right).
\end{equation*}
\end{definition}

\noindent
The framework of regularly-varying functions is well-suited for the characterization of the metric entropy of function classes defined through regularity constraints. 
The results on the metric entropy $H(\varepsilon)$ of such smoothness classes known in the literature, see, e.g., \cite[Theorem 3.1.2]{edmunds1992entropy}, \cite[Theorem 2.2]{kuhn2005entropy}, or the books \cite{edmundsFunctionSpacesEntropy1996,lorentzConstructiveApproximationAdvanced1996,triebel2006theofctspace}, take the form
\begin{equation}\label{eq:akjfvbeajkjslzzl}
c \, \varepsilon^{-1/{b}}
\leq H(\varepsilon) 
\leq C\,  \varepsilon^{-1/{b}},
\quad \text{for } \varepsilon \in (0, \varepsilon^*),
\end{equation}
with $\varepsilon$ the covering ball radius, $\varepsilon^*>0$, 
positive constants $c$ and $C$,
and $b = {s}/{d}$, where $s$ is the degree of smoothness and $d$ the dimension of the domain of the functions in the class.
These results are typically obtained by expanding the functions under consideration into a suitable orthonormal basis (e.g., wavelets) and then characterizing the metric entropy of the resulting set of expansion coefficient sequences.

The constants $c$ and $C$ appearing in (\ref{eq:akjfvbeajkjslzzl}) are generally not characterized in the literature.
The machinery developed in this paper 
yields sharp characterizations of these constants 
for a very wide
range of smoothness conditions and norms, and in some cases precise small-$\varepsilon$ asymptotics of $H(\varepsilon)$ or even exact expressions for $H(\varepsilon)$. 
Moreover, we show in Section~\ref{sec:Besovapplic} how to obtain significant improvements of the best known results on the metric entropy of general Besov spaces. Such sharp characterizations of the metric entropy of function classes allow, for example, to infer the minimum required size of deep neural networks achieving optimal approximation \cite{elbrachterDeepNeuralNetwork2021} or optimal regression \cite{focm2024} over the function class. 

The main technical idea underlying our results is a decomposition of the semi-axes of the infinite-dimensional ellipsoid under consideration
into a collection of finite blocks and a residual infinite block; we refer to this procedure as block decomposition. 
The thresholding technique introduced in 
\cite{firstpaper} can be seen as a simple special case of block decomposition, with one finite block only, whose
size was referred to as ``effective dimension'' in \cite{firstpaper}. 
We then derive sharp bounds on the metric entropy of
the constituent finite-dimensional ellipsoids corresponding to the individual finite blocks and judiciously glue
these bounds together.
The contribution of the residual infinite block will turn out negligible relative to the covering ball radius.
We will find that the volume-based techniques for bounding the metric entropy of the constituent finite-dimensional ellipsoids developed in \cite{firstpaper}
are too weak in the case of polynomially decaying semi-axes. 
Resorting to density arguments as introduced by Rogers \cite{rogersNoteCoverings1957, rogersCoveringSphereSpheres1963, rogersbook1964} resolves this problem.

A natural object that appears when performing block decomposition is the so-called mixed ellipsoid.
Mixed ellipsoids also arise in the study of Besov and Triebel-Lizorkin spaces \cite[Chapter 1.5]{triebel2006theofctspace}, Lorentz-Marcinkiewicz spaces \cite{cobos1987entropy,merucci2006applications}, modulation spaces \cite[Chapters 11-12]{groechenigFoundationsTimeFrequencyAnalysis2001}, and Wiener amalgam spaces \cite{feichtinger2020wiener}.
In Section~\ref{sec:mixedellipsoidsseparateresults}, we show how our approach leads to improvements of the best known results on the metric entropy of mixed ellipsoids reported in \cite{leopoldEmbeddingsEntropyNumbers2000}, \cite[Chapter 3]{vybiral2006function}, and \cite{kuhnEntropyNumbersEmbeddings2005,kuhnEntropyNumbersSequence2008}.

\paragraph{Notation.}
For the finite set $A$, we write $\# A$ for its cardinality.
Given the collection of sets $\{A_i\}_{i\in I}$, we denote its Cartesian product by $\prod_{i\in I} A_i$.
$\mathbb{N}$ and $\mathbb{N}^*$ stand for the set of natural numbers including and, respectively, excluding zero, 
$\mathbb{R}$ is the set of real numbers, and $\mathbb{R}^*_+$ denotes the positive real numbers.
We adopt the convention $\sum_{n=n_1}^{n_2}=0$ when $n_2 < n_1$.
$(x)_+$ assigns to $x\in\mathbb{R}$ the value $0$ if $x<0$ and the value $x$ if $x\geq 0$.
%\textcolor{blue}{
$\lfloor x \rfloor$ stands for the integer part of $x\in\mathbb{R}_+^*$, that is, the largest integer smaller than $x$. $\log (\cdot)$ stands for the logarithm to base $2$ and $\ln(\cdot)$ is the natural logarithm.
For $k\in \mathbb{N}^*$, we define the $k$-fold iterated logarithm according to
\begin{equation*}
\log^{(k)}(\cdot) 
= \underbrace{\log \circ \cdots \circ \log}_{k \text{ times}} \, (\cdot).
\end{equation*}
Given a finite collection of positive real numbers $\mu_1, \dots, \mu_d$, we denote its geometric mean by $\bar \mu_d$.
When comparing the asymptotic behavior of the functions $f$ and $g$ as $x \to \ell$, with $\ell \in \mathbb{R}\cup\{-\infty, \infty\}$, we use the notation 
\begin{equation*}
f = o_{x \to \ell}(g) \iff \lim_{x \to \ell} \frac{f(x)}{g(x)} =0
\quad \text{and} \quad 
f = \mathcal{O}_{x \to \ell}(g) \iff \lim_{x \to \ell} \left\lvert \frac{f(x)}{g(x)}\right\rvert \leq C,
\end{equation*}
for some constant $C>0$.
Asymptotic equivalence is indicated according to
\begin{equation*}
f \sim_{x \to \ell} g \iff \lim_{x \to \ell} \frac{f(x)}{g(x)} =1.
\end{equation*}
We occasionally use the notation $\simeq$ to mean ``of the order of'' in a general sense.
It will be convenient to introduce the notation 
\begin{equation*}
V_{p,q, d} \coloneqq \left(\frac{\text{\normalfont{vol}}\left(\mathcal{B}_p \right)}{\text{\normalfont{vol}}(\mathcal{B}_q)}\right)^{1/d},
\end{equation*}
where $\mathcal{B}_p$ is the $p$-unit ball and $\mathcal{B}_q$ the $q$-unit ball, both in $\mathbb{R}^{d}$. % dimensions.
For the Euclidean unit ball $\mathcal{B}_2$, we simply write $\mathcal{B}$.
We shall frequently make use of the asymptotic result \cite[Lemma 35]{firstpaper} 
\begin{equation}\label{eq: ratio volumes ksdjnazc}
V_{p,q, d}
=   \Gamma_{p,q} \,  d^{(1/q-1/p)} \left(1+O_{d\to\infty}\left(\frac{1}{d}\right)\right),
\end{equation}
where
\begin{equation}\label{eq:defgammapq}
\Gamma_{p,q}
\coloneqq \frac{\Gamma(1/p+1) \, p^{1/p}}{\Gamma(1/q+1) \, q^{1/q} \, e^{(1/q-1/p)} },
\quad \text{with } \Gamma \text{ the Euler function}.
\end{equation}

\section{Preliminaries and Technical Tools}

We first formally introduce ellipsoids in finite dimensions. The following definition is \cite[Definition 2]{firstpaper} particularized
to real vector spaces.

\begin{definition}[Finite-dimensional ellipsoids]\label{def: Finite dimensional ellipsoid norm}
Let $d \in \mathbb{N}^*$ and $p \in [1, \infty]$.
To a given set $\{\mu_1, \dots, \mu_d\}$ of positive real numbers, we associate the {$p$-ellipsoid norm} $\|\cdot\|_{p, \mu}$ on $\mathbb{R}^d$, defined as 
\begin{equation*}
\|\cdot\|_{p, \mu} \colon x \in \mathbb{R}^d \mapsto 
\begin{cases}
\left(\sum_{n=1}^d  \left\lvert x_n/\mu_n\right\rvert^p\right)^{1/p}, &\text{if } 1 \leq p < \infty, \\
\max_{1 \leq n \leq d} \, \left\lvert x_n/\mu_n\right\rvert ,  &\text{if } p = \infty.
\end{cases}
\end{equation*}
The {finite-dimensional $p$-ellipsoid} is the unit ball in $\mathbb{R}^d$ with respect to the norm $\|\cdot\|_{p, \mu}$ and is denoted by
\begin{equation*}
\mathcal{E}^d_p \left(\{\mu_n\}_{n=1}^d \right) 
\coloneqq 
\left\{ x \in \mathbb{R}^d \mid \|x\|_{p, \mu} \leq 1 \right\}.
\end{equation*}
The numbers $\{\mu_1, \dots, \mu_d\}$ are referred to as the {semi-axes} of the ellipsoid $\mathcal{E}^d_p (\{\mu_n\}_{n=1}^d) $.
We simply write $\mathcal{E}^d_p$ when the choice of semi-axes is clear from the context. 
For simplicity of exposition, and without loss of generality, we assume that the semi-axes $\{\mu_1, \dots, \mu_d\}$ are arranged in non-increasing order, i.e., 
\begin{equation}
\mu_1 \geq \dots \geq \mu_d > 0.
\end{equation}
\end{definition}

%\noindent

In \cite{firstpaper} we reduced the infinite-dimensional ellipsoid under consideration to a finite-dimensional ellipsoid by
truncation of the semi-axes after a certain number of components.
While this approach is sufficiently powerful in the case of exponentially 
decaying semi-axes as considered in \cite{firstpaper}, more 
subtlety is required when the semi-axes are only polynomially decaying.
Concretely, we truncate as in \cite{firstpaper}, but then organize the retained components into $k \in \mathbb{N}^{\ast}$ blocks.
Each of these blocks then corresponds to a finite-dimensional constituent ellipsoid, whose covering number will be analyzed individually.

A natural question that now arises is: \textit{How many components should be retained and how should we organize them into blocks?} 
Obviously, the number of components to be retained should, as in \cite{firstpaper}, depend on the covering ball radius $\varepsilon$.
For $p=q$, the picture is quite clear, namely one keeps the components corresponding to semi-axes greater than or equal to $\varepsilon$.
The intuition, already formulated in \cite{firstpaper} and illustrated in Fig. \ref{figdimred} for $p=q=2$, is that the remaining components are covered by $\varepsilon$-balls.

\begin{figure}[hbt!]
\begin{center}
  \begin{tikzpicture}[scale=3]
      % Axis style
      \tikzstyle{axis}=[thick, ->, >=stealth]
  
      % Softer colors
      %\definecolor{axis_color}{RGB}{108,92,231}
      \definecolor{axis_color}{rgb}{0.6, 0, 0.6}
      \definecolor{fill_color}{RGB}{116,185,255}
      \definecolor{line_color}{RGB}{0,69,148}
      \definecolor{ball_contour}{RGB}{48,113,87}
  
      % Define tick length
      \def\ticklength{2pt}

      \draw[help lines, color=gray!30, dashed] (-1.6,-.4) grid[step=.25] (1.6,.4);

      %\draw[line_color, thick, fill=fill_color!50, fill opacity=0.3] (0,0) circle (1);
      \node[ellipse,
      draw, line_color, thick,
      fill=fill_color!50, fill opacity=0.3,
      minimum width = 7.5cm, 
      minimum height = 1.5cm] (0,0) {};

      % Draw axes
      \draw[axis_color, axis] (-1.5,0) -- (1.5,0) node[below] {\tiny $x_1$};
      \draw[axis_color, axis] (0,-.4) -- (0,.4) node[right] {\tiny $x_2$};
      \node at (0,0) [anchor=north west, axis_color] {\tiny $0$};

        % Add ticks and labels
      \draw[axis_color] (1.25,\ticklength) -- (1.25,-\ticklength) node[anchor=south west, xshift=-3pt, yshift=3pt] {\tiny $\mu_1$};
      \draw[axis_color] (-1.25,\ticklength) -- (-1.25,-\ticklength) node[anchor=south east, xshift=3pt, yshift=3pt] {\tiny $-\mu_1$};
      \draw[axis_color] (\ticklength,.25) -- (-\ticklength,.25) node[anchor=south east, xshift=3pt, yshift=-1pt]  {\tiny $\mu_2$};
      \draw[axis_color] (\ticklength,-.25) -- (-\ticklength,-.25) node[anchor=north east, xshift=3pt, yshift=2pt] {\tiny $-\mu_2$};

      % Add balls 
      \draw[ball_contour, thick, fill=green!15, fill opacity=0.3] (-.97,0) node[opacity=1] {\tiny x} circle (.3);
      \draw[ball_contour, thick, fill=green!15, fill opacity=0.3] (-.61,0) node[opacity=1] {\tiny x} circle (.3);
      \draw[ball_contour, thick, fill=green!15, fill opacity=0.3] (-.275,0) node[opacity=1] {\tiny x} circle (.3);
      \draw[ball_contour, thick, fill=green!15, fill opacity=0.3] (0,0) node[opacity=1] {\tiny x} circle (.3);
      \draw[ball_contour, thick, fill=green!15, fill opacity=0.3] (.275,0) node[opacity=1] {\tiny x} circle (.3);
      \draw[ball_contour, thick, fill=green!15, fill opacity=0.3] (.61,0) node[opacity=1] {\tiny x} circle (.3);
      \draw[ball_contour, thick, fill=green!15, fill opacity=0.3] (.97,0) node[opacity=1] {\tiny x} circle (.3);

      % Add radius of top left ball 
      \draw[ball_contour, {Stealth}-{Stealth}] (-.97,0) --  node[anchor=south west, yshift=1pt, xshift=-5pt] {\tiny $\varepsilon$} (-1.21,.18);

      % Add the ellipsoid label
      \node[line_color] at (1.05,.21) (e) {\tiny \textbf{$\mathcal{E}^2_2$}};

  \end{tikzpicture}
  \captionof{figure}{The two-dimensional ellipsoid $\mathcal{E}^2_2$ is automatically covered in the dimension associated with $\mu_2$ by lining up $\varepsilon$-balls on the horizontal axis when $\varepsilon$ is larger than $\mu_2$.}
  \label{figdimred}      
\end{center}
  \end{figure}
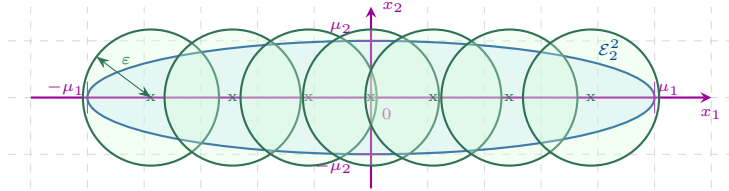

For $p\neq q$, more refined geometric considerations are needed. 
Indeed, as $d$ grows the proportion of the ellipsoid $\mathcal{E}^d_p$ covered by a $\|\cdot\|_q$-ball of fixed radius $\varepsilon$ decreases when $p>q$ and
increases for $p<q$. 
We can, however, compensate for this by scaling the covering ball radius by a factor of the order $ d^{(1/q-1/p)}$.
Fig. \ref{figdimredvol} illustrates the idea in the two-dimensional case.
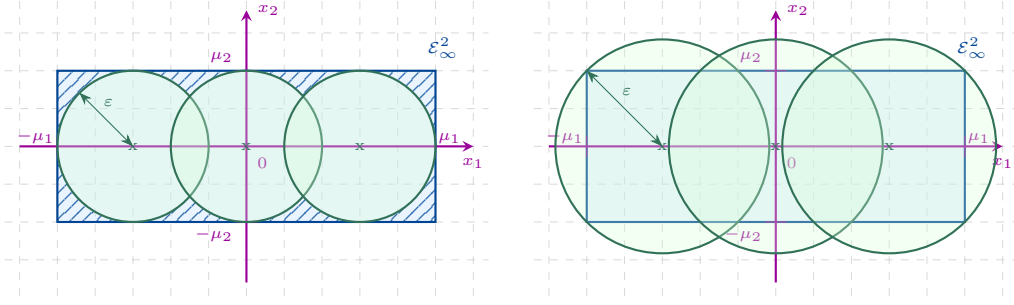
\begin{figure}[ht]
\begin{center}
  \begin{tikzpicture}[scale=2]
  
    % Axis style
    \tikzstyle{axis}=[thick, ->, >=stealth]

    % Define the colors 
    \definecolor{axis_color}{rgb}{0.6, 0, 0.6}
    \definecolor{fill_color}{RGB}{116,185,255}
    \definecolor{line_color}{RGB}{0,69,148}
    \definecolor{ball_contour}{RGB}{48,113,87}

    % Define tick length
    \def\ticklength{2pt}

    % Left Picture
    \begin{scope}[shift={(-1.75,0)}]

      \draw[pattern={Lines[angle=45, line width=.5pt, distance=1.25mm]}, pattern color=line_color] (1.25,.5) -- (1.25,-.5) -- (-1.25,-.5) -- (-1.25,.5) -- cycle;
      
      % Add white balls to remove hatch
      \draw[white, fill=white] (-.75,0) node[opacity=1] {\tiny x} circle (.5);
      \draw[white, fill=white] (0,0) node[opacity=1] {\tiny x} circle (.5);
      \draw[white, fill=white] (.75,0) node[opacity=1] {\tiny x} circle (.5);

      \draw[help lines, color=gray!30, dashed] (-1.6,-.99) grid[step=.25] (1.6,.99);
        
      \draw[line_color, thick, fill=fill_color!50, fill opacity=0.3] (1.25,.5) -- (1.25,-.5) -- (-1.25,-.5) -- (-1.25,.5) -- cycle;

      % Draw axes
      \draw[axis_color, axis] (-1.5,0) -- (1.5,0) node[below] {\tiny $x_1$};
      \draw[axis_color, axis] (0,-.9) -- (0,.9) node[right] {\tiny $x_2$};
      \node at (0,0) [anchor=north west, axis_color] {\tiny $0$};

      % Add ticks and labels
      \draw[axis_color] (1.25,\ticklength) -- (1.25,-\ticklength) node[anchor=south west, xshift=-3pt, yshift=1pt] {\tiny $\mu_1$};
      \draw[axis_color] (-1.25,\ticklength) -- (-1.25,-\ticklength) node[anchor=south east, xshift=3pt, yshift=1pt] {\tiny $-\mu_1$};
      \draw[axis_color] (\ticklength,.5) -- (-\ticklength,.5) node[anchor=south east, xshift=3pt, yshift=-1pt]  {\tiny $\mu_2$};
      \draw[axis_color] (\ticklength,-.5) -- (-\ticklength,-.5) node[anchor=north east, xshift=3pt, yshift=2pt] {\tiny $-\mu_2$};

      % Add balls 
      \draw[ball_contour, thick, fill=green!15, fill opacity=0.3] (-.75,0) node[opacity=1] {\tiny x} circle (.5);
      \draw[ball_contour, thick, fill=green!15, fill opacity=0.3] (0,0) node[opacity=1] {\tiny x} circle (.5);
      \draw[ball_contour, thick, fill=green!15, fill opacity=0.3] (.75,0) node[opacity=1] {\tiny x} circle (.5);

      % Add radius of top left ball 
      \draw[ball_contour, {Stealth}-{Stealth}] (-.75,0) --  node[anchor=south west, yshift=1pt, xshift=-5pt] {\tiny $\varepsilon$} (-1.105,.355);

      % Add the ellipsoid label
      \node[line_color] at (1.3,.65) (e) {\tiny \textbf{$\mathcal{E}^2_\infty$}};

    \end{scope}

    % second picture
    \begin{scope}[shift={(1.75,0)}]

      \draw[help lines, color=gray!30, dashed] (-1.6,-.99) grid[step=.25] (1.6,.99);
        
      \draw[line_color, thick, fill=fill_color!50, fill opacity=0.3] (1.25,.5) -- (1.25,-.5) -- (-1.25,-.5) -- (-1.25,.5) -- cycle;

      % Draw axes
      \draw[axis_color, axis] (-1.5,0) -- (1.5,0) node[below] {\tiny $x_1$};
      \draw[axis_color, axis] (0,-.9) -- (0,.9) node[right] {\tiny $x_2$};
      \node at (0,0) [anchor=north west, axis_color] {\tiny $0$};

      % Add ticks and labels
      \draw[axis_color] (1.25,\ticklength) -- (1.25,-\ticklength) node[anchor=south west, xshift=-3pt, yshift=1pt] {\tiny $\mu_1$};
      \draw[axis_color] (-1.25,\ticklength) -- (-1.25,-\ticklength) node[anchor=south east, xshift=3pt, yshift=1pt] {\tiny $-\mu_1$};
      \draw[axis_color] (\ticklength,.5) -- (-\ticklength,.5) node[anchor=south east, xshift=3pt, yshift=-1pt]  {\tiny $\mu_2$};
      \draw[axis_color] (\ticklength,-.5) -- (-\ticklength,-.5) node[anchor=north east, xshift=3pt, yshift=2pt] {\tiny $-\mu_2$};

      % Add balls 
      \draw[ball_contour, thick, fill=green!15, fill opacity=0.3] (-.75,0) node[opacity=1] {\tiny x} circle (.707);
      \draw[ball_contour, thick, fill=green!15, fill opacity=0.3] (0,0) node[opacity=1] {\tiny x} circle (.707);
      \draw[ball_contour, thick, fill=green!15, fill opacity=0.3] (.75,0) node[opacity=1] {\tiny x} circle (.707);

      % Add radius of top left ball 
      \draw[ball_contour, {Stealth}-{Stealth}] (-.75,0) --  node[anchor=south west, yshift=1pt, xshift=-5pt] {\tiny $\varepsilon$} (-1.25,.5);

      % Add the ellipsoid label
      \node[line_color] at (1.3,.65) (e) {\tiny \textbf{$\mathcal{E}^2_\infty$}};

    \end{scope}
  \end{tikzpicture} 
  \end{center}
  \captionof{figure}{%
    The shaded area of the two-dimensional ellipsoid $\mathcal{E}_\infty^2$ (left picture) is not covered when balls of radius $\varepsilon \simeq \mu_2$ are lined up on the horizontal axis.
    This problem disappears when considering balls of radius $\varepsilon \simeq  d^{(1/q-1/p)} \mu_d = \sqrt{2}\mu_2$ (right picture).
  }
  \label{figdimredvol}
\end{figure}
This observation leads us to the following heuristic. Choosing the effective dimension $d$ of the problem, i.e., the number of semi-axes of the infinite-dimensional ellipsoid to be retained, such that 
\begin{equation}\label{eq:heuristicseffdim1}
\varepsilon \simeq d^{(1/q-1/p)} \mu_d, \tag{Heur1}
\end{equation}
we expect the $\varepsilon$-covering number of the resulting $d$-dimensional ellipsoid to capture the behavior of the $\varepsilon$-covering number of the initial infinite-dimensional ellipsoid to a sufficient degree, concretely 
\begin{equation}\label{eq:heuristicseffdim2}
N\left(\varepsilon\semcol \mathcal{E}_p, \|\cdot\|_q \right)
\simeq N\left(\varepsilon\semcol \mathcal{E}^d_p, \|\cdot\|_q \right).\tag{Heur2}
\end{equation}
Rigorous definitions of the effective dimension and formalizations of the heuristics (\ref{eq:heuristicseffdim1}) and (\ref{eq:heuristicseffdim2}) are provided in Section~\ref{sec:proofofmainthmisinthissec}.
We proceed with the aspect of organizing the components retained into blocks.

\subsection{Block Decomposition}\label{sec:ellipsoidsinfinitedim3}

\begin{definition}[Block decomposition]\label{def:blockdecomp}
Let $p\in[1, \infty]$, let $d, k \in \mathbb{N}^*$, and let $d_1, \dots, d_k \in \mathbb{N}^* $ be such that $d = d_1 + \dots + d_k$.
We define the constituent ellipsoid corresponding to the $j$-th block of semi-axes of the
infinite-dimensional $p$-ellipsoid $\mathcal{E}_p$ as
\begin{equation*}
\mathcal{E}_p^{[j]} 
\coloneqq \left\{(x_{\bar d_{j-1}+1}, \dots, x_{\bar d_j}) \mid x \in \mathcal{E}_p \right\},
\quad j\in \{1, \dots, k+1\},
\end{equation*}
where 
\begin{equation}\label{eq:defdesbarres}
\bar d_j 
\coloneqq 
\begin{cases}
0, \quad &\text{if } j=0,\\
d_1+\dots+d_j, \quad &\text{if } j\in \{1, \dots, k\},\\
\infty , \quad &\text{if } j=k+1.
\end{cases}
\end{equation}
\end{definition}

\noindent
Note that the constituents $\mathcal{E}_p^{[j]}, j\in \{1, \dots, k\}$, are all finite-dimensional while
$\mathcal{E}_p^{[k+1]}$ is infinite-dimensional.
We next observe that
every $x\in \mathcal{E}_p$ can be decomposed into $k+1$ sub-vectors $x_j \in \mathcal{E}_p^{[j]}$, $j\in \{1, \dots, k+1\}$,
with the ellipsoid norm of $x_j$ upper-bounded by a non-negative number $\omega_j$ and $\omega_1^p + \dots + \omega_{k+1}^p = 1$. 
%\textcolor{blue}{

Formally, the ellipsoid $\mathcal E_p$ can be embedded into a union of Cartesian
products of rescaled versions of its constituents $\mathcal E_p^{[j]}$,
$j\in\{1,\ldots,k+1\}$, according to
\begin{equation}
\mathcal E_p \;\subseteq\;
\bigcup_{\omega\in\Omega}\;
\prod_{j=1}^{k+1} \omega_j \mathcal E_p^{[j]},
\label{eq:ellipsasunionofcartprodgeneralsett}
\end{equation}
where $\Omega\subset\mathbb R_+^{k+1}$ is a set of scaling vectors for which 
the inclusion \eqref{eq:ellipsasunionofcartprodgeneralsett} is valid. In particular, one may take
$\Omega=\{(1,1,\ldots,1)\}$. For %later reference, we define the canonical set of normalized weights
\begin{equation}
\Omega = \bigl\{\omega\in\mathbb R_+^{k+1} \,\big|\, \|\omega\|_p = 1 \bigr\}, \label{eq:baromegaset}
\end{equation}
the inclusion in \eqref{eq:ellipsasunionofcartprodgeneralsett} is, in fact, an equality.
This representation is, however, highly redundant. In what follows, we therefore work with finite sets $\Omega$, typically sections of
lattices, see, e.g., (\ref{eq:1326}) in the proof of Theorem~\ref{thm:bound on mixed norm} or (\ref{eq:defomegacasepeqq31EEE}) in the proof of Lemma~\ref{lem:klmjhgfddddfhqsdfg}.
%the intersection of $\bar \Omega$ with a lattice, see, e.g., (\ref{eq:1326}) in the proof of Theorem~\ref{thm:bound on mixed norm} or 
%(\ref{eq:defomegacasepeqq31EEE}) in the proof of Lemma~\ref{lem:klmjhgfddddfhqsdfg}.
The following result formalizes how representations of $\mathcal{E}_{p}$, or more specifically inclusion relations of the form \eqref{eq:ellipsasunionofcartprodgeneralsett}, with $\Omega$ finite, can 
lead to upper bounds on the covering number of $\mathcal{E}_{p}$ in terms of the covering numbers of its constituents.

\begin{lemma}[Decomposition]\label{lem:lemboundtriangleineqlike}
Let $p, q \in [1,\infty]$, $k, d\in \mathbb{N}^*$, let $d_1, \dots, d_k \in \mathbb{N}^* $ be such that $d = d_1 + \dots + d_k$, and let $\rho_1,\dots, \rho_k, \bar \rho \in (0,\infty)$.
 Let $\mathcal{E}_p$ be an infinite-dimensional $p$-ellipsoid with non-increasing semi-axes $\{\mu_n\}_{n \in \mathbb{N}^*}$ and let $\Omega \subset \mathbb{R}_+^{k+1}$ be a finite set such that 
 \begin{equation}\label{eq:firstdecompcartprodelli}
\mathcal{E}_{p} 
\subseteq \bigcup_{\omega \in \Omega} \prod_{j=1}^{k+1} \omega_j \mathcal{E}_p^{[j]},
\end{equation}
where $ \mathcal{E}_p^{[j]} $ is the $j$-th constituent of $\mathcal{E}_{p} $ as per Definition \ref{def:blockdecomp}.
Then, we have 
 \begin{equation}\label{eq:lknjhgdshlkqncvzezzzzsxtt}
N \left(\bar \rho\semcol \mathcal{E}_p, \|\cdot\|_q \right)
\leq \left(\# \Omega\right) \, \prod_{j=1}^k N \left(\rho_j\semcol \mathcal{E}_p^{[j]}, \|\cdot\|_q \right) ,
 \end{equation}
provided that either
 \begin{enumerate}[label=(\roman*)]
 \item 
  $p\leq q \leq \infty $ and $\max_{\omega\in\Omega} \rho_{\omega} \leq \bar \rho $, where, for $\omega\in\Omega$,
 \begin{equation*}
  \rho_{\omega} = 
 \begin{cases}
     \left[(\omega_1\rho_1)^q + \dots + (\omega_k\rho_k)^q + (\omega_{k+1}\mu_{d+1})^q \right]^{1/q},
     \quad &\text{if} \quad 1\leq q < \infty, \\[.25cm]
     \max\left\{\max_{j=1, \dots, k} \, \{\omega_j \rho_j\}, \omega_{k+1} \mu_{d+1} \right\},
     \quad &\text{if} \quad q = \infty\!;
 \end{cases}
 \end{equation*}

or

\item
$\{\mu_n\}_{n \in \mathbb{N}^*}$ is additionally regularly varying of index $-b$, with $b>0$ such that $p/(pb+1)<q<p$, 
and $\max_{\omega\in\Omega} \rho_{\omega} \leq \bar \rho $, where, for $\omega\in\Omega$, 
 \begin{equation*}
 \rho_{\omega} =
   \left[(\omega_1\rho_1)^q + \dots + (\omega_k\rho_k)^q + (\omega_{k+1}\alpha_{d})^q \right]^{1/q},
 \end{equation*}
 with $\{\alpha_d\}_{d\in \mathbb{N}^*}$ a sequence satisfying
 \begin{equation}\label{eq:ijankjvbhjbvjqookknsd}
 \alpha_d = \left(\frac{b}{\frac{1}{q}-\frac{1}{p}}  -1 \right)^{(\frac{1}{p}-\frac{1}{q})}  \, \mu_d \, d^{(\frac{1}{q}-\frac{1}{p})}
\left(1+o_{d\to\infty}(1)\right)\!;
 \end{equation}

or
  \item
  $\{\mu_n\}_{n \in \mathbb{N}^*}$ is additionally regularly varying of index $-b$, with $b>0$ such that $q=p/(pb+1)$, and satisfies 
  $\sum_{n\in\mathbb{N}^*}\mu_n^{1/b}<\infty$,
and $\max_{\omega\in\Omega} \rho_{\omega} \leq \bar \rho $, where, for $\omega\in\Omega$,  
 \begin{equation*}
 \rho_{\omega} =
   \left[(\omega_1\rho_1)^q + \dots + (\omega_k\rho_k)^q + (\omega_{k+1}\alpha_{d})^q \right]^{1/q},
 \end{equation*}
 with $\alpha_d = \left(\sum_{n=d+1}^\infty \mu_n^{1/b}\right)^b$, for all $d\in \mathbb{N}^*$.
 \end{enumerate}
\end{lemma}

\begin{proof}[Proof.]
See Section~\ref{sec:prooflemboundtriangleineqlike}.
\end{proof}

\noindent
%\textcolor{blue}{
In Lemma~\ref{lem:lemboundtriangleineqlike}--and in fact throughout the paper--we understand the ratio $p/(pb+1)$ to equal $1/b$ when $p=\infty$.
%}

As already mentioned, the thresholding technique employed in \cite{firstpaper} is a simple special case of the block decomposition method developed here.
For instance, for $p\leq q < \infty$ and given $\rho>0$, \cite[Lemma 42]{firstpaper} states that there exist a $d^* \in \mathbb{N}^*$ and a constant $K \geq 1$ such that, for all $d \geq d^*$, it holds that
\begin{equation*}
  N \left(\bar \rho\semcol \mathcal{E}_p, \|\cdot\|_q \right)
\leq N \left(\rho\semcol \mathcal{E}_p^{d}, \|\cdot\|_q \right),
\quad \text{where} \quad 
\bar \rho = (\rho^q + K\mu_{d+1}^q)^{1/q}.
\end{equation*}
Case (i) in Lemma~\ref{lem:lemboundtriangleineqlike} above, with $k=1$ and $\Omega= \{(1, 1)\}$, is readily seen to strengthen this result to
\begin{equation*}
  N \left(\bar \rho\semcol \mathcal{E}_p, \|\cdot\|_q \right)
\leq N \left(\rho\semcol \mathcal{E}_p^{d}, \|\cdot\|_q \right), 
\quad \text{where} \quad 
\bar \rho = (\rho^q + \mu_{d+1}^q)^{1/q},
\end{equation*}
for all $d \in \mathbb{N}^*$. The added flexibility---in terms of the choice of $k$ and $\Omega$---afforded by Lemma~\ref{lem:lemboundtriangleineqlike} is instrumental in the case of polynomially decaying semi-axes.

We proceed to develop the tools for bounding the covering numbers of the individual constituent ellipsoids $\mathcal{E}_p^{[j]}$.

\subsection{Volume and Density Arguments}\label{sec:volumedensityarg}

The metric entropy bounds on the (single) finite-dimensional constituent ellipsoid in \cite{firstpaper} are based on volume arguments.
We now show that this approach alone, while sufficiently powerful in the case of exponentially decaying semi-axes, is too weak for
polynomially decaying semi-axes. To this end, we first recall a result on volume estimates for metric entropy.

\begin{lemma}[Volume estimates, {\cite[Lemma 5.7]{wainwrightHighDimensionalStatistics2019}}]\label{lem: volume estimates}
Let $d \in \mathbb{N}^*$ and fix $\varepsilon>0$.
Consider the norms $\|\cdot\|$ and $\|\cdot\|'$ on $\mathbb{R}^d$
and let $\mathcal{B}$ and $\mathcal{B}'$ be their respective unit balls.
Then, the $\varepsilon$-covering number $N (\varepsilon\semcol \mathcal{B}, \|\cdot\|' ) $ satisfies
\begin{equation}\label{eq: volume estimates}
 \frac{\text{\normalfont{vol}}(\mathcal{B})}{\text{\normalfont{vol}}(\mathcal{B}')}
\leq N \left(\varepsilon \semcol \mathcal{B}, \|\cdot\|' \right) \, \varepsilon^d 
\leq 2^d \frac{\text{\normalfont{vol}}\left(\mathcal{B} + \frac{\varepsilon}{2} \mathcal{B}' \right)}{\text{\normalfont{vol}}(\mathcal{B}')}.
\end{equation}
\end{lemma}

In our setting the norms $\|\cdot\|$ and $\|\cdot\|'$ in Lemma~\ref{lem: volume estimates} will typically be identified with $\|\cdot\|_{p,\mu}$ and $\|\cdot\|_q$, respectively, so that the associated unit balls are $\mathcal{B} = \mathcal{E}_p^d$ and $\mathcal{B}'= \mathcal{B}_q$.
With these choices, considering the case $p=q$ and recalling that by (\ref{eq:heuristicseffdim1}) we are interested in the behavior of $N\left(\varepsilon \semcol \mathcal{E}^d_p, \|\cdot\|_p \right)$ for
$\varepsilon \simeq \mu_d$, it follows from Lemma~\ref{lem: volume estimates} that
\begin{equation}\label{eq:lkojihuezdjksfpoezss}
\frac{\bar \mu_d^d}{\mu_d^d}
\leq N\left(\varepsilon \semcol \mathcal{E}^d_p, \|\cdot\|_p \right) 
\leq 4^d \, \frac{\bar \mu_d^d}{\mu_d^d},
\end{equation}
where we used $\mathcal{E}^d_p = A_\mu \mathcal{B}_p$ and the upper bound is valid for $\varepsilon \in (0,2\mu_d]$.
Recall that $\bar{\mu}_d$ is the geometric mean of the semi-axes as introduced in the notation paragraph at the end of Section~\ref{introduction}.
For exponentially decaying semi-axes $\mu_n = c\,e^{-bn}$, we have $\frac{\bar \mu_d^d}{\mu_d^d} = e^{(\frac{bd^2}{2}-\frac{bd}{2})}$, which shows that the $4^d$-gap between the lower and the upper bound in \eqref{eq:lkojihuezdjksfpoezss} is negligible as $d$ grows large. 
However, for polynomially
decaying semi-axes $\mu_n = \frac{c}{n^b}$, it follows that 
\begin{equation*}
\frac{\bar \mu_d^d}{\mu_d^d}
= \frac{d^{bd}}{(d !)^b}
= e^{bd\, (1+o_{d\to\infty}(1))},
\end{equation*}
and hence 
the $4^d$-gap is non-negligible as $4^d$ is of the same order as $\frac{\bar \mu_d^d}{\mu_d^d}$.

The $4^d$-gap in (\ref{eq:lkojihuezdjksfpoezss}) stems from two sources, namely the factor $2^d$ in the upper bound in (\ref{eq: volume estimates}) and the discrepancy between the terms $\text{\normalfont{vol}}(\mathcal{B})$ in the lower and $\text{\normalfont{vol}}(\mathcal{B} + \frac{\varepsilon}{2} \mathcal{B}')$ in the upper bound.
The factor $2^d$ can be handled using density arguments akin to those introduced by Rogers in \cite{rogersNoteCoverings1957,rogersbook1964}.
The idea of using density arguments to upper-bound the metric entropy of ellipsoids has been used extensively in the literature, see e.g. \cite{donohoCountingBitsKolmogorov2000,donohoDataCompressionHarmonic1998,luschgySharpAsymptoticsKolmogorov2004}, and can be traced back to \cite{prosserEntropyCapacityCertain1966}.
The discrepancy between $\text{\normalfont{vol}}(\mathcal{B})$ and $\text{\normalfont{vol}}(\mathcal{B} + \frac{\varepsilon}{2} \mathcal{B}')$, resolved above by using $\mathcal{E}^d_p = A_\mu \mathcal{B}_p$ and restricting $\varepsilon$ to the interval $(0,2\mu_d]$,
can be dealt with 
through arguments from convex geometry (see e.g. \cite{hugLecturesConvexGeometry2020}).
Combining these two techniques in a suitable manner, along with additional ideas, leads to the following central component of our main results on infinite-dimensional ellipsoids reported in Section~\ref{sec:proofofmainthmisinthissec}.

\begin{theorem}\label{thm:profinitethm}
Let $p,q\in [1, \infty]$, $\eta>0$, $d\in \mathbb{N}^*$, $\mu_1, \dots,\mu_d \in \mathbb{R}^*_+$, and let $\mathcal{E}_p^d$ be the $p$-ellipsoid with semi-axes $\{\mu_1, \dots,\mu_d \}$. Then, we have
\begin{equation}\label{eq:lkjbkjbkjbckjbkebzhzzzqdfguyuy}
N\left(\varepsilon\semcol \mathcal{E}^d_p, \|\cdot\|_q \right)^{1/d} \, \varepsilon
\geq  V_{p, q, d} \, \bar \mu_d, 
\quad \text{for all } \varepsilon > 0.
\end{equation}
Moreover, there exists a sequence $\{\kappa(d)\}_{d \in \mathbb{N}^*}$ satisfying 
\begin{equation*}
 \kappa(d)= 1+ O_{d\to \infty}\left(\frac{\log d}{d}\right),
\end{equation*} 
such that
\begin{enumerate}[label={[\textbf{FD\arabic*}]},
leftmargin=1.5cm
]

\item \label{item:ifinitedim}
\begin{equation}\label{eq:lkjbkjbkjbckjbkebzhzzzqdfguyuy1}
\left(N\left(\varepsilon\semcol \mathcal{E}^d_p, \|\cdot\|_q \right)-1\right)^{1/d} \varepsilon 
\leq \kappa(d) \left(1+ {\eta}\right) \, V_{p, q, d} \,  \bar \mu_d ,
\end{equation}
for all
\begin{equation}\label{eq:extraassumptionub1}
\varepsilon \in \left(0, \eta \, d^{-(1/p-1/q)_+} \mu_d\right].
\end{equation}

\item \label{item:iifinitedim}
\begin{equation}\label{eq:lkjbkjbkjbckjbkebzhzzzqdfguyuy2}
\left(N\left(\varepsilon\semcol \mathcal{E}^d_p, \|\cdot\|_q \right)-1\right)^{1/d} \varepsilon 
\leq \kappa(d) \left(1+ {\eta}\right) \, d^{(1/q-1/p)}  \,  \bar \mu_d ,
\end{equation}
whenever
\begin{equation}\label{eq:extraassumptionub2}
p\geq q 
\quad \text{and} \quad 
\varepsilon \in \left(0, \eta \, d^{(1/q-1/p)} \, \mu_d\right].
\end{equation}

\end{enumerate}
\end{theorem}

\begin{proof}[Proof.]
The proof is presented in Section~\ref{sec:prooftheoremprofinitethm}.
\end{proof}

Comparing (\ref{eq:lkjbkjbkjbckjbkebzhzzzqdfguyuy}) and (\ref{eq:lkjbkjbkjbckjbkebzhzzzqdfguyuy1}), we see that the bounds in Theorem~\ref{thm:profinitethm} are essentially sharp as $d\to\infty$, up to the multiplicative constant $1+\eta$.
The restrictions on the range of $\varepsilon$ in (\ref{eq:extraassumptionub1}) and (\ref{eq:extraassumptionub2}) are not fully satisfactory. 
Specifically, as we are interested in the regime $\varepsilon \simeq d^{(1/q-1/p)} \mu_d$, they imply that $\eta $ cannot be taken arbitrarily small, which would be required to make the $(1+\eta)$-factor discrepancy between the lower bound (\ref{eq:lkjbkjbkjbckjbkebzhzzzqdfguyuy}) and the upper bound (\ref{eq:lkjbkjbkjbckjbkebzhzzzqdfguyuy1}) 
vanish.
We do not know whether it is possible to get around this tradeoff on the choice of $\eta$.
In fact, the discussion after Theorem~\ref{thm: scaling metric entropy infinite ellipsoids regular} below strongly suggests that the answer is negative
for general values of $p$ and $q$. For $p=q=2$, however, one can apply stronger density arguments and circumvent the tradeoff.
This is the focus of the next section.

\subsection{Mixed Ellipsoids}\label{sec:mixedellipsoidsseparateresults}

We show how, for $p=q=2$, Theorem~\ref{thm:profinitethm} can be strengthened by 
resorting to sharp density results on the covering of Euclidean balls by Euclidean balls due to Rogers \cite{rogersCoveringSphereSpheres1963}. 
A similar idea was used by Donoho in \cite{donohoCountingBitsKolmogorov2000,donohoDataCompressionHarmonic1998}, and later exploited in \cite{luschgySharpAsymptoticsKolmogorov2004},
in the context of ellipsoids resulting from the analysis of unit balls in Sobolev spaces.
We will need to work with so-called mixed ellipsoids defined as follows.

\begin{definition}[Mixed ellipsoid norm]\label{def: Finite dimensional mixed ellipsoid norm}
Let $p_1,p_2 \in[1, \infty]$ and $\{d_n\}_{n\in\mathbb{N}^*} \subset \mathbb{N}^* $.
To a given sequence $\{\mu_n\}_{n\in\mathbb{N}^*}$ of positive real numbers, we associate the \emph{mixed ellipsoid norm} $\|\cdot\|_{p_1, p_2, \mu}$, defined as 
\begin{equation}
\|\cdot\|_{p_1, p_2, \mu} 
\colon (x_1, x_2, \dots) \in \mathbb{R}^{d_1} \times\mathbb{R}^{d_2} \times \dots \mapsto 
\left\| \left\{\left \|x_j \right\|_{p_2}\right\}_{j=1}^\infty\right\|_{p_1, \mu},
\end{equation}
where $\|\cdot\|_{p_2}$ is the $\ell^{p_2}$-norm and $\|\cdot\|_{p_1, \mu}$ is the ellipsoid norm as per Definition \ref{def: Infinite dimensional ellipsoid}.
\end{definition}

\noindent
For $p_1$ and $p_2$ finite, the mixed ellipsoid norm of the vector $x=(x_{j,i})_{j \in \mathbb{N}^*, 1 \leq i \leq d_j}$ is given by
\begin{equation}
\|x\|_{p_1, p_2, \mu} 
= \left( \sum_{j=1}^\infty \frac{1}{\mu_j^{p_1}} \left(\sum_{i=1}^{d_j} x_{j,i}^{p_2} \right)^{p_1/p_2} \right)^{1/p_1}.
\end{equation}

\begin{definition}[Mixed ellipsoid]\label{def: Finite dimensional mixed ellipsoid}
Let $p_1,p_2 \in[1, \infty]$ and $\{d_n\}_{n\in\mathbb{N}^*} \subset \mathbb{N}^* $.
To a given sequence $\{\mu_n\}_{n\in\mathbb{N}^*}$ of positive real numbers, we associate the \emph{mixed ellipsoid}
\begin{equation}
\mathcal{E}_{p_1, p_2} 
\coloneqq 
\left\{ x \mid \|x\|_{p_1, p_2, \mu} \leq 1 \right\}.
\end{equation}
The elements of $\{\mu_n\}_{n\in\mathbb{N}^*}$ are referred to as the \emph{semi-axes} of the mixed ellipsoid and $\{d_n\}_{n\in\mathbb{N}^*}$
are its dimensions.
\end{definition}

\noindent
Using arguments akin to those leading to (\ref{eq:ellipsasunionofcartprodgeneralsett})--(\ref{eq:baromegaset}),
we can represent a mixed ellipsoid as a union of Cartesian products of balls, concretely 
\begin{equation}\label{eq:ksjhdacozrgheht}
\mathcal{E}_{p_1, p_2} 
= \bigcup_{\delta \in \mathcal{E}_{p_1}} \prod_{j=1}^\infty \mathcal{B}_{p_2}(0 \semcol \delta_j).
\end{equation}
A reduction of the continuum $\mathcal{E}_{p_1}$ in \eqref{eq:ksjhdacozrgheht} to a finite set,
effected by first intersecting $\mathcal{E}_{p_1}$ with a regular infinite lattice and then suitably truncating the result, will allow us, in the proof of Theorem~\ref{thm:bound on mixed norm},
to apply Lemma~\ref{lem:lemboundtriangleineqlike} and bound the metric entropy of $\mathcal{E}_{p_1, p_2} $ in terms of 
the metric entropies of the balls $\mathcal{B}_{p_2}(0\semcol\delta_j)$.
The latter can be upper-bounded, for $p_2=2$, through the following result from \cite{rogersCoveringSphereSpheres1963}.

\begin{lemma}\label{lem:rogernb2}
There exists $K>0$ such that
\begin{equation*}
N\left(\varepsilon \semcol \mathcal{B}(0\semcol\mu), \|\cdot\|_2 \right)^{1/d} \, \varepsilon
\leq \left(K d^{5/2}\right)^{1/d} \mu,
\end{equation*}
for all $d\geq 9$ and all $\mu >0$, with $ \varepsilon \in (0,\mu)$.
\end{lemma}

\begin{proof}[Proof.]
This is simply a reformulation of \cite[Theorem 3]{rogersCoveringSphereSpheres1963} in terms of covering numbers.
\end{proof}

\noindent
We are now ready to present the result, which will lead to the announced strengthening of Theorem~\ref{thm:profinitethm}. 

\begin{theorem}\label{thm:bound on mixed norm}
  Let $\mathcal{E}_{2, 2}$ be a mixed ellipsoid with non-increasing semi-axes $\{\mu_n\}_{n\in\mathbb{N}^*}$ and dimensions 
  $\{d_n\}_{n\in\mathbb{N}^*}$,  
  let $\varepsilon \in (0, \mu_1)$, 
  and let $\gamma\geq 1$ be allowed to depend on $\varepsilon$.
  Let $k\in\mathbb{N}^*$ be such that $\mu_{k+1} \leq \varepsilon < \mu_k$.
  Then, under the assumption $d_1, \dots, d_k\geq 9$, we have
  \begin{equation}\label{eq:pojohoaooooooab1}
  N \left(\varepsilon_\gamma \semcol {\mathcal{E}}_{2, 2}, \|\cdot\|_2 \right)^{1/\bar d_k} \, \varepsilon_\gamma
  \leq \kappa_k \prod_{j=1}^k \mu_j^{d_j/\bar d_k},
  \end{equation}
  where we recall that $\bar d_k$ is defined in (\ref{eq:defdesbarres}), $\varepsilon_\gamma \coloneqq \varepsilon \left(1+ \bar d_k^{-\gamma} \sqrt{k+1}\right)$, and $\{\kappa_k\}_{k\in\mathbb{N}^*}$ is a sequence satisfying
  \begin{equation}\label{eq:jsqhkqoioiloio}
%\textcolor{blue}{  
\log(\kappa_k)
      = O_{k \to \infty} \left(\frac{\gamma \, k \, \log (\jrdn)}{\jrdn}\right).%}
  \end{equation}
\end{theorem}

\begin{proof}
    See Section~\ref{sec:proofofthmmixednorm}.
\end{proof}

\noindent
The assumption $\varepsilon \in (0, \mu_1)$ comes without loss of generality as the ellipsoid $\mathcal{E}_{2, 2}$ is included in a ball of radius $\mu_1$, so that the corresponding covering number is equal to $1$ whenever $\varepsilon \geq \mu_1$. Theorem~\ref{thm:bound on mixed norm} leaves little room for improvement. This can be seen by applying \eqref{eq:lkjbkjbkjbckjbkebzhzzzqdfguyuy} to conclude that
\begin{equation}\label{eq:pppoooaaazeeerffghccc}
 N \left(\varepsilon_\gamma \semcol {\mathcal{E}}_{2, 2}, \|\cdot\|_2 \right)^{1/\bar d_k} \, \varepsilon_\gamma
\geq \prod_{j=1}^k \mu_j^{d_j/\bar d_k}.
\end{equation}
The upper and lower bounds (\ref{eq:pojohoaooooooab1}) and (\ref{eq:pppoooaaazeeerffghccc}) agree up to a multiplicative factor $\kappa_k$, which by (\ref{eq:jsqhkqoioiloio}) is close to $1$ for large $k$ if $\gamma \, k \, \log (\jrdn)/\jrdn \to_{k\to\infty}0$.

We conclude by showing how Theorem~\ref{thm:bound on mixed norm} improves upon Theorem~\ref{thm:profinitethm} in the case $p=q=2$.
First, we observe that ${\mathcal{E}}_{2, 2}$
can be thought of as 
a finite, say $d$-dimensional, ellipsoid if we impose that the semi-axes above a certain index are very close to zero.
Therefore, upon application of case \ref{item:iifinitedim} in Theorem~\ref{thm:profinitethm}, with $p=q=2$ and $\eta =1$, we obtain 
\begin{equation}\label{eq:pojohoaooooooab2}
\left[N\left(\varepsilon\semcol {\mathcal{E}}_{2, 2}, \|\cdot\|_2 \right)-1\right]^{1/d} \varepsilon 
\leq \kappa(d) \left(1+ {\eta}\right)  \prod_{i=1}^k \mu_i^{d_i/d}
= 2\, \kappa(d) \prod_{i=1}^k \mu_i^{d_i/d}
\end{equation}
from (\ref{eq:lkjbkjbkjbckjbkebzhzzzqdfguyuy2}).
Therefore, Theorem~\ref{thm:bound on mixed norm} improves on Theorem~\ref{thm:profinitethm} by getting rid of the multiplicative factor $2$ in the upper bound (\ref{eq:pojohoaooooooab2}) compared to (\ref{eq:pojohoaooooooab1}).

\section{Ellipsoids in Infinite Dimensions}\label{sec:infdimell}

In this section, we state our main results on the metric entropy of infinite-dimensional ellipsoids.
The general case is treated in Section~\ref{sec:mainresultsgeneral}, and in Sections \ref{sec:hilbertiancase} and \ref{sec:optimalityres1} we show that the results can be sharpened for $p=q=2$ and $p=q=\infty$, respectively.

\subsection{The General Case}\label{sec:mainresultsgeneral}

\begin{theorem}\label{thm: scaling metric entropy infinite ellipsoids regular}
Let $p, q \in [1, \infty], b>0$.
Let $\{\mu_n\}_{n\in \mathbb{N}^*}$ be a non-increasing sequence that is regularly varying with index $-b$, and 
let $\phi \colon \mathbb{R}^*_+ \to \mathbb{R}^*_+$ be the step function given by 
$\phi(t) = \mu_{\lfloor t \rfloor + 1}$, for $t > 0$. 
Let $\mathcal{E}_p$ be the $p$-ellipsoid with semi-axes $\{\mu_n\}_{n\in \mathbb{N}^*}$.
\begin{enumerate}[label=(\roman*)]
\item 
If (a) $q < p/(pb+1)$ or (b) $q = p/(pb+1)$ with $ \lim_{n\to\infty}n \mu_n^{1/b}>0 $,
then the ellipsoid $\mathcal{E}_p$ is not compact in $\ell^q$, i.e., there exists $\varepsilon^*>0$ such that 
\begin{equation*}
H \left(\varepsilon\semcol \mathcal{E}_p, \|\cdot\|_q \right) 
= \infty,
\quad \text{for all } \varepsilon \in (0, \varepsilon^*);
\end{equation*}

\item
If $q = p/(pb+1)$ and $\sum_{n\in\mathbb{N}^*}\mu_n^{1/b}<\infty$, then 
\begin{equation*}
    \Gamma_{p,q} 
     \leq \lim_{\varepsilon \to 0} \varepsilon \, \frac{H \left(\varepsilon\semcol \mathcal{E}_p, \|\cdot\|_q \right)^{\left(\frac{1}{p}-\frac{1}{q}\right)}}{\phi \left(H \left(\varepsilon\semcol \mathcal{E}_p, \|\cdot\|_q \right) \right)}
    \quad \text{and} \quad 
    \lim_{\varepsilon \to 0} \frac{\varepsilon}{{\psi} \left(H \left(\varepsilon\semcol \mathcal{E}_p, \|\cdot\|_q \right) \right)}
    \leq 1,
\end{equation*}
with $\Gamma_{p,q}$ as in (\ref{eq:defgammapq}) and $\psi \colon \mathbb{R}_+ \to \mathbb{R}_+^*$ given by 
\begin{equation*}
    \psi \colon x \longmapsto 
    \left(\sum_{n=\lfloor x \rfloor+1}^\infty\mu_n^{1/b}\right)^b;
\end{equation*}

\item
If $p/(pb+1)<q \leq p$, it holds that 
\begin{equation*}
\Gamma_{p,q} \left(\frac{b}{\ln(2)} \right)^{b^*}
\leq \lim_{\varepsilon \to 0} \varepsilon \, \frac{H \left(\varepsilon\semcol \mathcal{E}_p, \|\cdot\|_q \right)^{\left(\frac{1}{p}-\frac{1}{q}\right)}}{\phi \left(H \left(\varepsilon\semcol \mathcal{E}_p, \|\cdot\|_q \right) \right)}
\leq \gamma_{p,q, b} \left(\frac{b}{\ln(2)}  + 1\right)^{b^*},
\end{equation*}
with $\Gamma_{p,q}$ as in (\ref{eq:defgammapq}) and 
\begin{equation*}
b^* = b+\frac{1}{p}-\frac{1}{q}, \quad
\gamma_{p,q,b}
=  \left(\frac{b}{b+\frac{1}{p}-\frac{1}{q}}   \right)^{(\frac{1}{q}-\frac{1}{p})};
\end{equation*}

\item
If $p<q$, it holds that
\begin{equation}\label{eq:jdhhjnjbyugssfkl}
  \Gamma_{p,q} \left(\frac{b}{\ln(2)} \right)^{b^*}
  \leq \lim_{\varepsilon \to 0} \varepsilon \, \frac{H \left(\varepsilon\semcol \mathcal{E}_p, \|\cdot\|_q \right)^{\left(\frac{1}{p}-\frac{1}{q}\right)}}{\phi \left(H \left(\varepsilon\semcol \mathcal{E}_p, \|\cdot\|_q \right) \right)} 
\end{equation}
and
\begin{equation*}
   H \left(\varepsilon\semcol \mathcal{E}_p, \|\cdot\|_q \right) \,
\varepsilon^{1/b^*}
=
\begin{cases}
O_{\varepsilon \to 0}\left(\log^{(2)}(\varepsilon^{-1})\right),
\quad &\text{if } b\geq 1,\\[.25cm]
O_{\varepsilon \to 0}\left(\log^{1-b}(\varepsilon^{-1})\right),
\quad &\text{if } b< 1,
\end{cases}
\end{equation*}%}
with $\Gamma_{p,q}$ as in (\ref{eq:defgammapq}) and 
\begin{equation*}
b^* = b+\frac{1}{p}-\frac{1}{q}.
\end{equation*}
\end{enumerate}
\end{theorem}

\begin{proof}[Proof.]
See Section~\ref{sec:proofofmainthmisinthissec}.
\end{proof}

\noindent
Theorem~\ref{thm: scaling metric entropy infinite ellipsoids regular} covers all possible combinations of $p,q\in[1,\infty]$ and $\phi$ regularly varying with index $-b<0$, except for the case $q=p/(pb+1)$ with semi-axes such that $\lim_{n\to\infty}n \mu_n^{1/b}=0$ and $\sum_{n\in\mathbb{N}^*}\mu_n^{1/b}=\infty$, e.g., 
$\mu_n=c/(n\log(n))^{b}$. 
In the canonical example (\ref{eq:canonicaldecayregvar}) this corner case can not occur. 
In the particular case $p=q$,
Theorem~\ref{thm: scaling metric entropy infinite ellipsoids regular} characterizes the constants $c$ and $C$ in (\ref{eq:akjfvbeajkjslzzl}) according to 
\begin{equation}\label{eq:lkjbhvghbkjljjljoo}
  \left(\frac{b}{\ln(2)} \right)^{b}
  \leq \lim_{\varepsilon \to 0} \frac{\varepsilon}{\phi \left(H \left(\varepsilon\semcol \mathcal{E}_p, \|\cdot\|_p \right) \right)}
  \leq \left(\frac{b}{\ln(2)}  + 1\right)^{b},
  \end{equation}
which
shows that they match in the large-$b$ regime.
We shall see in Sections \ref{sec:hilbertiancase} and \ref{sec:optimalityres1}, respectively, that, for $p=q=2$, the lower and, for $p=q=\infty$, the upper bound in
(\ref{eq:lkjbhvghbkjljjljoo}) is tight. 

For illustration and interpretation purposes, we now particularize Theorem~\ref{thm: scaling metric entropy infinite ellipsoids regular} to semi-axes decaying, to first order, as in the canonical example (\ref{eq:canonicaldecayregvar}).
\begin{corollary}\label{cor:oadcpzrjefnkznrs}
  Let $p, q \in [1, \infty]$, $b>0$, $c>0$, and let $\{\mu_n\}_{n\in \mathbb{N}^*}$ be a non-increasing sequence of positive real numbers such that
    \begin{equation*}
    \mu_n
    =  \frac{c}{n^b} + o_{n \to \infty} \left(\frac{1}{n^{b}}\right).
  \end{equation*}
Let $\mathcal{E}_p$ be the $p$-ellipsoid with semi-axes $\{\mu_n\}_{n\in \mathbb{N}^*}$.
Then, 
\begin{enumerate}[label=(\roman*)]

\item
if $q \leq p/(pb+1)$, the ellipsoid $\mathcal{E}_p$ is not compact in $\ell^q$, i.e.,  
 there exists $\varepsilon^*>0$ such that 
\begin{equation*}
H \left(\varepsilon\semcol \mathcal{E}_p, \|\cdot\|_q \right) 
= \infty,
\quad \text{for all } \varepsilon \in (0, \varepsilon^*);
\end{equation*}

\item 
if $p/(pb+1)<q \leq p$, we have 
\begin{align*}
\frac{b}{\ln(2)}{\left(\frac{\Gamma_{p,q}\,  c}{\varepsilon}\right)}^{1/b^*} + o_{\varepsilon \to 0}\left(\varepsilon^{-1/b^*}\right)
&\leq H \left(\varepsilon\semcol \mathcal{E}_p, \|\cdot\|_q \right)\\
&\leq  \left(\frac{b}{\ln(2)}  + 1\right){\left(\frac{\gamma_{p,q, b}\,  c}{\varepsilon}\right)}^{1/b^*} + o_{\varepsilon \to 0}\left(\varepsilon^{-1/b^*}\right),
\end{align*}
with $\Gamma_{p,q}$ as in (\ref{eq:defgammapq}) and 
\begin{equation*}
b^* = b+\frac{1}{p}-\frac{1}{q},
\quad %\text{and} \quad 
\gamma_{p,q,b}
=  \left(\frac{b}{b+\frac{1}{p}-\frac{1}{q}}   \right)^{(\frac{1}{q}-\frac{1}{p})};
\end{equation*}

\item
if $p<q$, we have 
\begin{align*}
\frac{b}{\ln(2)}{\left(\frac{\Gamma_{p,q}\,  c}{\varepsilon}\right)}^{1/b^*} + o_{\varepsilon \to 0}\left(\varepsilon^{-1/b^*}\right)
&\leq H \left(\varepsilon\semcol \mathcal{E}_p, \|\cdot\|_q \right) \\
 &=
\begin{cases}
O_{\varepsilon \to 0}\left(\varepsilon^{-1/b^*}\,\log^{(2)}(\varepsilon^{-1})\right),
\quad &\text{if } b\geq 1,\\[.25cm]
O_{\varepsilon \to 0}\left(\varepsilon^{-1/b^*}\,\log^{1-b}(\varepsilon^{-1})\right),
\quad &\text{if } b< 1,
\end{cases}
\end{align*}
with $\Gamma_{p,q}$ as in (\ref{eq:defgammapq}) and
\begin{equation*}
b^* = b+\frac{1}{p}-\frac{1}{q}.
\end{equation*}
\end{enumerate}
\end{corollary}

\begin{proof}
  The result follows directly from Theorem~\ref{thm: scaling metric entropy infinite ellipsoids regular} 
  by noting that the non-increasing sequence  $\mu_n  =  \frac{c}{n^b} + o_{n \to \infty} \left(\frac{1}{n^{b}}\right)$ is regularly varying with index $-b$.
\end{proof}

\subsection{The Hilbertian Case}\label{sec:hilbertiancase}

We now show that in the Hilbertian case $p=q=2$ techniques developed in Section~\ref{sec:mixedellipsoidsseparateresults} can be leveraged 
to obtain a precise asymptotic characterization of metric entropy.

\begin{theorem}\label{thm:mainrespqtwo}
Let $b>0$, let $\{\mu_n\}_{n\in \mathbb{N}^*}$ be a non-increasing sequence that is regularly varying with index $-b$, and 
let $\phi \colon \mathbb{R}^*_+ \to \mathbb{R}^*_+$ be the step function given by 
$\phi(t) = \mu_{\lfloor t \rfloor + 1}$, for $t > 0$. Let $\mathcal{E}_2$ be the $2$-ellipsoid with semi-axes $\{\mu_n\}_{n\in \mathbb{N}^*}$.
  Then,
\begin{equation*}
\lim_{\varepsilon \to 0} \frac{\varepsilon}{\phi \left(H \left(\varepsilon\semcol \mathcal{E}_2, \|\cdot\|_2 \right) \right)}
= \left(\frac{b}{\ln(2)}\right)^b.
\end{equation*}
\end{theorem}

\begin{proof}[Proof.]
  The proof is presented in Section~\ref{sec:proofmainresultspqtwo}.
\end{proof}
\noindent
Theorem~\ref{thm:mainrespqtwo} shows that, for $p=q=2$, the lower bound (\ref{eq:lkjbhvghbkjljjljoo}) in Theorem~\ref{thm: scaling metric entropy infinite ellipsoids regular} is tight.
We again particularize to the canonical example (\ref{eq:canonicaldecayregvar}) for illustration purposes.

\begin{corollary}\label{cor:firstorderhilbertzz}
  Let $b>0$, $c>0$, and let $\{\mu_n\}_{n\in \mathbb{N}^*}$ be a non-increasing sequence of positive real numbers such that
    \begin{equation*}
    \mu_n
    =  \frac{c}{n^b} + o_{n \to \infty} \left(\frac{1}{n^{b}}\right).
  \end{equation*}
  Let $\mathcal{E}_2$ be the $2$-ellipsoid with semi-axes $\{\mu_n\}_{n\in \mathbb{N}^*}$.
Then, 
\begin{equation*}
H \left(\varepsilon\semcol \mathcal{E}_2, \|\cdot\|_2\right)
= \frac{b \, {c}^{1/b}}{\ln(2)}\, {\varepsilon}^{-1/b} + o_{\varepsilon \to 0}\left(\varepsilon^{-1/b}\right).
\end{equation*}
\end{corollary}

\begin{proof}
  The result follows directly from Theorem~\ref{thm:mainrespqtwo} by noting that the non-increasing sequence  $\mu_n  =  \frac{c}{n^b} + o_{n \to \infty} \left(\frac{1}{n^{b}}\right)$ is regularly varying with index $-b$.
\end{proof}

\noindent
For $b\geq 1/2$, the results in Theorem~\ref{thm:mainrespqtwo} and Corollary \ref{cor:firstorderhilbertzz} were reported previously in \cite[Theorem 2.1 and Corollary 2.4]{luschgySharpAsymptoticsKolmogorov2004}. 
We next show that Corollary \ref{cor:firstorderhilbertzz}
can be strengthened even further, by providing the second-order term in the asymptotic expansion of metric entropy.

\begin{theorem}\label{thm: Metric entropy for polynomial ellipsoids}
Let $\alpha_1, \alpha_2>0$ be such that
\begin{equation}\label{eq:conditiononalpha12}
\alpha_1<\alpha_2 <  \alpha_1 + 1/2.
\end{equation}
Let $c_1>0$, $c_2 \in \mathbb{R}$, and let $\{\mu_n\}_{n\in \mathbb{N}^*}$ be a non-increasing sequence of positive real numbers such that
\begin{equation}\label{eq:formsemiaxessecondterm}
\mu_n 
= \frac{c_1}{n^{\alpha_1}} + \frac{c_2}{n^{\alpha_2}} + o_{n \to \infty} \left(\frac{1}{n^{\alpha_2}}\right).
\end{equation}
Let $\mathcal{E}_2$ be the $2$-ellipsoid with semi-axes $\{\mu_n\}_{n\in \mathbb{N}^*}$.
Then,
\begin{equation*}
H \left(\varepsilon\semcol \mathcal{E}_2, \|\cdot\|_2\right)
= \frac{\alpha_1{c_1}^{\frac{1}{\alpha_1}}}{\ln(2)} \, {\varepsilon}^{-\frac{1}{\alpha_1}}
+ \frac{c_2\, {c_1}^{\frac{1-\alpha_2}{\alpha_1}}}{\ln(2)(\alpha_1-\alpha_2+1)} \, \varepsilon^{-\frac{\alpha_1-\alpha_2+1}{\alpha_1}} 
+ o_{\varepsilon \to 0}\left(\varepsilon^{-\frac{\alpha_1-\alpha_2+1}{\alpha_1}}\right).
\end{equation*}
\end{theorem}

\begin{proof}[Proof.]
  See Section~\ref{sec:proofofhighlightresult}.
\end{proof}

\noindent
Theorem~\ref{thm: Metric entropy for polynomial ellipsoids} was applied in \cite{thirdpaper} by the authors of the present paper to obtain significant improvements over the previously best known results on the metric entropy of unit balls in Sobolev spaces. 
Moreover, Theorems~\ref{thm:mainrespqtwo} and
\ref{thm: Metric entropy for polynomial ellipsoids} allow us to sharpen the
best known characterization of entropy numbers of diagonal operators given in
\cite[Proposition~1.3.2]{carlEntropyCompactnessApproximation1990}. As shown in
\cite[Theorem~3]{thirdpaper}, this sharpening consists of (i) a full 
determination of the first-order term, rather than bounds differing by a
multiplicative factor~6 as in
\cite[Proposition~1.3.2]{carlEntropyCompactnessApproximation1990}, and (ii) an
explicit characterization of the second-order term when information on the
second-order behavior of the semi-axes is available.
Finally, an extension of Theorem~\ref{thm: Metric entropy for polynomial
ellipsoids} beyond polynomially decaying semi-axes is provided in \cite[Theorem~2]{Pinskerpaper}.

\subsection{Hyperrectangles}\label{sec:optimalityres1}

We next show that, for $p=q=\infty$, an exact, as opposed to asymptotic, expression for metric entropy, valid for all $\varepsilon > 0$, can be obtained.
To the best of our knowledge, this is the first time an exact expression for the metric entropy of an infinite-dimensional body has been derived.

{\begin{theorem}\label{thm:infellpinfnorm}
Let $\{\mu_n\}_{n \in \mathbb{N}^*}$ be a sequence of positive real numbers and let $\mathcal{E}_\infty$ be the $\infty$-ellipsoid with semi-axes $\{\mu_n\}_{n\in \mathbb{N}^*}$.
Then, for all $\varepsilon >0$, we have 
\begin{equation}\label{eq:okijhugffcdsvhcdz2214}
H \left(\varepsilon\semcol \mathcal{E}_\infty, \|\cdot\|_\infty\right)
= \sum_{k=1}^{\infty} \log \left(1+\frac{1}{k} \right)M_k(\varepsilon),
\end{equation}
with the counting function
\begin{equation}\label{eq:hjqjhjhbsfvbfbvsjkhbvjhbsfvbsjksjfd}
  M_k \colon t \in \mathbb{R}_+^* \longmapsto \#\left\{n \in \mathbb{N}^* |\, \mu_n>  k  t \right\}, 
  \quad \text{for } k\in \mathbb{N}^*.
\end{equation}
\end{theorem}%}

\begin{proof}[Proof.]
See Section~\ref{sec:proofthminfellpinfnorm}.
\end{proof}

\noindent
We emphasize that Theorem~\ref{thm:infellpinfnorm} applies to general---i.e., not necessarily non-increasing and tending to zero---semi-axes.
For decreasing semi-axes $\{\mu_n\}_{n \in \mathbb{N}^*}$,
it is readily verified that (\ref{eq:okijhugffcdsvhcdz2214}) can be rewritten 
as
\begin{equation}\label{eq:okijhugffcdsvhcdz}
  H \left(\varepsilon\semcol \mathcal{E}_\infty, \|\cdot\|_\infty\right)
  = \sum_{k=1}^{\infty} \log \left(1+\frac{1}{k} \right)\left(\left\lceil\phi^{-1}(k\varepsilon)\right\rceil-1\right),
\end{equation}
where $\phi \colon \mathbb{R}^*_+ \to \mathbb{R}^*_+$ is an arbitrary decreasing function satisfying $\phi(n)=\mu_n$, for all $n\in \mathbb{N}^*$.
Particularizing (\ref{eq:okijhugffcdsvhcdz}) to the canonical example, yields
an expression for the asymptotic behavior of $\mathcal{E}_\infty$ in terms of samples of the Riemann zeta function as follows.

\begin{corollary}\label{cor:infellpinfnorm}
Let $b>0$, $c>0$, and let
\begin{equation*}
\mu_n
= \frac{c}{n^b},
\quad \text{for } n\in \mathbb{N}^*.
\end{equation*}
Let $\mathcal{E}_\infty$ be the $\infty$-ellipsoid with semi-axes $\{\mu_n\}_{n\in \mathbb{N}^*}$.
Then, 
\begin{equation}\label{eq:lbforoptimalityinfinfmainres}
H \left(\varepsilon\semcol \mathcal{E}_\infty, \|\cdot\|_\infty\right)
= \frac{c^{1/b}\, \varepsilon^{-1/b}}{\ln(2)}\sum_{\ell=1}^\infty \frac{(-1)^{\ell+1}}{\ell}\zeta\left(\ell+\frac{1}{b}\right) + O_{\varepsilon \to 0}\left( \log\left(\varepsilon^{-1}\right)\right),
\end{equation}
where $\zeta(\cdot)$ denotes the Riemann zeta function.
In particular, we have the lower bound
\begin{equation}\label{eq:lbforoptimalityinfinf}
H \left(\varepsilon\semcol \mathcal{E}_\infty, \|\cdot\|_\infty\right)
\geq c^{1/b}\, \varepsilon^{-1/b} + O_{\varepsilon \to 0}\left( \log\left(\varepsilon^{-1}\right)\right).
\end{equation}
\end{corollary}

\begin{proof}[Proof.]
The proof is presented in Section~\ref{sec:proofcorinfellpinfnorm}.
\end{proof}

\noindent
We have already argued in Section~\ref{sec:hilbertiancase} that the lower bound (\ref{eq:lkjbhvghbkjljjljoo}) in Theorem~\ref{thm: scaling metric entropy infinite ellipsoids regular} cannot be improved in general.  
Corollary \ref{cor:infellpinfnorm} now allows us to draw the same conclusion for the upper bound 
in (\ref{eq:lkjbhvghbkjljjljoo}).
Indeed, it follows from (\ref{eq:lbforoptimalityinfinf}) and (\ref{eq:lkjbhvghbkjljjljoo}) that
\begin{align*}
c^{1/b}\, \varepsilon^{-1/b}\left(1 + o_{\varepsilon \to 0}\left(1\right)\right)
\leq H \left(\varepsilon\semcol \mathcal{E}_\infty, \|\cdot\|_\infty\right)
\leq {c}^{1/b} \, \left(\frac{b}{\ln(2)}  + 1\right)  {\varepsilon}^{-1/b}\left(1 + o_{\varepsilon \to 0}\left(1\right)\right),
\end{align*}
which shows that the lower and upper bounds match for small values of $b$.

We conclude by remarking that the proofs of Theorem~\ref{thm:infellpinfnorm} and Corollary \ref{cor:infellpinfnorm} explicitly construct optimal coverings for $\infty$-ellipsoids.

\section{Applications to Besov Spaces}\label{sec:Besovapplic}

We now show how the tools developed in this paper can be applied to characterize the metric entropy of function sets in 
Besov spaces $B_{p_1, p_2}^s(\Omega)$ with
domain (nonempty, bounded, and open) $\Omega \subset \mathbb{R}^d$ (see, e.g., \cite[Section~C.3]{grohsPhaseTransitionsRate2021} for a definition). These spaces provide a rich framework that encompasses a wide range of classical function spaces.
Notably, they correspond to Hölder-Zygmund spaces when $p_1=p_2=\infty$ (see, e.g., \cite[Chapter 2.5.7]{triebelTheoryFunctionSpaces1983}) and are closely related to Sobolev spaces or, more generally, Triebel-Lizorkin spaces (see, e.g., \cite[Chapter 2.3.2, Proposition 2 and Chapter 2.5.6]{triebelTheoryFunctionSpaces1983} for a formalization of this relation). 
The result we obtain for Besov spaces in Theorem~\ref{thm:depomdomainn} below hence immediately applies to all these spaces.

Summarily, the results available in the literature say that the metric entropy of a class of functions $\mathcal{F}_s$ with
domain $\Omega \subset \mathbb{R}^d$ in Sobolev spaces scales according to
\begin{equation}\label{eq:besov-scaling}
H \left( \varepsilon \semcol \mathcal{F}_s, \|\cdot\|_{L^2(\Omega)} \right)
\asymp_{\varepsilon \to 0} \varepsilon^{-d/s}.
\end{equation}
Such a scaling behavior was first formally established for $p$-Sobolev spaces in \cite{Birman_1967,birman1980quantitative}.
The corresponding method of proof has become standard and was since used widely in the literature, see, e.g., \cite{edmunds1989entropy,Carl_1981,edmunds1979embeddings,konigEigenvalueDistributionCompact1986,bookeigen1987pietsch,edmunds1992entropy,edmundsFunctionSpacesEntropy1996}.
Specifically, \cite{Carl_1981} used similar techniques to prove \eqref{eq:besov-scaling} for function classes $\mathcal{F}_s$ in general Besov spaces (see also \cite[Proposition 3.c.9]{konigEigenvalueDistributionCompact1986} and \cite{bookeigen1987pietsch} for an approach based on splines).
Further extensions to the embedding of a Besov space into another Besov space are provided in \cite{edmunds1989entropy,edmunds1992entropy}.
A summary of the techniques underlying all these results can be found in \cite[Sections 3.3.2 and 3.3.3]{edmundsFunctionSpacesEntropy1996}.

None of the results available in the literature identifies how the metric entropy depends on the domain $\Omega$ of the functions in $\mathcal{F}_{s}$.
We next provide such a result by first exhibiting the ellipsoidal structure of Besov spaces and then applying Corollary \ref{cor:oadcpzrjefnkznrs}.

\begin{theorem}\label{thm:depomdomainn}
Let $d\in \mathbb{N}^*$, $p_1\in [2, \infty]$,  $p_2 \in [1, \infty]$, and $s>0$.
Then, there exist $\varepsilon^*>0$, $c>0$, and $C>0$ such that, for all nonempty, bounded, and open sets $\Omega \subset \mathbb{R}^d$ and for all $\varepsilon \in (0, \varepsilon^*)$, we have 
\begin{equation}\label{eq:knadcnzaefe}
c \,  \vol{\left(\Omega\right)}^{1-\frac{d}{s}(\frac{1}{p_1}-\frac{1}{2})}\,\varepsilon^{-d/s}
\leq  H \left( \varepsilon \semcol \mathcal{B}_{p_1, p_2}^s(\Omega), \|\cdot\|_{L^2(\Omega)}  \right)
\leq C \, \vol{\left(\Omega\right)}^{1-\frac{d}{s}(\frac{1}{p_1}-\frac{1}{2})}\,\varepsilon^{-d/s},
\end{equation}
where $\mathcal{B}_{p_1, p_2}^s(\Omega)$ denotes the unit ball in the corresponding Besov space.
\end{theorem}

\begin{proof}
The proof can be found in Section~\ref{sec:proofthm199}.
\end{proof}

\noindent
We emphasize that the constants $c$ and $C$ in (\ref{eq:knadcnzaefe}) may depend on $p_1$, $p_2$, $d$, $s$, and $\varepsilon^*$, but are independent of $\Omega$ and $\varepsilon$. 
Consequently, the dependence of $H \left( \varepsilon \semcol \mathcal{B}_{p_1, p_2}^s(\Omega), \|\cdot\|_2 \right)$
on the domain $\Omega$ enters solely through the proportionality---to $\vol{\left(\Omega\right)}^{1-\frac{d}{s}(\frac{1}{p_1}-\frac{1}{2})}$---of the constants in the lower and upper bound in \eqref{eq:knadcnzaefe}. In particular, for $p_1=2$ this yields proportionality to $\vol(\Omega)$.
Although Theorem~\ref{thm:depomdomainn} is stated for $p_1 \in [2, \infty]$, the same techniques used in its proof, combined with
(iii) in Corollary \ref{cor:oadcpzrjefnkznrs}, extend the result to $p_1\in [1, \infty]$.

Further, the statement in Theorem~\ref{thm:depomdomainn} can be made more precise for $p$-Sobolev spaces, in the case $p=2$. 
Indeed, in \cite{thirdpaper} we determined the exact constant in the leading term of the small-$\varepsilon$ expansion of the metric entropy and, for $d\geq 3$ under mild additional assumptions on $\Omega$, we also identified the second-order term.
These results build on the techniques developed in Section~\ref{sec:hilbertiancase}.

%\textcolor{blue}{
Finally, the bounds in \eqref{eq:knadcnzaefe} admit an equivalent formulation in terms of the entropy numbers $\{e_n\}_{n\in\mathbb{N}^*}$ of the embedding map of ${B}_{p_1, p_2}^s(\Omega)$ into $L^2(\Omega)$.
This reformulation enables direct comparison with the literature on the metric entropy of Besov spaces (see, e.g., \cite{edmunds1989entropy,Carl_1981,edmunds1979embeddings,konigEigenvalueDistributionCompact1986,bookeigen1987pietsch,edmunds1992entropy,edmundsFunctionSpacesEntropy1996} cited above), where results are typically stated in terms of entropy numbers.
Specifically, there exist constants $c',C'>0$ independent of $\Omega$ and $n$, such that
%we get%}
\begin{equation*}
%\textcolor{blue}{ 
c' \vol{\left(\Omega\right)}^{\frac{s}{d} +\frac{1}{2}  -\frac{1}{p_1}} n^{-\frac{s}{d}}
  \leq e_n 
  \leq C' \vol{\left(\Omega\right)}^{\frac{s}{d} +\frac{1}{2}  -\frac{1}{p_1}} n^{-\frac{s}{d}},
  \quad \text{for all }n\in\mathbb{N}^*.%}
\end{equation*}

\section{Proofs of the Main Results}

\subsection{Proof of Lemma~\ref{lem:lemboundtriangleineqlike}}\label{sec:prooflemboundtriangleineqlike}

From assumption (\ref{eq:firstdecompcartprodelli}) we have
 \begin{equation}\label{eq:covnb1324}
N \left( \rho\semcol \mathcal{E}_p, \|\cdot\|_q \right)
\leq \left(\# \Omega\right)
\,  \max_{\omega \in \Omega}  \, N \left( \rho\semcol \prod_{j=1}^{k+1} \omega_j \mathcal{E}_p^{[j]} , \|\cdot\|_q \right),
 \end{equation}
 for all $\rho>0$.
 In particular, setting $\rho =\bar \rho $ in (\ref{eq:covnb1324}) and using $\rho_\omega \leq \bar \rho$, for all $\omega \in \Omega$, it follows that
 
 \begin{equation*}
   N \left(\bar \rho\semcol \mathcal{E}_p, \|\cdot\|_q \right)
\leq \left(\# \Omega\right)
\,  \max_{\omega \in \Omega}  \, N \left( \rho_\omega\semcol \prod_{j=1}^{k+1} \omega_j \mathcal{E}_p^{[j]} , \|\cdot\|_q \right).
 \end{equation*}
 In order to simplify the exposition, let us introduce the notation
 \begin{equation*}
 \mathcal{E}_{p, \omega}^{\times}
 \coloneqq \prod_{j=1}^{k+1} \omega_j \mathcal{E}_p^{[j]},
 \quad \text{for }
 \omega \in \Omega.
 \end{equation*}
 The proof can be completed by showing that
 \begin{equation}\label{eq:lknjhgdshlkqncvzezzzzsxx2}
   \, N \left( \rho_\omega\semcol \mathcal{E}_{p, \omega}^{\times} , \|\cdot\|_q \right)
   \leq \prod_{j=1}^{k} N (\rho_j\semcol \mathcal{E}_p^{[j]}, \|\cdot\|_q ),
    \quad \text{for all} \ 
 \omega \in \Omega.
 \end{equation}
To this end, we fix $\omega \in \Omega$.
For $j\in \{1, \dots, k\}$, let $\{x^{(1, j)}, \dots, x^{(N_j,j)}\}$ be a $\rho_j$-covering of $\mathcal{E}_p^{[j]}$, in $\|\cdot\|_q$-norm, of cardinality $N_j \coloneqq N (\rho_j\semcol \mathcal{E}_p^{[j]}, \|\cdot\|_q )$.
We further introduce the infinite-dimensional vectors 
\begin{equation*}
\bar x^{(i_1, \dots, i_k)}
\coloneqq (\omega_1x^{(i_1, 1)}, \omega_2x^{(i_2, 2)}, \dots, \omega_kx^{(i_k, k)},  0,  \dots), 
\quad \text{where } i_j \in \{1, \dots, N_j\},
\end{equation*}
for all $j\in \{1, \dots, k\}$.
There are $N \coloneqq N_1 \times \dots \times N_k$ such vectors.
To establish (\ref{eq:lknjhgdshlkqncvzezzzzsxx2}), it hence suffices to show that the set 
$\{\bar x^{(i_1 \dots, i_k)}\}_{i_1, \dots, i_k} $ forms a $\rho_\omega$-covering of $\mathcal{E}_{p, \omega}^{\times}$.
With this goal in mind, we arbitrarily fix $\bar x\in \mathcal{E}_{p, \omega}^{\times}$ and define the associated constituent vectors
\begin{equation*}
x^{(j)}
\coloneqq (\bar x_{\bar d_{j-1}+1}, \dots, \bar x_{\bar d_j}),
\quad \text{for } j\in \{1, \dots, k\}.
\end{equation*}
(Recall the notation $\bar d_j$ introduced in Definition \ref{def:blockdecomp}.)
Now, for all $j\in \{1, \dots, k\}$, note that $x^{(j)} \in \omega_j \mathcal{E}_p^{[j]}$ and let $i^*_j \in \{1, \dots, N_j\}$ be such that 
\begin{equation}\label{eqlkiudizerufhhjhjs0}
\left\|x^{(j)}-\omega_j x^{(i^*_j, j)}\right\|_q\leq \omega_j \rho_j.
\end{equation}
The existence of this $i^*_j$ follows from $\{x^{(1, j)}, \dots, x^{(N_j,j)}\}$ being a $\rho_j$-covering of $\mathcal{E}_p^{[j]}$.

First, we consider the case $q=\infty$ and note that (\ref{eqlkiudizerufhhjhjs0}) becomes
\begin{equation}\label{eqlkiudizerufhhjhjs0prime}
\max_{\ell =1,\dots, d_j} \left\lvert x^{(j)}_{\ell }-\omega_j x^{(i^*_j, j)}_{\ell}\right\rvert 
= \max_{\ell=1,\dots, d_j} \left\lvert \bar x_{\bar d_{j-1}+\ell}- \bar x^{(i^*_1, \dots, i^*_k)}_{\bar d_{j-1}+\ell}\right\rvert 
\leq \omega_j \rho_j,
\end{equation}
for all $j\in \{1, \dots, k\}$.
We therefore have the decomposition
\begin{align}
\left\|\bar x- \bar x^{(i^*_1, \dots, i^*_k)} \right\|_\infty
&= \sup_{n\in \mathbb{N}^*} \left\lvert\bar x_n - \bar x_n^{(i^*_1, \dots, i^*_k)}\right\rvert \nonumber \\
&= \max\left\{\max_{j=1, \dots, k} \ \max_{\ell=1,\dots, d_j} \left\lvert \bar x_{\bar d_{j-1}+\ell}- \bar x^{(i^*_1, \dots, i^*_k)}_{\bar d_{j-1}+\ell}\right\rvert , \sup_{n > d}\left\lvert\bar x_n\right\rvert\right\}. \label{eqlkiudizerufhhjhjsinf222}
\end{align}
The first term inside the outer-most maximum in (\ref{eqlkiudizerufhhjhjsinf222}) is upper-bounded by $\omega_j \rho_j$ thanks to (\ref{eqlkiudizerufhhjhjs0prime}), while the second term is upper-bounded by $\omega_{k+1}\mu_{d+1}$. This shows that the set $\{\bar x^{(i_1 \dots, i_k)}\}_{i_1, \dots, i_k} $, indeed, constitutes a $\rho_\omega$-covering of $\mathcal{E}_{p, \omega}^\times$, thereby concluding the proof for $q=\infty$.

We proceed to the case $1\leq q < \infty$
and start by observing that 
\begin{align}
\left\|\bar x- \bar x^{(i^*_1, \dots, i^*_k)} \right\|_q^q
&= \sum_{n\in \mathbb{N}^*} \left\lvert\bar x_n - \bar x_n^{(i^*_1, \dots, i^*_k)}\right\rvert^q \label{eqlkiudizerufhhjhjs21} \\
&= \sum_{j=1}^k \sum_{\ell=1}^{d_j} \left\lvert \bar x_{\bar d_{j-1}+\ell}- \omega_j x^{(i^*_j, j)}_{\ell} \right\rvert^q + \sum_{n= d+1}^\infty \left\lvert\bar x_n\right\rvert^q. \label{eqlkiudizerufhhjhjs22}
\end{align}
The sum over the index $\ell$ in (\ref{eqlkiudizerufhhjhjs22}) is equal to $\|x^{(j)}-\omega_j x^{(i^*_j, j)}\|_q^q$, and is, via (\ref{eqlkiudizerufhhjhjs0}), upper-bounded by $(\omega_j \rho_j)^q$, for all $j\in \{1, \dots, k\}$.
The term $\sum_{n= d+1}^\infty \left\lvert\bar x_n\right\rvert^q$ can be dealt with by distinguishing three cases.

\paragraph{The case $p\leq q<\infty$.}
We write 
\begin{align*}
\sum_{n=d+1}^\infty \left\lvert\bar x_n\right\rvert^q
&= \sum_{n=d+1}^\infty \left\lvert\bar x_n/\mu_n\right\rvert^q \mu_n^q \\
&\leq \mu_{d+1}^q \sum_{n=d+1}^\infty \left\lvert\bar x_n/\mu_n\right\rvert^q 
\leq \mu_{d+1}^q \sum_{n= d+1}^\infty \left\lvert\bar x_n/\mu_n\right\rvert^p 
\leq (\omega_{k+1}\mu_{d+1})^q,
\end{align*}
where the first inequality is a consequence of the semi-axes being non-increasing, the second inequality follows from $q\geq p$, and the last one is by
$\bar x \in \mathcal{E}_{p, \omega}^\times$.
In summary, we have 
\begin{align*}
\left\|\bar x- \bar x^{(i^*_1, \dots, i^*_k)} \right\|_q^q
\leq (\omega_1 \rho_1)^q + \dots + (\omega_k\rho_k)^q + (\omega_{k+1} \mu_{ d+1})^q 
=\rho_\omega^q, 
\end{align*}
which shows that the set $\{\bar x^{(i_1 \dots, i_k)}\}_{i_1, \dots, i_k} $, indeed, constitutes a $\rho_\omega$-covering of $\mathcal{E}_{p, \omega}^\times$, thereby concluding the proof of statement (i).

\paragraph{The case $p/(pb+1)<q<p$.}
We proceed to prove statement (ii). In this case, the quantity $\sum_{n= d+1}^\infty \left\lvert\bar x_n\right\rvert^q$ in (\ref{eqlkiudizerufhhjhjs22}) can be upper-bounded, by application of the Hölder inequality with the parameters $p/q$ and $(1-q/p)^{-1}$, according to
\begin{align}
\sum_{n= d+1}^\infty \left\lvert\bar x_n\right\rvert^q
&=\sum_{n= d +1}^\infty \left\lvert\bar x_n/\mu_n \right\rvert^q \mu_n^q \label{eq:kjkhqdcsbindczoeoajzidnerbgt1} \\
&\leq \left(\sum_{n= d +1}^\infty \left\lvert\bar x_n/\mu_n\right\rvert^p\right)^{q/p}
\left(\sum_{n= d +1}^\infty \mu_n^{q\left(1-\frac{q}{p}\right)^{-1}}\right)^{1-\frac{q}{p}},
\label{eq:kjkhqdcsbindczoeoajzidnerbgt2}
\end{align}
with the usual modifications when $p=\infty$.
As $\bar x \in \mathcal{E}_{p, \omega}^\times$, we must have $\{\bar x_n\}_{n=d+1}^\infty \in  \omega_{k+1} \mathcal{E}_p^{[k+1]}$, which in turn implies 
\begin{equation}\label{eq:upper-bound-last-ellipsoid}
\left(\sum_{n= d +1}^\infty \left\lvert\bar x_n/\mu_n\right\rvert^p\right)^{q/p}
\leq \omega_{k+1}^q.
\end{equation}
Now, let $\phi \colon \mathbb{R}^*_+ \to \mathbb{R}^*_+$ be the step function given by 
$\phi(t) = \mu_{\lfloor t \rfloor + 1}$, for $t > 0$.
In particular, $\phi$ is non-increasing and \cite[Theorem~1.9.5]{binghamRegularVariation1987} ensures that it is regularly varying with index $-b$.
Moreover, observing that 
\begin{equation*}
     \mu_n^{q\left(1-\frac{q}{p}\right)^{-1}}
    =  \int_{n-1}^{n} \phi(t)^{q\left(1-\frac{q}{p}\right)^{-1}} dt,
    \quad \text{for all } n \in \mathbb{N}^*,
\end{equation*}
we can write the second sum in (\ref{eq:kjkhqdcsbindczoeoajzidnerbgt2}) as an integral according to
\begin{equation}\label{eq:flksnvgjkrnsdserint}
  \sum_{n= d +1}^\infty \mu_n^{q\left(1-\frac{q}{p}\right)^{-1}}
  = \int_{d}^{\infty} \phi(t)^{q\left(1-\frac{q}{p}\right)^{-1}}dt.
\end{equation}%}
Next, applying Karamata's theorem (see \cite[Theorem 1.5.11, case (ii)]{binghamRegularVariation1987}) to the function 
$\phi(\cdot)^{q\left(1-\frac{q}{p}\right)^{-1}}$, 
which is regularly varying with index $-b/(1/q-1/p)<-1$, the integral in (\ref{eq:flksnvgjkrnsdserint}) is seen to satisfy
\begin{equation*}
    \int_{d}^{\infty} \phi(t)^{q\left(1-\frac{q}{p}\right)^{-1}}dt
    \sim_{d\to\infty} \left(\frac{b}{\frac{1}{q}-\frac{1}{p}}  -1 \right)^{-1}  d \, \phi(d)^{q\left(1-\frac{q}{p}\right)^{-1}}. 
\end{equation*}
With $\mu_{d+1}=\mu_{d}(1+o_{d\to\infty}(1))$ thanks to \cite[Lemma 1.9.6] {binghamRegularVariation1987}, we have
thus established that 
\begin{align}\label{eq:lkjhgfzaghidlzdoqpp} 
 \sum_{n= d+1}^\infty {\mu_n}^{q\left(1-\frac{q}{p}\right)^{-1} } 
 \leq \left(\frac{b}{\frac{1}{q}-\frac{1}{p}}  -1 \right)^{-1}  d \, {\mu_d}^{q\left(1-\frac{q}{p}\right)^{-1} } 
\left(1+o_{ d\to\infty}(1)\right).
\end{align}
Inserting (\ref{eq:lkjhgfzaghidlzdoqpp}) into (\ref{eq:kjkhqdcsbindczoeoajzidnerbgt1})--(\ref{eq:kjkhqdcsbindczoeoajzidnerbgt2}) yields
\begin{equation}\label{eq:jkojzoeijoifjerognjjnj}
\sum_{n= d+1}^\infty \left\lvert\bar x_n\right\rvert^q
\leq (\omega_{k+1}\alpha_{ d})^q,
\end{equation}
with $\{\alpha_d\}_{d\in \mathbb{N}^*}$ a sequence satisfying
 \begin{equation*}
     \alpha_d 
= \left(\frac{b}{\frac{1}{q}-\frac{1}{p}}  -1 \right)^{(\frac{1}{p}-\frac{1}{q})}  \, \mu_d \, d^{(\frac{1}{q}-\frac{1}{p})}
    \left(1+o_{d\to\infty}(1)\right).
 \end{equation*}
Finally, combining (\ref{eq:jkojzoeijoifjerognjjnj}) with (\ref{eqlkiudizerufhhjhjs21})--(\ref{eqlkiudizerufhhjhjs22}), we obtain
\begin{align*}
\left\|\bar x- \bar x^{(i^*_1, \dots, i^*_k)} \right\|_q^q
\leq (\omega_{1} \rho_1)^q + \dots + (\omega_{k} \rho_k)^q + (\omega_{k+1} \alpha_{ d})^q
=\rho_\omega^q, 
\end{align*}
which establishes that the set $\{\bar x^{(i_1 \dots, i_k)}\}_{i_1, \dots, i_k} $, indeed, constitutes a $\rho_\omega$-covering of $\mathcal{E}_{p, \omega}^\times$, 
thereby concluding the proof of statement (ii).

\paragraph{The case $q=p/(pb+1)$.}
We finalize the proof by establishing statement (iii).
As in case (ii) we start with the Hölder inequality to get \eqref{eq:kjkhqdcsbindczoeoajzidnerbgt1}--\eqref{eq:kjkhqdcsbindczoeoajzidnerbgt2} and we note that \eqref{eq:upper-bound-last-ellipsoid} holds.
Now, as $q=p/(pb+1)$ is equivalent to $b=1/q-1/p$, it follows that
\begin{equation}\label{eq:jkojzoeijoifjerognjjnjooo}
\sum_{n= d+1}^\infty \left\lvert\bar x_n\right\rvert^q
\leq (\omega_{k+1}\alpha_{ d})^q,
\end{equation}
with $\alpha_d = \left(\sum_{n=d+1}^\infty \mu_n^{1/b}\right)^b$. The remainder of the proof proceeds exactly as in case (ii).

\subsection{Proof of Theorem~\ref{thm:profinitethm}}\label{sec:prooftheoremprofinitethm}

The proof relies on density arguments and is inspired by the foundational work in \cite{rogersNoteCoverings1957, rogersCoveringSphereSpheres1963, rogersbook1964}.
Take $\varepsilon > 0$ to be fixed throughout the proof.
We start with the observation, used before in \cite[Theorems 4 and 5]{firstpaper}, that
$\mathcal{E}^d_p$ is the image of 
$\mathcal{B}_p$ under the diagonal operator 
\begin{equation}\label{eq:nkjnqlragjkenazkjbvaaaaassqa}
A_\mu = 
\begin{pmatrix}
\mu_1 & 0 & \dots & 0 \\
0 & \mu_2 & \dots & 0 \\
\vdots & \vdots & \ddots & \vdots \\
0 & 0 & \dots & \mu_d 
\end{pmatrix}.
\end{equation}
Indeed, we have 
\begin{align*}
\mathcal{E}^d_p 
&= \left\{ x \in \mathbb{R}^d \mid \|x\|_{p, \mu} \leq 1 \right\} \\
&= \left\{ x \in \mathbb{R}^d \mid \|z\|_{p}\leq 1, \text{ such that  } x_n = \mu_n z_n, \text{ for } n \in \{1, \dots, d\} \right\} \\
&= \left\{ A_\mu z \mid  z \in \mathbb{R}^d, \text{ with } \|z\|_{p} \leq 1 \right\} 
= A_\mu \mathcal{B}_p.
\end{align*}
It hence follows that
\begin{equation}\label{eq: volume ellipsoid}
\vol(\mathcal{E}^d_p) 
=\vol(\mathcal{B}_p) \, \text{det}(A_\mu) 
=  \vol(\mathcal{B}_p) \, \bar \mu_d^{d}.
\end{equation}
We can now apply the volume estimates in Lemma~\ref{lem: volume estimates}, with $\mathcal{B}= \mathcal{E}^d_p$ and $\mathcal{B}'= \mathcal{B}_q$, to obtain the first part of the statement according to
\begin{equation*}
N\left(\varepsilon\semcol \mathcal{E}^d_p, \|\cdot\|_q \right) \, \varepsilon^d
\geq \frac{\vol{\left(\mathcal{E}^d_p\right)}}{\vol{(\mathcal{B}_q)}} 
\stackrel{(\ref{eq: volume ellipsoid})}{=} \frac{\vol{\left(\mathcal{B}_p\right)}}{\vol{(\mathcal{B}_q)}} \bar \mu_d^d = V_{p, q, d}^d \, \bar \mu_d^d.
\end{equation*}

It will be convenient to introduce the following quantities
\begin{equation*}
\varepsilon_1 \coloneqq \frac{d \, \ln(d)}{1+ d \, \ln(d)} \, \varepsilon 
\quad \text{and} \quad
\varepsilon_2 \coloneqq \frac{1}{1+ d \, \ln(d)} \, \varepsilon,
\end{equation*}
which obviously satisfy 
\begin{equation}\label{eq:njhgfdfgfhqsfdjhvdhjze}
\frac{\varepsilon_1}{\varepsilon_2}
= d \, \ln(d)
\quad \text{and} \quad
\varepsilon = \varepsilon_1 + \varepsilon_2.
\end{equation}
Given $N_1 \in \mathbb{N^*}$, a compact set $\mathcal{K}\subset \mathbb{R}^d$, and a set $S\coloneqq {\{x_1, \dots, x_{N_1}\}\subset \mathbb{R}^d}$ of $N_1$ points, we next consider the proportion of $\mathcal{K}$ not covered by $\|\cdot\|_q$-balls of radius $\varepsilon_1$ centered at the points in $S$ as
\begin{equation*}
s_{\mathcal{K}} (x_1, \dots, x_{N_1})
\coloneqq 1-  \frac{\vol\left(\left(\bigcup_{i=1}^{N_1}\mathcal{B}_q(x_i, \varepsilon_1)\right)\cap\mathcal{K} \right)}{\vol\left(\mathcal{K}\right)},
\end{equation*}
which, upon introduction of the indicator function $\chi(\cdot)$ on $\mathcal{B}_q(0, \varepsilon_1)$, can be rewritten as
\begin{equation*}
s_{\mathcal{K}} (x_1, \dots, x_{N_1})
= \frac{1}{\vol\left(\mathcal{K}\right)}\int_{\mathcal{K}} \prod_{i=1}^{N_1}\left(1- \chi(x-x_i)\right)\, \text{d} x.
\end{equation*}
We next bound $ s_{\mathcal{K}} $ for $\mathcal{K} = \mathcal{E}_p^d + \varepsilon_2 \mathcal{B}_q$ by applying the probabilistic method.
To this end, with $N_1$ i.i.d. random variables $\{X_1, \dots, X_{N_1}\}$ uniformly distributed in $\bar{\mathcal{K}}\coloneqq \mathcal{E}_p^d + \varepsilon \mathcal{B}_q$, we consider the quantity
\begin{align*}
\mathbb{E}\left[s_{\mathcal{K}} (X_1, \dots, X_{N_1})\right]
&=\frac{1}{\vol\left(\bar{\mathcal{K}}\right)^{N_1}} \int_{\bar{\mathcal{K}}^{N_1}} s_{\mathcal{K}} (x_1, \dots, x_{N_1}) \, \text{d} x_1 \dots  \text{d} x_{N_1}\\
&=\frac{1}{\vol\left(\bar{\mathcal{K}}\right)^{N_1}\vol\left(\mathcal{K}\right)} \int_{\bar{\mathcal{K}}^{N_1}} \int_{\mathcal{K}} \prod_{i=1}^{N_1}\left(1- \chi(x-x_i)\right)\, \text{d} x \, \text{d} x_1 \dots  \text{d} x_{N_1}\\
&=\frac{1}{\vol\left(\bar{\mathcal{K}}\right)^{N_1}\vol\left(\mathcal{K}\right)} \int_{\mathcal{K}} \int_{\bar{\mathcal{K}}^{N_1}}  \prod_{i=1}^{N_1}\left(1- \chi(x-x_i)\right) \, \text{d} x_1 \dots  \text{d} x_{N_1} \, \text{d} x\\
&=\frac{1}{\vol\left(\mathcal{K}\right)} \int_{\mathcal{K}} \prod_{i=1}^{N_1}\left(\frac{1}{\vol\left(\bar{\mathcal{K}}\right)}\int_{\bar{\mathcal{K}}} (1- \chi(x-x_i))\, \text{d} x_i\right)\, \text{d} x \\
&=\frac{1}{\vol\left(\mathcal{K}\right)} \int_{\mathcal{K}} \prod_{i=1}^{N_1}\left(1- \left(\frac{\varepsilon_1^d\, \vol\left(\mathcal{B}_q\right)}{\vol\left(\bar{\mathcal{K}}\right)}\right)\right)\, \text{d} x
= \left(1- \frac{\rho}{N_1}\right)^{N_1} < e^{-\rho},
\end{align*}
where 
\begin{equation}\label{eq:thisthedensityboundrgrs1}
\rho
\coloneqq \frac{\varepsilon_1^d\, \vol\left(\mathcal{B}_q\right)}{\vol\left(\mathcal{E}_p^d + \varepsilon\mathcal{B}_q\right)}\, N_1,
\end{equation}
and we used $1-t<e^{-t}$, for all $t>0$.
In particular, this implies the existence of $N_1$ points $\{x_1^*, \dots, x^*_{N_1}\}\subset \mathcal{E}_p^d + \varepsilon \mathcal{B}_q$ such that 
\begin{equation}\label{eq:kjbdjhsqkvhjezhjvf}
s_{\mathcal{K}} (x_1^*, \dots, x^*_{N_1}) < e^{-\rho}.
\end{equation}

We next define $N_2\in \mathbb{N}$ to be the maximum number of points $\{y_1^*, \dots, y_{N_2}^*\}\subset \mathcal{E}_p^d$ such that 
$$
\mathcal{B}_q(x^*_i, \varepsilon_1) \cap \mathcal{B}_q(y^*_j, \varepsilon_2)= \emptyset, \quad \text{ for all } i \in \{1, \dots, N_1\} \text{ and all } j\in \{1, \dots, N_2\},
$$
and 
$$
\mathcal{B}_q(y^*_j, \varepsilon_2) \cap \mathcal{B}_q(y^*_k, \varepsilon_2)= \emptyset, \quad \text{ for all } 
j,k\in \{1, \dots, N_2\} \text{ with } j \neq k.
$$
(Note that we allow $N_2=0$.) In particular, the (nonoverlapping) balls $\mathcal{B}_q(y^*_j, \varepsilon_2)$, $j \in \{1, \dots, N_2\} $, are contained in the portion of $\mathcal{K} = \mathcal{E}_p^d + \varepsilon_2\mathcal{B}_q$ not covered by the balls $\mathcal{B}_q(x^*_i, \varepsilon_1)$, $i \in \{1, \dots, N_1\} $, so that
\begin{equation*}
N_2 \, \varepsilon_2^d \, \vol\left(\mathcal{B}_q\right)
\leq \vol\left(\mathcal{E}_p^d + \varepsilon_2 \mathcal{B}_q\right) \, s_{\mathcal{K}} (x_1^*, \dots, x^*_{N_1}),
\end{equation*}
which, upon using (\ref{eq:kjbdjhsqkvhjezhjvf}), yields
\begin{equation}\label{eq:thisthedensityboundrgrs2}
N_2 
< \frac{\vol\left(\mathcal{E}_p^d + \varepsilon_2 \mathcal{B}_q\right)\, e^{-\rho}}{\vol\left(\mathcal{B}_q\right)} \, \varepsilon_2^{-d} .
\end{equation}

We now argue by contradiction that $\{x_1^*, \dots, x_{N_1}^*, y_1^*, \dots, y_{N_2}^*\}$ is an $\varepsilon$-covering of $\mathcal{E}_p^d$.
Assume that this is not the case.
Then, there would exist a point $y_{N_2+1}^* \in \mathcal{E}_p^d $ that is at distance further than $\varepsilon = \varepsilon_1+\varepsilon_2$ from all $x^*_i$ 
and further than $\varepsilon \geq 2\varepsilon_2$ from all $y^*_j$, 
thereby contradicting the maximality of $N_2$.

Next, using (\ref{eq:thisthedensityboundrgrs1}) and (\ref{eq:thisthedensityboundrgrs2}), we bound the covering number of $\mathcal{E}^d_p$ according to 
\begin{align}
N\left(\varepsilon\semcol \mathcal{E}^d_p, \|\cdot\|_q \right)\, \varepsilon^d
&\leq \left(N_1+N_2\right)\, \varepsilon^d\label{eq:jhgfdsfghjkjhgfty1}\\
&<  \frac{\vol\left(\mathcal{E}_p^d + \varepsilon\mathcal{B}_q\right)}{\vol\left(\mathcal{B}_q\right)}\left(\frac{\varepsilon}{\varepsilon_1} \right)^d\left(\rho + \frac{\vol\left(\mathcal{E}_p^d + \varepsilon_2\mathcal{B}_q\right)}{\vol\left(\mathcal{E}_p^d + \varepsilon\mathcal{B}_q\right)} \left(\frac{\varepsilon_1}{\varepsilon_2} \right)^d e^{-\rho}\right) \nonumber\\
&\leq  \frac{\vol\left(\mathcal{E}_p^d + \varepsilon\mathcal{B}_q\right)}{\vol\left(\mathcal{B}_q\right)}\left(\frac{\varepsilon}{\varepsilon_1} \right)^d\left(\rho +  \left(\frac{\varepsilon_1}{\varepsilon_2} \right)^d e^{-\rho}\right), \label{eq:jhgfdsfghjkjhgfty2}
\end{align}
where the last inequality is a direct consequence of $\varepsilon_2 \leq \varepsilon$.
Now, choosing $N_1$ according to
\begin{equation*}
N_1 = \left\lceil \frac{\vol\left(\mathcal{E}_p^d + \varepsilon\mathcal{B}_q\right)}{\varepsilon_1^d\,\vol\left(\mathcal{B}_q\right)} \, d \ln\left(\frac{\varepsilon_1}{\varepsilon_2}\right) \right \rceil,
\end{equation*}
we get
\begin{equation}\label{eq:mlkjhgfgjbkhbkhjbkhjbbjkszz}
d \ln\left(\frac{\varepsilon_1}{\varepsilon_2}\right)
\leq \rho 
\leq  d \ln\left(\frac{\varepsilon_1}{\varepsilon_2}\right) + \frac{\varepsilon_1^{d} \, \vol\left(\mathcal{B}_q\right)}{\vol\left(\mathcal{E}_p^d + \varepsilon\mathcal{B}_q\right)} .
\end{equation}
Inserting (\ref{eq:njhgfdfgfhqsfdjhvdhjze}) and (\ref{eq:mlkjhgfgjbkhbkhjbkhjbbjkszz}) into (\ref{eq:jhgfdsfghjkjhgfty1})--(\ref{eq:jhgfdsfghjkjhgfty2}), we obtain
\begin{equation}\label{eq:johhvoshpzijevjrnzvengrzfrr}
N\left(\varepsilon\semcol \mathcal{E}^d_p, \|\cdot\|_q \right)\, \varepsilon^d
< \frac{\vol\left(\mathcal{E}_p^d + \varepsilon\mathcal{B}_q\right)}{\vol\left(\mathcal{B}_q\right)}\left(\frac{\varepsilon}{\varepsilon_1} \right)^d\left(d \ln\left(d\right) + d\ln^{(2)}(d) +  1 \right) + \varepsilon^d.
\end{equation}
We next observe that 
\begin{equation*}
\frac{\varepsilon}{\varepsilon_1} 
= 1 + \frac{1}{d \, \ln(d)} 
= 1+  O_{d\to \infty}\left(\frac{\log d}{d}\right)
\end{equation*}
and
\begin{equation*}
\left(d \ln\left(d\right) + d\ln^{(2)}(d) +  1 \right)^{1/d}
= \exp\left\{O_{d\to \infty}\left(\frac{\log d}{d}\right)\right\}
= 1+O_{d\to \infty}\left(\frac{\log d}{d}\right),
\end{equation*}
so that (\ref{eq:johhvoshpzijevjrnzvengrzfrr}) can be rewritten as
\begin{equation}\label{eq:lpoiughfjhkjqhslczmkef}
\left(N\left(\varepsilon\semcol \mathcal{E}^d_p, \|\cdot\|_q \right)-1\right)^{1/d}\, \varepsilon
< \kappa(d)\left(\frac{\vol\left(\mathcal{E}_p^d + \varepsilon\mathcal{B}_q\right)}{\vol\left(\mathcal{B}_q\right)}\right)^{1/d},
\end{equation}
with $\{\kappa(d)\}_{d \in \mathbb{N}^*}$ a sequence satisfying 
\begin{equation*}
 \kappa(d)= 1+ O_{d\to \infty}\left(\frac{\log d}{d}\right).
\end{equation*}

We now distinguish two cases, namely $p\geq q$ and $p< q$.
For $p\geq q$, we first apply Hölder's inequality with parameters $p/q$ and $(1-q/p)^{-1}$ (where $p\geq q$ ensures that these parameters are, indeed, in $[1, \infty]$), to get 
\begin{equation*}
\|x\|_{q, \mu}^q
= \sum_{n=1}^d \left(\left \lvert \frac{x_n}{\mu_n}\right\rvert^q \cdot 1\right)
\leq \left(\sum_{n=1}^d \left\lvert \frac{x_n}{\mu_n}\right\rvert^{p}\right)^{q/p}\cdot \left(\sum_{n=1}^d 1 \right)^{1-{q/p}}
= d^{1-{q/p}} \|x\|_{p, \mu}^{q}, \quad \text{for all } x \in \mathbb{R}^d.
\end{equation*}
This allows us to conclude that 
\begin{equation*}
\mathcal{E}^d_p \subseteq d^{(1/q-1/p)} \mathcal{E}^d_q.
\end{equation*}
Together with assumption (\ref{eq:extraassumptionub2}) relevant to \ref{item:iifinitedim}, we hence obtain
\begin{align}
\mathcal{E}^d_p + \varepsilon \mathcal{B}_q
&\subseteq d^{(1/q-1/p)}  \mathcal{E}^d_q + \eta \, d^{(1/q-1/p)} \mu_d \mathcal{B}_q \label{eq:kjlhgfdqssaqezzaaza1}\\
&\subseteq d^{(1/q-1/p)}  \mathcal{E}^d_q + \eta \, d^{(1/q-1/p)}\mathcal{E}^d_q \nonumber \\
&= (1+ \eta ) \, d^{(1/q-1/p)} \mathcal{E}^d_q .\label{eq:kjlhgfdqssaqezzaaza2}
\end{align}
The upper bound (\ref{eq:lpoiughfjhkjqhslczmkef}) then becomes 
\begin{align*}
\left(N\left(\varepsilon\semcol \mathcal{E}^d_p, \|\cdot\|_q \right)-1\right)^{1/d}\, \varepsilon
&< \kappa(d)\, (1+ \eta ) \, d^{(1/q-1/p)}\left(\frac{\vol\left(\mathcal{E}_q^d\right)}{\vol\left(\mathcal{B}_q\right)}\right)^{1/d}\\
&=\kappa(d)\, (1+ \eta ) \, d^{(1/q-1/p)} \bar \mu_d,
\end{align*}
where in the second step we used \eqref{eq: volume ellipsoid}. This establishes the desired result in the case \ref{item:iifinitedim}.
For \ref{item:ifinitedim}, with $p \geq q$, we have the stronger assumption (\ref{eq:extraassumptionub1}), in which case the inclusion relation (\ref{eq:kjlhgfdqssaqezzaaza1})--(\ref{eq:kjlhgfdqssaqezzaaza2}) is replaced by 
\begin{equation*}
\mathcal{E}^d_p + \varepsilon \mathcal{B}_q
\subseteq  \mathcal{E}^d_p + \eta \,  \mu_d \mathcal{B}_q 
\subseteq  \mathcal{E}^d_p + \eta \, \mathcal{E}^d_p 
= (1+ \eta ) \, \mathcal{E}^d_p .
\end{equation*}
Correspondingly, the upper bound (\ref{eq:lpoiughfjhkjqhslczmkef}) becomes
\begin{align}
\left(N\left(\varepsilon\semcol \mathcal{E}^d_p, \|\cdot\|_q \right)-1\right)^{1/d}\, \varepsilon
&< \kappa(d)\, (1+ \eta ) \,\left(\frac{\vol\left(\mathcal{E}_p^d\right)}{\vol\left(\mathcal{B}_q\right)}\right)^{1/d} \label{eq:kjlhgfdqssaqezzaaza11}\\
&=\kappa(d)\, (1+ \eta ) \,V_{p, q, d}\,  \bar \mu_d. \label{eq:kjlhgfdqssaqezzaaza12}
\end{align}

To complete the proof of \ref{item:ifinitedim}, we need to consider the case $p < q$.
Again applying Hölder's inequality, we can conclude that 
\begin{equation*}
\mathcal{B}_q
\subseteq d^{(1/p-1/q)} \mathcal{B}_p,
\end{equation*}
which, when used together with assumption (\ref{eq:extraassumptionub1}), yields
\begin{align*}
\mathcal{E}^d_p + \varepsilon \mathcal{B}_q
&\subseteq  \mathcal{E}^d_p + \eta \, d^{(1/q-1/p)} \mu_d \mathcal{B}_q\\
&\subseteq  \mathcal{E}^d_p + \eta \, \mathcal{E}^d_p 
= (1+ \eta ) \, \mathcal{E}^d_p .
\end{align*}
The upper bound (\ref{eq:lpoiughfjhkjqhslczmkef}) then becomes 
\begin{align*}
\left(N\left(\varepsilon\semcol \mathcal{E}^d_p, \|\cdot\|_q \right)-1\right)^{1/d}\, \varepsilon
&< \kappa(d)\, (1+ \eta ) \left(\frac{\vol\left(\mathcal{E}_p^d\right)}{\vol\left(\mathcal{B}_q\right)}\right)^{1/d} \\
&=\kappa(d)\, (1+ \eta ) \, V_{p, q, d} \,  \bar \mu_d,
\end{align*}
which, 
together with (\ref{eq:kjlhgfdqssaqezzaaza11})--(\ref{eq:kjlhgfdqssaqezzaaza12}), establishes \ref{item:ifinitedim} for all $p,q \in [1,\infty]$.

\subsection{Proof of Theorem~\ref{thm:bound on mixed norm}}\label{sec:proofofthmmixednorm}

Fix $\varepsilon \in (0, \mu_1)$ and $\gamma \geq 1$. 
We introduce the finite lattice 
\begin{equation}\label{eq:1326}
\Omega\coloneqq 
\left\{\omega \in \mathbb{R}_+^{k+1} \mid \| \omega \|_{2} \leq 1+ \frac{\sqrt{k+1}}{\jrdn^{\,\gamma}} \text{ and } \jrdn^{\,2\gamma} \omega_j^2 \in \mathbb{N}, \text{ for all } j \in \{1, \dots, k+1\}\right\},
\end{equation}
where $\jrdn$ is as defined in \eqref{eq:defdesbarres}, and note that
\begin{equation}\label{eq:1325}
\# \Omega \leq 
K \jrdn^{\,2\gamma(k+1)},
\quad \text{for some } K>0.
\end{equation}
With a view towards application of Lemma~\ref{lem:lemboundtriangleineqlike}, we first prove that condition (\ref{eq:firstdecompcartprodelli}) is satisfied with $\Omega$ according to \eqref{eq:1326} and 
\begin{equation*}
  \mathcal{E}_2^{[j]} = \mathcal{B} (0 \semcol \mu_j),
  \quad \text{for }\ j=1, \dots, k, \ \text{ and} \quad
  \mathcal{E}_2^{[k+1]} = \prod_{j=k+1}^\infty   \mathcal{B} (0 \semcol \mu_j).
\end{equation*}

\begin{lemma}\label{lem: Mixed ellipsoid as a product of balls}
Let $\mathcal{E}_{2, 2}$ be a mixed ellipsoid with semi-axes $\{\mu_n\}_{n\in\mathbb{N}^*}$ and dimensions $\{d_n\}_{n\in\mathbb{N}^*}$, and let $\Omega$ be as defined in (\ref{eq:1326}).
Then, we have 
\begin{equation}
{\mathcal{E}}_{2, 2}
\subseteq \bigcup_{\omega\in\Omega} \left(\prod_{j=1}^k \omega_j \,  \mathcal{B} (0 \semcol \mu_j)\right) \times \left(\omega_{k+1} \,\prod_{j=k+1}^\infty   \mathcal{B} (0 \semcol \mu_j)\right).
\end{equation}
\end{lemma}

\vspace{-3mm}

\begin{proof}[Proof.]
See Appendix \ref{sec:proofMixed ellipsoid as a product of balls}.
\end{proof}

\noindent
Next, set $\rho_j = \varepsilon$, for $j=1, \dots, k$, so that, for all $\omega \in \Omega$, we have
\begin{equation*}
  \rho_\omega 
  \coloneqq \left((\omega_1\rho_1)^2+\dots +(\omega_k\rho_k)^2 + (\omega_{k+1}\mu_{k+1})^2 \right)^{1/2}
  \leq \varepsilon \, \|\omega\|_2 
  \leq \varepsilon  \left(1+ \jrdn^{-\gamma} \sqrt{k+1}\right).
\end{equation*}
Thanks to Lemma~\ref{lem: Mixed ellipsoid as a product of balls}, we can now apply Lemma~\ref{lem:lemboundtriangleineqlike} to conclude that 
 \begin{equation}\label{eq:lknqzkjnakrvfkae}
N \left(\varepsilon\left(1+ \jrdn^{-\gamma} \sqrt{k+1}\right) \semcol {\mathcal{E}}_{2, 2}, \|\cdot\|_2 \right)
\leq \left(\# \Omega\right)\prod_{j=1}^k N \left(\rho_j \semcol \mathcal{B} (0 \semcol \mu_j), \|\cdot\|_2 \right).
 \end{equation}
We next apply Lemma~\ref{lem:rogernb2} to upper-bound each of the covering numbers on the right-hand side of (\ref{eq:lknqzkjnakrvfkae}) 
(recall that we assumed $d_j\geq 9$, for all $j=1, \dots, k$, and that $k$ is chosen such that $\varepsilon < \mu_k \leq \mu_j$)
according to 
\begin{equation*}
N \left(\rho_j \semcol \mathcal{B} (0 \semcol \mu_j), \|\cdot\|_2 \right)
\leq \left(\frac{\kappa(d_j) \,  \mu_j}{\rho_j}\right)^{d_j}, \quad \text{for } j\in \{1, \dots, k\},
\end{equation*}
with $\kappa(d_j)$ satisfying
\begin{equation}\label{eq:bound-kappadj}
 \kappa(d_j) \leq  1+ \frac{C\log (d_j)}{d_j},
\end{equation} 
for a constant $C>0$.
The bound (\ref{eq:lknqzkjnakrvfkae}) can hence be rewritten as
 \begin{equation*}
N \left(\varepsilon\left(1+ \jrdn^{-\gamma} \sqrt{k+1}\right) \semcol  {\mathcal{E}}_{2, 2}, \|\cdot\|_2 \right) \,  \varepsilon^{\jrdn}
\leq \left(\# \Omega\right)\prod_{j=1}^k (\kappa(d_j) \mu_j)^{d_j},
 \end{equation*}
 which, in turn, leads to
\begin{equation*}
      N \left(\varepsilon \left(1+ \jrdn^{-\gamma} \sqrt{k+1}\right) \semcol {\mathcal{E}}_{2, 2}, \|\cdot\|_2 \right)^{1/\jrdn} \varepsilon_\gamma
  \leq \frac{\varepsilon_\gamma \left(\prod_{j=1}^k \kappa(d_j)^{d_j/\jrdn} \right)\left(\# \Omega\right)^{1/\jrdn} }{\varepsilon} \prod_{j=1}^k  \mu_j^{d_j/\jrdn}.
 \end{equation*}
 Now, we observe  that 
  \begin{equation}\label{eq:scaling-pt-1}
\log\left(\prod_{j=1}^k \kappa(d_j)^{d_j/\jrdn}\right)
 \leq  \sum_{j=1}^k \frac{d_j}{\jrdn} \log\left(1+\frac{C \log (d_j)}{d_j}\right)
 \leq   \frac{C\, k \, \log (\jrdn)}{\ln(2)\,\jrdn},
 \end{equation}
with $C$ the constant from \eqref{eq:bound-kappadj}.}
We further note that
\begin{equation}\label{eq:scaling-pt-2}
\varepsilon_\gamma = \varepsilon \left(1+ O_{k \to \infty} \left(\frac{\sqrt{k}}{\jrdn^{\,\gamma}}\right)\right)
\quad \text{and} \quad 
\log\left(\left(\# \Omega\right)^{1/\jrdn}\right) =  O_{k \to \infty} \left(\frac{\gamma \, k \log\left(\jrdn\right)}{\jrdn}\right),
\end{equation}
where the latter is a consequence of (\ref{eq:1325}).
This allows us to conclude that 
\begin{equation*}
 N \left(\varepsilon_\gamma \semcol {\mathcal{E}}_{2, 2}, \|\cdot\|_2 \right)^{1/\jrdn} \, \varepsilon_\gamma
\leq \kappa_k \prod_{i=1}^k \mu_j^{d_j/\jrdn},
\end{equation*}
where 
%we verify, using \eqref{eq:scaling-pt-1} and \eqref{eq:scaling-pt-2}, that}
\begin{equation*}
\kappa_k = \frac{\varepsilon_\gamma \left(\prod_{j=1}^k \kappa(d_j)^{d_j/\jrdn} \right)\left(\# \Omega\right)^{1/\jrdn} }{\varepsilon}
    % = 1+ O_{k \to \infty} \left(\frac{\gamma \, k \, \log (\jrdn)}{\jrdn}\right),
\end{equation*}
can be shown to satisfy \eqref{eq:jsqhkqoioiloio} thanks to \eqref{eq:scaling-pt-1} and \eqref{eq:scaling-pt-2}.
%to 
%satisfies \eqref{eq:jsqhkqoioiloio},
%thereby finishing the proof.

\subsection{Proof of Theorem~\ref{thm: scaling metric entropy infinite ellipsoids regular}}\label{sec:proofofmainthmisinthissec}

We start by introducing the notions of effective dimension.

\begin{definition}[Upper and lower effective dimension]\label{def:uploweffdim}
Let $p, q \in [1, \infty]$, $s\geq 1$, $\varepsilon>0$, and let $\mathcal{E}_p$ be a $p$-ellipsoid with semi-axes $\{\mu_n\}_{n\in \mathbb{N}^*}$.
We define the \emph{upper effective dimension} of $\mathcal{E}_p$ as
\begin{equation}\label{eq: definition of underbar n pt1mg3}
\ueffdim_s
\coloneqq \min \left\{ k \in \mathbb{N}^* \mid N \left(\varepsilon\semcol \mathcal{E}_p, \|\cdot\|_q\right) \, \mu_k^k \leq s^k  \, \bar \mu_k^k\right\},
\end{equation}
and the \emph{lower effective dimension} according to
\begin{equation}\label{eq: definition of bar n pt1mg3fst}
\leffdim_s 
\coloneqq \max \left\{ k \in \mathbb{N}^* \mid N \left(\varepsilon\semcol \mathcal{E}_p, \|\cdot\|_q\right) \,  \mu_k^k  > s^k  \, \bar \mu_k^k  \right\}.
\end{equation}
\end{definition}

\noindent
For $s=1$, we simply write $\ueffdim$ and $\leffdim$.
We emphasize that, even though the notation does not reflect it, the quantities $\ueffdim_s$ and $\leffdim_s$ depend on $\varepsilon$, $p$, $q$, and the semi-axes $\{\mu_n\}_{n\in \mathbb{N}^*}$.
The following lemma states properties of $\ueffdim_s$ and $\leffdim_s$ that will be used subsequently.

\begin{lemma}\label{lem:propertieseffdim}
Let $p, q \in [1, \infty]$, $s\geq 1$, $\varepsilon>0$, and let $\mathcal{E}_p$ be a $p$-ellipsoid with non-increasing semi-axes $\{\mu_n\}_{n\in \mathbb{N}^*}$. 
Then, the upper and the lower effective dimensions $\ueffdim_s$ and $\leffdim_s$ of $\mathcal{E}_p$ are well-defined and satisfy
\begin{equation*}
  \ueffdim_s = \leffdim_s +1 
  \quad \text{and} \quad 
  \lim_{\varepsilon \to 0} \ueffdim_s = \infty.
\end{equation*}
\end{lemma}

\begin{proof}[Proof.]
  See Appendix \ref{sec:prooflempropertieseffdim}.
\end{proof}

\noindent 
The notion of effective dimension allows to relate the behavior of the semi-axes of an ellipsoid to its metric entropy.
This is the content of the following lemma.

\begin{lemma}\label{lem:1328}
Let $p, q \in [1, \infty]$,  $s\geq 1$, and $b>0$. 
Let $\{\mu_n\}_{n\in \mathbb{N}^*}$ be a non-increasing sequence that is regularly varying with index $-b$, and 
let $\phi \colon \mathbb{R}^*_+ \to \mathbb{R}^*_+$ be the step function given by 
$\phi(t) = \mu_{\lfloor t \rfloor + 1}$, for $t > 0$. Let $\mathcal{E}_p$ be the $p$-ellipsoid with semi-axes $\{\mu_n\}_{n\in \mathbb{N}^*}$.
  Then, we have 
\begin{equation*}
  \lim_{\varepsilon \to 0} \frac{\mu_{\leffdim_s}}{\phi \left(H \left(\varepsilon\semcol \mathcal{E}_p, \|\cdot\|_q \right)\right)}
  =\lim_{\varepsilon \to 0} \frac{\mu_{\ueffdim_s}}{\phi \left(H \left(\varepsilon\semcol \mathcal{E}_p, \|\cdot\|_q \right)\right)}
  =\left(\frac{b}{\ln(2)} +\log(s)\right)^b,
\end{equation*}
where $\leffdim_s$ and $\ueffdim_s$ are, respectively, the lower and upper effective dimension of $\mathcal{E}_p$.
\end{lemma}

\begin{proof}[Proof.]
  The proof is provided in Appendix \ref{sec:proofoflemma1335}.
\end{proof}

\noindent
With a proof similar to that of Lemma~\ref{lem:1328}, we get 
\begin{equation}\label{eq:mlkojihgdqskjcdknazmqcfkjnjfdezq}
\lim_{\varepsilon \to 0} \frac{H \left(\varepsilon\semcol \mathcal{E}_p, \|\cdot\|_q \right)}{{\ueffdim_s}}
=\lim _{\varepsilon \to 0} \frac{H \left(\varepsilon\semcol \mathcal{E}_p, \|\cdot\|_q \right)}{{\leffdim_s}}
= \frac{b}{\ln(2)} +\log(s), \quad \text{for } s \geq 1.
\end{equation}
We remark that (\ref{eq:mlkojihgdqskjcdknazmqcfkjnjfdezq}) formalizes the heuristic (\ref{eq:heuristicseffdim2}), namely that
for $\varepsilon \rightarrow 0$ the metric entropy of the ellipsoid $\mathcal{E}_{p}$, indeed, behaves like that of its truncated
$\ueffdim_s$-dimensional version. Recall that, owing to Lemma~\ref{lem:propertieseffdim}, for $\varepsilon \rightarrow 0$, we have $\leffdim_s \simeq \ueffdim_s$.

The next four lemmata formalize the heuristic (\ref{eq:heuristicseffdim1}).

\begin{lemma}\label{lem:1327}
Let $p, q \in [1, \infty]$, $s\geq 1$, $\varepsilon>0$, and let $\mathcal{E}_p$ be a $p$-ellipsoid with semi-axes $\{\mu_n\}_{n\in \mathbb{N}^*}$.
Let $\ueffdim_s$ be the upper effective dimension of $\mathcal{E}_p$.
Then, %we have 
\begin{equation*}
s \varepsilon 
\geq V_{p,q, \ueffdim_s}\mu_{\ueffdim_s}.
\end{equation*}
\end{lemma}

\begin{proof}[Proof.]
See Appendix \ref{sec:proofintermediary bound1337}.
\end{proof}

\begin{lemma}\label{lem:klmjhgfddddfhqsdfg}
Let $b > 0$, let $p, q \in [1, \infty]$ be such that  $p/(pb+1) < q \leq p$, and let $\{\mu_n\}_{n\in \mathbb{N}^*}$ be a non-increasing sequence that is regularly varying with index $-b$. Let $\mathcal{E}_p$ be the $p$-ellipsoid with semi-axes $\{\mu_n\}_{n\in \mathbb{N}^*}$.
Then, as $\varepsilon \rightarrow 0$, there exists an $s > 1$ such that
\begin{equation*}
 s = 2 + o_{\varepsilon \to 0}(1) 
  \quad \text{and} \quad
   \varepsilon \leq \gamma_{p,q,b} \,  \leffdim_{s}^{(1/q-1/p)}\, \mu_{\leffdim_{s}}\left(1+ o_{\varepsilon \to 0}(1)\right),
\end{equation*}
where $\leffdim_s$ is the lower effective dimension of $\mathcal{E}_p$ and
\begin{equation*}
\gamma_{p,q,b}
\coloneqq  \left(\frac{b}{b+\frac{1}{p}-\frac{1}{q}}   \right)^{(\frac{1}{q}-\frac{1}{p})}.
%}
\end{equation*}
\end{lemma}

\begin{proof}[Proof.]
See Appendix \ref{sec:prooflemklmjhgfddddfhqsdfg}.
\end{proof}

\begin{lemma}\label{lem:klmjhgfddddfhqsdfgEEE}
Let $b > 0$, let $p, q \in [1, \infty]$ be such that $q=p/(pb+1)$, and let
$\{\mu_n\}_{n\in \mathbb{N}^*}$ be a non-increasing sequence that is regularly varying with index $-b$ and satisfies $\sum_{n\in\mathbb{N}^*}\mu_n^{1/b}<\infty$.
Let $\mathcal{E}_p$ be the $p$-ellipsoid with semi-axes $\{\mu_n\}_{n\in \mathbb{N}^*}$.
Then, for $\varepsilon \rightarrow 0$, there exists an $s > 1$ such that
\begin{equation*}
 s = 2 + o_{\varepsilon \to 0}(1) 
  \quad \text{and} \quad
   \varepsilon 
   \leq \left(\sum_{n=\leffdim_s+1}^\infty \mu_n^{1/b}\right)^b
   \left(1+ o_{\varepsilon \to 0}(1)\right),
\end{equation*}
where $\leffdim_s$ is the lower effective dimension of $\mathcal{E}_p$.
\end{lemma}%}

\begin{proof}[Proof.]
See Appendix \ref{sec:prooflemklmjhgfddddfhqsdfgEEE}.
\end{proof}

\noindent
We proceed with the proof of Theorem~\ref{thm: scaling metric entropy infinite ellipsoids regular}.

\paragraph{The case (a) $q< p/(pb+1)$ or (b) $q = p/(pb+1)$ with $\lim_{n\to\infty}n \mu_n^{1/b}>0$.}
Application of the lower bound (\ref{eq:lkjbkjbkjbckjbkebzhzzzqdfguyuy}) in Theorem~\ref{thm:profinitethm} yields
\begin{equation}\label{eq:kjhhknpolplkkjnn0}
N\left(\varepsilon\semcol \mathcal{E}_p, \|\cdot\|_q \right)
\geq N\left(\varepsilon\semcol \mathcal{E}^d_p, \|\cdot\|_q \right)
\geq  \frac{V_{p, q, d}^d \, \bar \mu_d^d}{\varepsilon^d}, 
 \quad \text{for all } \varepsilon \in (0, 1) \text{ and all } d \in \mathbb{N}^*.
\end{equation}
Further, as the semi-axes are non-increasing, we get 
\begin{equation}\label{eq:kjhhknpolplkkjnn0bb}
N\left(\varepsilon\semcol \mathcal{E}_p, \|\cdot\|_q \right)^{1/d}
\geq  \frac{V_{p, q, d} \, \mu_d}{\varepsilon}, 
 \quad \text{for all } \varepsilon \in (0, 1) \text{ and all } d \in \mathbb{N}^*.
\end{equation}
Thanks to (\ref{eq: ratio volumes ksdjnazc}), the right-hand side of (\ref{eq:kjhhknpolplkkjnn0bb}) scales according to 
\begin{equation}\label{eq:V-lower-bound}
    \frac{V_{p, q, d} \, \mu_d}{\varepsilon}
    = \frac{\Gamma_{p, q} \, d^{(1/q-1/p)} \, \mu_d}{\varepsilon} \left(1+O_{d\to\infty}\left(\frac{1}{d}\right)\right),
\end{equation}
for all $\varepsilon \in (0, 1)$.

Let us first consider the case $q = p/(pb+1)$ with $ \lim_{n\to\infty}n \mu_n^{1/b}>0 $ and note that, owing to
$(1/q-1/p)=b$, we have $\lim_{d\to\infty} d^{(1/q-1/p)} \, \mu_d>0$.
In particular, there exists $\varepsilon^*>0$ such that 
\begin{equation*}
    \lim_{d\to\infty} \frac{\Gamma_{p, q} \, d^{(1/q-1/p)} \, \mu_d}{\varepsilon} > 1,
    \quad \text{for all } \varepsilon \in (0, \varepsilon^*).
\end{equation*}
From (\ref{eq:kjhhknpolplkkjnn0bb}), we then obtain 
\begin{equation*}
N\left(\varepsilon\semcol \mathcal{E}_p, \|\cdot\|_q \right)
\geq \lim_{d\to\infty}   \left(\frac{\Gamma_{p, q} \, d^{(1/q-1/p)} \, \mu_d}{\varepsilon}\right)^d = \infty, 
 \quad \text{for all } \varepsilon \in (0, \min\{1, \varepsilon^*\}),
\end{equation*}
which is the desired result.

Now, for $q< p/(pb+1)$,
using 
the fact that $\{\mu_n\}_{n\in \mathbb{N}^*}$ is regularly varying with index $-b$ together with \cite[Theorem~1.9.5 case (ii) and Theorem~1.4.1]{binghamRegularVariation1987}, we can conclude the existence of a slowly varying function $\ell(\cdot)$ such that
\begin{equation}\label{eq:qsvknlkvsenkdnfgrehjbstdjgr}
    \frac{\Gamma_{p, q} \, d^{(1/q-1/p)} \, \mu_d}{\varepsilon} 
    =\frac{\Gamma_{p, q} \, d^{(1/q-1/p-b)} \, \ell(d)}{\varepsilon} , \quad \text{for all } \varepsilon \in (0, 1) \text{ and all } d \in \mathbb{N}^*.
\end{equation}
As 
$q < p/(pb+1)$ implies $1/q-1/p-b > 0$,
it follows from \cite[Proposition~1.3.6 case (v)]{binghamRegularVariation1987} that the right-hand side of (\ref{eq:qsvknlkvsenkdnfgrehjbstdjgr}) goes to infinity for $d\to\infty$, which, together with (\ref{eq:kjhhknpolplkkjnn0bb}) and \eqref{eq:V-lower-bound}, yields the desired result.

\paragraph{The case $q=p/(pb+1)$ and $\sum_{n\in\mathbb{N}^*}\mu_n^{1/b}<\infty$.}

%\textcolor{blue}{
Application of Lemma~\ref{lem:klmjhgfddddfhqsdfgEEE} allows us to conclude the existence of an $s$ such that $s = 2 + o_{\varepsilon \to 0}(1)$ and
\begin{equation}\label{eq:kjqfhvkjjkcjjjjjjjjjnsxn2}
    \lim_{\varepsilon \to 0}\frac{\varepsilon}{\psi\left(\leffdim_{s}\right)} \leq 1,
\end{equation}
where $\leffdim_s$ is the lower effective dimension of $\mathcal{E}_p$.
Moreover, 
by definition of $\phi$ and $\psi$, 
one gets 
\begin{equation}\label{eq:jkdfsnkjskjkknfvkjkfgbhbjkbkbh}
    \psi(x) 
    \sim_{x \to \infty}     
    %\stackrel{x\to\infty}{\sim} 
    \left(\int_{ x }^\infty\phi(t)^{1/b}dt\right)^b 
    = \left(\int_{ x }^\infty\frac{\ell(t)}{t}dt\right)^b,
\end{equation}
where we introduced the slowly varying function $\ell\colon t\mapsto t \, \phi(t)^{1/b}$.
Now, applying \cite[Proposition 1.5.9b]{binghamRegularVariation1987}, we can conclude that the right-most term in (\ref{eq:jkdfsnkjskjkknfvkjkfgbhbjkbkbh}) is slowly varying and therefore so is $\psi$.
As $\psi$ is also non-increasing, one can apply Lemma~\ref{lme:1335199} (cf. the argument in the proof of Lemma~\ref{lem:1328} in Appendix~\ref{sec:proofoflemma1335}) together with the limit in (\ref{eq:mlkojihgdqskjcdknazmqcfkjnjfdezq}) to obtain
\begin{equation}\label{eq:kjqfhvkjjkcjjjjjjjjjnsxn}
    \lim_{\varepsilon \to 0}\frac{\varepsilon}{\psi\left(H \left(\varepsilon\semcol \mathcal{E}_p, \|\cdot\|_q \right)\right)}
     = \lim_{\varepsilon \to 0} \frac{\psi\left(\leffdim_{s}\right)}{\psi\left(\leffdim_{s} \cdot H \left(\varepsilon\semcol \mathcal{E}_p, \|\cdot\|_q \right)/\leffdim_{s}\right)}
    \cdot \frac{\varepsilon}{\psi\left(\leffdim_{s}\right)}
    = \lim_{\varepsilon \to 0}\frac{\varepsilon}{\psi\left(\leffdim_{s}\right)}.
\end{equation}
The second statement in (ii) of Theorem~\ref{thm: scaling metric entropy infinite ellipsoids regular} finally follows by combining (\ref{eq:kjqfhvkjjkcjjjjjjjjjnsxn2}) and (\ref{eq:kjqfhvkjjkcjjjjjjjjjnsxn}). 

The first statement is established by employing Lemma~\ref{lem:1327} with $s=1$ combined with 
Lemma~\ref{lme:1335199} and arguments thereafter, using $\phi(\ueffdim) =\mu_{\ueffdim+1}$, (\ref{eq: ratio volumes ksdjnazc}), and
$b=1/q-1/p$, to get
\begin{align*}
    \lim_{\varepsilon \to 0} \varepsilon \, \frac{H \left(\varepsilon\semcol \mathcal{E}_p, \|\cdot\|_q \right)^{\left(\frac{1}{p}-\frac{1}{q}\right)}}{\phi \left(H \left(\varepsilon\semcol \mathcal{E}_p, \|\cdot\|_q \right) \right)}
    =
    \lim_{\varepsilon \to 0} \varepsilon \frac{\ueffdim^{-b}}{\phi \left(\ueffdim \right)}
    \geq \lim_{\varepsilon \to 0}  \, \Gamma_{p,q} \ueffdim^{1/q-1/p} \,  \ueffdim^{-b} \frac{\mu_{\ueffdim}}{\phi \left(\ueffdim \right)}
    = \Gamma_{p,q}.
\end{align*}

\paragraph{The case $p/(pb+1) < q \leq p$.}
Using Lemma~\ref{lem:klmjhgfddddfhqsdfg}, we obtain
\begin{align*}
\varepsilon \frac{H \left(\varepsilon\semcol \mathcal{E}_p, \|\cdot\|_q \right)^{\left(\frac{1}{p}-\frac{1}{q}\right)}}{\phi \left(H \left(\varepsilon\semcol \mathcal{E}_p, \|\cdot\|_q \right) \right)}
 \leq  \frac{\gamma_{p,q,b}\, \mu_{\leffdim_s}}{\phi \left(H \left(\varepsilon\semcol \mathcal{E}_p, \|\cdot\|_q \right)\right)} 
\left(\frac{H \left(\varepsilon\semcol \mathcal{E}_p, \|\cdot\|_q \right)}{\leffdim_{s}} \right)^{(\frac{1}{p}-\frac{1}{q})}
\left(1+ o_{\varepsilon \to 0}(1)\right), 
\end{align*}
with $s$ satisfying
\begin{equation*}
s = 2 + o_{\varepsilon \to 0}(1).
\end{equation*}
Application of Lemma~\ref{lem:1328} together with (\ref{eq:mlkojihgdqskjcdknazmqcfkjnjfdezq}), allows us to conclude that
\begin{equation*}
\lim_{\varepsilon \to 0} \varepsilon \frac{H \left(\varepsilon\semcol \mathcal{E}_p, \|\cdot\|_q \right)^{\left(\frac{1}{p}-\frac{1}{q}\right)}}{\phi \left(H \left(\varepsilon\semcol \mathcal{E}_p, \|\cdot\|_q \right) \right)}
\leq \gamma_{p,q,b} \left(\frac{b}{\ln(2)}  + 1 \right)^{\left(b+\frac{1}{p}-\frac{1}{q}\right)},
\end{equation*}
which is the desired upper bound.
The lower bound is established analogously by employing Lemma~\ref{lem:1327} with $s=1$ in place of Lemma~\ref{lem:klmjhgfddddfhqsdfg}.
Specifically, we obtain
 \begin{equation*}
\lim_{\varepsilon \to 0} \varepsilon \frac{H \left(\varepsilon\semcol \mathcal{E}_p, \|\cdot\|_q \right)^{\left(\frac{1}{p}-\frac{1}{q}\right)}}{\phi \left(H \left(\varepsilon\semcol \mathcal{E}_p, \|\cdot\|_q \right) \right)}
\geq \Gamma_{p,q} \left(\frac{b}{\ln(2)} \right)^{\left(b+\frac{1}{p}-\frac{1}{q}\right)},
\end{equation*}
where we have used (\ref{eq: ratio volumes ksdjnazc}).

\paragraph{The case $p < q$.}
The lower bound (\ref{eq:jdhhjnjbyugssfkl}) is proved analogously to the lower bound in the case $p/(pb+1)<q \leq p$.
To establish the upper bound, 
we introduce the quantities $\newldeff, \newudeff \in \mathbb{N}^{*}$ satisfying
\begin{equation}\label{eq:njksdbvkjlbbgsazz}
(\newldeff+1)^{(1/q-1/p)} \mu_{\newldeff+1} \leq \varepsilon < \newldeff^{(1/q-1/p)} \mu_{\newldeff}
\quad \text{and} \quad 
\mu_{\newudeff+1} \leq \varepsilon < \mu_{\newudeff}.
\end{equation}
Note that, owing to  the semi-axes $\{\mu_n\}_{n\in \mathbb{N}^*}$ being regularly varying with index $-b$, we have the asymptotic equivalences 
\begin{equation} \label{eq:asymptotic-equivalences}
  (1/p-1/q+b)\log(\newldeff) 
  \sim_{\varepsilon\to 0} b \log(\newudeff)
  \sim_{\varepsilon\to 0} \log(\varepsilon^{-1}).
\end{equation}%}
We further let
\begin{equation}\label{eq:parameters-for-decomposition}
k \coloneqq \left\lfloor \log^{\frac{1}{1+b(({1/p-1/q})^{-1}-1)}}\left(\varepsilon^{-1}\right)\right\rfloor, \ \ 
d_1 = \dots = d_{k-1} \coloneqq \newldeff,
\text{ and } d_k \coloneqq \newudeff - (k-1) \newldeff.
\end{equation}
As we are interested in the regime $\varepsilon \rightarrow 0$, we can assume that $\varepsilon$ is small enough for $k \geq 2$ and $d_{k} \geq 1$ to be ensured.
We now decompose $\mathcal{E}_{p}$ according to Definition~\ref{def:blockdecomp} with the parameters in \eqref{eq:parameters-for-decomposition}
and employ case (i) in Lemma~\ref{lem:lemboundtriangleineqlike} 
to get
\begin{equation}\label{eq:joiefhboznrkpjbers}
H \left(\varepsilon\semcol \mathcal{E}_p, \|\cdot\|_q \right)  
\leq \sum_{j=1}^{k} H \left(\varepsilon\semcol \mathcal{E}_p^{[j]}, \|\cdot\|_q \right) + O_{\varepsilon \to 0}\left(k \log(\newldeff)\right).
\end{equation}
The detailed arguments leading to \eqref{eq:joiefhboznrkpjbers} are skipped as the underlying technique has already been used repeatedly. We only remark that the basis for the application of Lemma~\ref{lem:lemboundtriangleineqlike} is a
lattice construction akin to that in the proof of Theorem~\ref{thm:bound on mixed norm} together with the insight that, owing to \eqref{eq:heuristicseffdim1},
$\newldeff$ plays the role of effective dimension. 
%\textcolor{blue}{
Each summand on the right-hand side of (\ref{eq:joiefhboznrkpjbers}) can now be upper-bounded individually by invoking case \ref{item:ifinitedim} in Theorem~\ref{thm:profinitethm} to get
\begin{equation}\label{eq:lllllldfejzbjre1}
\left(N \left(\varepsilon\semcol \mathcal{E}_p^{[j]}, \|\cdot\|_q\right) -1 \right)^{1/\newldeff}
\leq \kappa(\newldeff)\left(1+\frac{\varepsilon}{\newldeff^{(1/q-1/p)} \mu_{j\newldeff}} \right) \frac{ V_{p,q,\newldeff}\,  \mu_{1+(j-1) \newldeff}}{\varepsilon},
\end{equation}
for $j\in\{1, \dots, k-1\}$.
Further, using (\ref{eq:njksdbvkjlbbgsazz}) together with the asymptotics (\ref{eq: ratio volumes ksdjnazc})% in Appendix~\ref{NotationandTerminology}
, we get that 
\begin{equation}\label{eq:lllllldfejzbjre2}
\kappa(\newldeff)\left(1+\frac{\varepsilon}{\newldeff^{(1/q-1/p)} \mu_{j\newldeff}} \right) \frac{ V_{p,q,\newldeff}\,  \mu_{1+(j-1) \newldeff}}{\varepsilon}
\sim_{\newldeff\to\infty} \kappa(\newldeff) \, \Gamma_{p,q}\left(1 + \frac{\mu_{j \newldeff}}{\mu_{\newldeff}}\right)\frac{\mu_{1+(j-1) \newldeff}}{\mu_{j \newldeff}},
\end{equation}
for $j\in\{1, \dots, k-1\}$.
For $j=k$, we note that the ellipsoid $ \mathcal{E}_p^{[k]}$ is contained in $\mu_{\bar d_{k-1}} \mathcal{B}_p^{d_k}$, the $d_k$-dimensional $p$-ball of radius $\mu_{\bar d_{k-1}}$. This allows us to conclude that
\begin{equation}\label{eq:jksdhsjdhjfdbdshjjhjjhggqshvdgsfjdhuyuy}
    H \left(\varepsilon\semcol \mathcal{E}_p^{[k]}, \|\cdot\|_q \right)
    \leq H \left(\varepsilon\semcol \mu_{\bar d_{k-1}} \mathcal{B}_p^{d_k}, \|\cdot\|_q \right).
\end{equation}%}
%\textcolor{blue}{
Let us now set $\eta = {\varepsilon}/({d_k^{(1/q-1/p)}\mu_{\bar d_{k-1}}})$ and note that this choice, thanks to $p<q$ and the asymptotic equivalences \eqref{eq:asymptotic-equivalences}, meets condition \eqref{eq:extraassumptionub1}. We can therefore apply 
\ref{item:ifinitedim} in Theorem~\ref{thm:profinitethm} to $\mu_{\bar d_{k-1}} \mathcal{B}_p^{d_k}$, viewed as a $d_k$-dimensional $p$-ellipsoid with all semi-axes equal to $\mu_{\bar d_{k-1}}$,
to get
\begin{equation}\label{eq:ball-as-elipsoid}
   \left(N \left(\varepsilon\semcol \mu_{\bar d_{k-1}} \mathcal{B}_p^{d_k}, \|\cdot\|_q \right)-1\right)^{1/d_k} \varepsilon  
   \leq C\left(1+\frac{\varepsilon}{d_k^{(1/q-1/p)}\mu_{\bar d_{k-1}}}\right) V_{p, q,d_k}  \mu_{\bar d_{k-1}}\leq C_{p,q} \, \varepsilon.
\end{equation}
Here, $C$ is a universal constant, $C_{p,q}$ is a constant that depends on $p$ and $q$ only, and for the second inequality (valid for sufficiently small $\varepsilon$) we used that, again thanks to $p<q$ and the asymptotic equivalences \eqref{eq:asymptotic-equivalences}, $\eta$ goes to infinity as $\varepsilon \to 0$. We can further simplify \eqref{eq:ball-as-elipsoid} to
\begin{equation*}
  N \left(\varepsilon\semcol \mu_{\bar d_{k-1}} \mathcal{B}_p^{d_k}, \|\cdot\|_q \right)^{1/d_k}   
  \leq C_{p,q}',
\end{equation*}
where $C_{p,q}'$ is a constant that depends on $p$ and $q$ only and can, without loss of generality, assumed to be greater than $2$.
In particular, upon defining the dimension $\mathfrak{d}' \coloneqq \lceil d_k\log\left(C_{p,q}'\right)\rceil \geq d_k$, we obtain
\begin{equation}\label{eq:hivjehbvhjfdh}
    H \left(\varepsilon\semcol \mu_{\bar d_{k-1}} \mathcal{B}_p^{d_k}, \|\cdot\|_q \right)
    \leq \mathfrak{d}'.
\end{equation}
%\textcolor{blue}{
We proceed by applying \cite[Lemma 12.2.2]{pietsch1980ideals} with the choices $n=\left\lceil H \left(\varepsilon\semcol \mu_{\bar d_{k-1}} \mathcal{B}_p^{d_k}, \|\cdot\|_q \right)\right\rceil - 1$, $m=\mathfrak{d}'$, $u=p$, and $v=q$, noting that the condition $n \leq m$ in \cite[Lemma 12.2.2]{pietsch1980ideals} is met thanks to
(\ref{eq:hivjehbvhjfdh}), to get 
\begin{equation*}
    e_n'
    \leq 4 \left[8 \frac{\log\left(1+\mathfrak{d}'\right)}{\left\lceil H \left(\varepsilon\semcol \mu_{\bar d_{k-1}} \mathcal{B}_p^{d_k}, \|\cdot\|_q \right)\right\rceil - 1}\right]^{1/p-1/q},
\end{equation*}
where $e_n'$ is the $n$-th entropy number (see \cite[12.1.2]{pietsch1980ideals} for a formal definition) associated with the ball $\mathcal{B}_p^{\mathfrak{d}'}$.
 Next, we note that $\mathfrak{d}'\geq d_k$ implies $e_n'\geq e_n$, with $e_n$ the entropy number associated with the ball $\mathcal{B}_p^{d_k}$.
 Moreover, as
  \begin{equation*}
    n
    =\left\lceil H \left(\varepsilon\semcol \mu_{\bar d_{k-1}} \mathcal{B}_p^{d_k}, \|\cdot\|_q \right)\right\rceil - 1 
    < H \left(\varepsilon\semcol \mu_{\bar d_{k-1}} \mathcal{B}_p^{d_k}, \|\cdot\|_q \right)
    = H \left(\varepsilon/ \mu_{\bar d_{k-1}} \semcol \mathcal{B}_p^{d_k}, \|\cdot\|_q \right),
  \end{equation*}
 it follows that $ \varepsilon/ \mu_{\bar d_{k-1}}  \leq  e_n$.
 We have therefore established that
\begin{equation}\label{eq:hjkfdshkjfgdhdhhdhsdsffggg}
    \varepsilon 
    \leq 4 \left[8 \frac{\log\left(1+\mathfrak{d}'\right)}{\left\lceil H \left(\varepsilon\semcol \mu_{\bar d_{k-1}} \mathcal{B}_p^{d_k}, \|\cdot\|_q \right)\right\rceil - 1}\right]^{1/p-1/q} \mu_{\bar d_{k-1}}.
\end{equation}
%}%
%\textcolor{blue}{
Combining (\ref{eq:jksdhsjdhjfdbdshjjhjjhggqshvdgsfjdhuyuy}) with (\ref{eq:hjkfdshkjfgdhdhhdhsdsffggg}), now yields (for sufficiently small $\varepsilon$) 
\begin{equation*}
    H \left(\varepsilon\semcol \mathcal{E}_p^{[k]}, \|\cdot\|_q \right)
  \leq C''_{p,q} \log\left(1+\newldeff'\right)\left(\frac{\mu_{\bar d_{k-1}}}{\varepsilon}\right)^{\frac{1}{{1/p-1/q}}},
\end{equation*}
where $C''_{p,q}$ is a constant that depends on $p$ and $q$ only.
We further get 
\begin{align*}
\log\left(1+\newldeff'\right)\left(\frac{\mu_{\bar d_{k-1}}}{\varepsilon}\right)^{\frac{1}{{1/p-1/q}}}
=O_{\varepsilon\to 0} \left( \newldeff \log \left(\varepsilon^{-1}\right) k^{- \frac{b}{1/p-1/q}}\right),
\end{align*}
where we  used $\log(\newldeff') = O_{\varepsilon\to 0}(\log (\varepsilon^{-1}))$, $\bar d_{k-1}=(k-1)\newldeff$, and the fact that 
$\{\mu_n\}_{n\in \mathbb{N}^*}$ is regularly varying with index $-b$.
Moreover, it follows from the definition of $k$ in (\ref{eq:parameters-for-decomposition}) that
\begin{equation}\label{eq:jqbfzvhjsdbfhjbhsbfhb}
    \newldeff \log (\varepsilon^{-1}) \, k^{- \frac{b}{1/p-1/q}}
= O_{\varepsilon \to 0}\left(k^{1-b}\,\newldeff\right).
\end{equation}
Combining (\ref{eq:jqbfzvhjsdbfhjbhsbfhb}) with (\ref{eq:joiefhboznrkpjbers}) and (\ref{eq:lllllldfejzbjre1})--(\ref{eq:lllllldfejzbjre2}),
it follows that the metric entropy of $\mathcal{E}_p$ can be upper-bounded by a term scaling (as $\varepsilon \rightarrow 0$) according to
\begin{equation}\label{eq:venjvhrfelzdhsgkhesddfet}
\newldeff \sum_{j=1}^{k-1} \left[\log\left(1 + \frac{\mu_{j \newldeff}}{\mu_{\newldeff}}\right) + \log\left(\frac{\mu_{1+(j-1) \newldeff}}{\mu_{j \newldeff}} \right)\right] + O_{\varepsilon \to 0}\left(k^{1-b}\,\newldeff\right). 
\end{equation}
Next, using that $\{\mu_n\}_{n\in \mathbb{N}^*}$ is regularly varying with index $-b$, we obtain
\begin{equation}\label{eq:nbeihjzbkjgtvoij1}
    \sum_{j=1}^{k-1} \log\left(1 + \frac{\mu_{j \newldeff}}{\mu_{\newldeff}}\right) = 
    \begin{cases}   
    O_{\varepsilon \to 0}\left(k^{1-b}\right), & \textit{for}\,\, b\neq 1, \\[0.25cm]
    O_{\varepsilon \to 0}\left(\log(k)\right), & \textit{for}\,\, b = 1.
\end{cases}
\end{equation}
Likewise, we have 
\begin{equation}\label{eq:bvhjajhebjvrrr}
    \sum_{j=1}^{k-1}  \log\left(\frac{\mu_{1+(j-1) \newldeff}}{\mu_{j \newldeff}} \right) = O_{\varepsilon \to 0}\left(\log(k)\right).
\end{equation}
Combining (\ref{eq:venjvhrfelzdhsgkhesddfet}) with (\ref{eq:nbeihjzbkjgtvoij1}) and (\ref{eq:bvhjajhebjvrrr}), finally yields
\begin{equation*}
 H \left(\varepsilon\semcol \mathcal{E}_p, \|\cdot\|_q \right)\,
\varepsilon^{1/b^*}
=O_{\varepsilon \to 0}\left(\log^{(2)}(\varepsilon^{-1})\right),
\quad \text{if } b \geq 1,
\end{equation*}
and 
\begin{align*}
 H \left(\varepsilon\semcol \mathcal{E}_p, \|\cdot\|_q \right)\,
\varepsilon^{1/b^*}
&=O_{\varepsilon \to 0}\left(\log^{{\frac{1-b}{1+b(({1/p-1/q})^{-1}-1)}}}(\varepsilon^{-1})\right)\\
&=O_{\varepsilon \to 0}\left(\log^{1-b}(\varepsilon^{-1})\right),
\quad \text{if } b< 1,
\end{align*}
which is the desired result.%}

\subsection{Proof of Theorem~\ref{thm:mainrespqtwo}}\label{sec:proofmainresultspqtwo}

We start by deriving an upper bound on $\varepsilon$ valid for $p=q=2$ and matching the corresponding lower bound in Lemma~\ref{lem:1327}.

\begin{lemma}\label{lem:relatpqtwomuneps}
  %\textcolor{blue}{
  Let $b>0$ and let $\{\mu_n\}_{n\in \mathbb{N}^*}$ be a non-increasing sequence that is regularly varying with index $-b$. Let $\mathcal{E}_2$ be the $2$-ellipsoid with semi-axes $\{\mu_n\}_{n\in \mathbb{N}^*}$.
Then, as $\varepsilon \rightarrow 0$, we have 
\begin{equation*}
    \varepsilon \left(1+\varepsilon^{\max\{b, b^{-1}\}} \right)^{-1} \leq \mu_{\leffdim_{s}},
     \text{ with } s\geq 1 \text{ satisfying }
     s = 1+ O_{\varepsilon \to 0}\left(\varepsilon^{(2b)^{-1}} \log\left(\varepsilon^{-1}\right)\right),
\end{equation*}
and $\leffdim_s$ the lower effective dimension of $\mathcal{E}_2$ for $p=q=2$.
\end{lemma}

\begin{proof}[Proof.]
The proof is provided in Appendix \ref{sec:proofkljhgfdwghfdfudzzz}.
\end{proof}

\noindent
Combining Lemma~\ref{lem:1327} (with $p=q=2$ and $s=1$), Lemma~\ref{lem:relatpqtwomuneps}, and
Lemma~\ref{lem:1328}, we get
\begin{equation*}
\lim_{\varepsilon \to 0} \frac{\varepsilon}{\phi\left(H \left(\varepsilon\semcol \mathcal{E}_2, \|\cdot\|_2 \right) \right)} 
= \left(\frac{b}{\ln(2)}\right)^b
\end{equation*}
which is the desired result.

\subsection{Proof of Theorem~\ref{thm: Metric entropy for polynomial ellipsoids}}\label{sec:proofofhighlightresult}

Let $s \geq 1$ be as in Lemma~\ref{lem:relatpqtwomuneps}.
Recalling the definitions (\ref{eq: definition of underbar n pt1mg3}) and (\ref{eq: definition of bar n pt1mg3fst}) of upper and lower effective dimension, respectively, we get
\begin{equation}\label{eq:hgfzoaberzhfqzeezeaa}
 \leffdim_s \log(s) + \sum_{n=1}^{\leffdim_s} \log \left ( \frac{\mu_{n}}{\mu_{\leffdim_s}} \right )
  < H \left(\varepsilon\semcol \mathcal{E}_2, \|\cdot\|_2\right)
  \leq \ueffdim_s \log(s) + \sum_{n=1}^{\ueffdim_s} \log \left ( \frac{\mu_{n}}{\mu_{\ueffdim_s}} \right ).
\end{equation}
The asymptotic behavior of the upper and the lower bounds in (\ref{eq:hgfzoaberzhfqzeezeaa}) can be characterized through the following result.

\begin{lemma}\label{lem:kojihgfyfqsifgsquzehfv}
Let $\alpha_1,  \alpha_2 \in \mathbb{R}_{+}^{*}$ be such that $\alpha_1 < \alpha_2$ and ${\mathfrak{a} \coloneqq \alpha_1 - \alpha_2 + 1} > 0$. 
Let $\{\mu_n\}_{n\in \mathbb{N}^*}$ be a sequence of positive real numbers satisfying
\begin{equation}\label{eq:jhkqdolfebzjbfqebtgbww}
\mu_n 
= \frac{c_1}{n^{\alpha_1}} + \frac{c_2}{n^{\alpha_2}} + o_{n \to \infty} \left(\frac{1}{n^{\alpha_2}}\right),
\end{equation}
with $c_1>0$ and $c_2\in \mathbb{R}$.
Then, we have
 \begin{equation*}
 \sum_{n=1}^{d} \log \left ( \frac{\mu_{n}}{\mu_{d}} \right )
 = \frac{\alpha_1 }{\ln(2)} \, d + \left(\frac{1}{\mathfrak{a}}-1\right)\frac{c_2}{c_1\ln(2)} \, d^{\mathfrak{a}}
 +o_{d\to\infty}\left(d^{\mathfrak{a}}\right).
 \end{equation*}
 \end{lemma} 
 \begin{proof}[Proof.]
 See Appendix \ref{sec:prooflemkojihgfyfqsifgsquzehfv}.
 \end{proof}

\noindent
Recalling that $\leffdim_s = \ueffdim_s -1$, application of Lemma~\ref{lem:kojihgfyfqsifgsquzehfv} to the upper and the lower bounds in (\ref{eq:hgfzoaberzhfqzeezeaa}) yields
\begin{equation}\label{eq:mlkjhgfkkkphgfftttt}
  H \left(\varepsilon\semcol \mathcal{E}_2, \|\cdot\|_2\right) 
  = \left( \log(s) + \frac{\alpha_1 }{\ln(2)}\right) \, \ueffdim_s + \left(\frac{1}{\mathfrak{a}}-1\right)\frac{c_2}{c_1\ln(2)} \, \ueffdim_s^{\mathfrak{a}}
  +o_{\ueffdim_s\to\infty}\left(\ueffdim_s^{\mathfrak{a}}\right),
\end{equation}
where the assumption $\mathfrak{a} > 0$ is satisfied thanks to 
(\ref{eq:conditiononalpha12}).
We proceed to make the dependence of $\ueffdim_s$ on $\varepsilon$ explicit.
To this end, we first apply Lemma~\ref{lem:1327} with $p=q=2$ and Lemma~\ref{lem:relatpqtwomuneps} to conclude that 
\begin{equation}\label{eq:lower-upper-bounds-mus}
 s^{-1} \mu_{\ueffdim_s}
  \leq \varepsilon 
  \leq \mu_{\leffdim_s} \left(1+\varepsilon^{\max\{\alpha_1, \alpha_1^{-1}\}} \right),
\end{equation}
with $s$ satisfying
\begin{equation*}
s = 1+ O_{\varepsilon \to 0}\left(\varepsilon^{(2\alpha_1)^{-1}} \log\left(\varepsilon^{-1}\right)\right).
\end{equation*}
The lower bound in \eqref{eq:lower-upper-bounds-mus} directly implies
\begin{equation}\label{eq:oiuhgfgfhdakaplksd}
  \mu_{\ueffdim_s} \leq \varepsilon \left(1+ O_{\varepsilon \to 0}\left(\varepsilon^{(2\alpha_1)^{-1}} \log\left(\varepsilon^{-1}\right)\right)\right).
\end{equation}
To see that \eqref{eq:oiuhgfgfhdakaplksd} in fact holds with equality, we first note that 
\begin{align}\label{eq:equivalence-lower-upper-dimension}
  \mu_{\leffdim_s} 
  &= \frac{c_1}{\leffdim_s^{\alpha_1}} + \frac{c_2}{\leffdim_s^{\alpha_2}} + o_{\leffdim_s \to \infty} \left(\frac{1}{\leffdim_s^{\alpha_2}}\right)\\
  &= \frac{c_1}{(\ueffdim_s-1)^{\alpha_1}} + \frac{c_2}{(\ueffdim_s-1)^{\alpha_2}} + o_{\ueffdim_s \to \infty} \left(\frac{1}{\ueffdim_s^{\alpha_2}}\right) \nonumber\\
  &= \frac{c_1}{\ueffdim_s^{\alpha_1}}+ \frac{c_2}{\ueffdim_s^{\alpha_2}} + o_{\ueffdim_s \to \infty} \left(\frac{1}{\ueffdim_s^{\alpha_2}}\right), \label{eq:equivalence-lower-upper-dimension-last}
\end{align}
where we used $\ueffdim_{s} = \leffdim_{s} +1$ and (\ref{eq:conditiononalpha12}).
Thanks to \eqref{eq:equivalence-lower-upper-dimension}--\eqref{eq:equivalence-lower-upper-dimension-last}, we can replace $\mu_{\leffdim_s}$ in the upper bound of \eqref{eq:lower-upper-bounds-mus} by $\mu_{\ueffdim_s}$ and, upon noting that
$$
\frac{1}{\left(1+\varepsilon^{\max\{\alpha_1, \alpha_1^{-1}\}} \right)} = 1+ O_{\varepsilon \to 0}\left(\varepsilon^{(2\alpha_1)^{-1}} \log\left(\varepsilon^{-1}\right)\right),
$$
get 
\begin{equation*}
  \mu_{\ueffdim_s} \geq \varepsilon \left(1+ O_{\varepsilon \to 0}\left(\varepsilon^{(2\alpha_1)^{-1}} \log\left(\varepsilon^{-1}\right)\right)\right).
\end{equation*}
We now require the following inversion result to make the dependence of $\ueffdim_s$ on $\varepsilon$ from 
\begin{equation*}
  \mu_{\ueffdim_s} = \varepsilon \left(1+ O_{\varepsilon \to 0}\left(\varepsilon^{(2\alpha_1)^{-1}} \log\left(\varepsilon^{-1}\right)\right)\right)
\end{equation*}
explicit.

\begin{lemma}\label{lem:oijhbhjavqiaxvsregvfd}
Let $\alpha_1,  \alpha_2 \in \mathbb{R}_{+}^{*}$ be such that $\alpha_1 < \alpha_2$, let $c_1>0$, $c_2 \in \mathbb{R}$, and define 
$g\colon \mathbb{R_+^*}\to \mathbb{R_+^*}$ as
\begin{equation}\label{eq:mljhgfhghgjgqsuk}
  g(u) 
  = \frac{c_1}{u^{\alpha_1}}+ \frac{c_2}{u^{\alpha_2}} + o_{u \to \infty} \left(\frac{1}{u^{\alpha_2}}\right).
\end{equation}
Then, we have 
  \begin{equation*}
    u 
    = {c_1}^{1/\alpha_1}{g(u) }^{-1/\alpha_1} 
+ \frac{c_2\, {c_1}^{\frac{1-\alpha_2}{\alpha_1}}}{\alpha_1}\, g(u) ^{-\frac{\alpha_1-\alpha_2+1}{\alpha_1}} 
+ o_{u \to \infty}\left(g(u) ^{-\frac{\alpha_1-\alpha_2+1}{\alpha_1}}\right).
  \end{equation*}
\end{lemma}

\begin{proof}[Proof.]
  See Appendix \ref{sec:prooflemmaoijhbhjavqiaxvsregvfd}.
\end{proof}

\noindent
Application of Lemma~\ref{lem:oijhbhjavqiaxvsregvfd}, with $u= \ueffdim_s$ and $g(u)=\varepsilon (1+ O_{\varepsilon \to 0}(\varepsilon^{(2\alpha_1)^{-1}} \log(\varepsilon^{-1})))$, in combination with the assumption (\ref{eq:conditiononalpha12}), now yields
\begin{equation}\label{eq: eps1330}
\ueffdim_s 
= {c_1}^{1/\alpha_1}{\varepsilon}^{-1/\alpha_1} 
+ \frac{c_2\, {c_1}^{\frac{1-\alpha_2}{\alpha_1}}}{\alpha_1}\varepsilon^{-\frac{\alpha_1-\alpha_2+1}{\alpha_1}} 
+ o_{\varepsilon \to 0}\left(\varepsilon^{-\frac{\alpha_1-\alpha_2+1}{\alpha_1}}\right).
\end{equation}
In particular, we have 
\begin{align*}
  \ueffdim_s \log(s)
  &= O_{\varepsilon \to 0}\left({\varepsilon}^{-1/\alpha_1} \log\left(1+ \varepsilon^{(2\alpha_1)^{-1}} \log\left(\varepsilon^{-1}\right)\right) \right)\\
  &= O_{\varepsilon \to 0}\left(\varepsilon^{-(2\alpha_1)^{-1}} \log\left(\varepsilon^{-1}\right) \right)
  =o_{\varepsilon \to 0}\left(\varepsilon^{-\frac{\alpha_1-\alpha_2+1}{\alpha_1}}\right),
\end{align*}
where again the assumption (\ref{eq:conditiononalpha12}) was invoked.
Finally, using (\ref{eq: eps1330}) in (\ref{eq:mlkjhgfkkkphgfftttt}) yields
\begin{align*}
H \left(\varepsilon\semcol \mathcal{E}_2, \|\cdot\|_2\right)
= \ & \frac{\alpha_1}{\ln(2)} \, \left( {c_1}^{1/\alpha_1}{\varepsilon}^{-1/\alpha_1} 
+ \frac{c_2\, {c_1}^{\frac{1-\alpha_2}{\alpha_1}}}{\alpha_1}\varepsilon^{-\frac{\mathfrak{a}}{\alpha_1}} \right) \\
&+\left(\frac{1}{\mathfrak{a}}-1\right)\frac{c_2}{c_1\ln(2)} \, \left( {c_1}^{1/\alpha_1}{\varepsilon}^{-1/\alpha_1} \right)^{\mathfrak{a}}
+ o_{\varepsilon \to 0}\left(\varepsilon^{-\frac{\mathfrak{a}}{\alpha_1}}\right).
\end{align*}
Recalling that ${\mathfrak{a} = \alpha_1 - \alpha_2 + 1}$, this expression can be simplified to deliver the desired result
\begin{equation*}
  H \left(\varepsilon\semcol \mathcal{E}_2, \|\cdot\|_2\right)
  = \frac{\alpha_1{c_1}^{\frac{1}{\alpha_1}}}{\ln(2)} \, {\varepsilon}^{-\frac{1}{\alpha_1}}
  + \frac{c_2\, {c_1}^{\frac{1-\alpha_2}{\alpha_1}}}{\ln(2)(\alpha_1-\alpha_2+1)} \, \varepsilon^{-\frac{\alpha_1-\alpha_2+1}{\alpha_1}} 
  + o_{\varepsilon \to 0}\left(\varepsilon^{-\frac{\alpha_1-\alpha_2+1}{\alpha_1}}\right).
\end{equation*}

\subsection{Proof of Theorem~\ref{thm:infellpinfnorm}}\label{sec:proofthminfellpinfnorm}

We first state a result which relates the metric entropy of the ellipsoid $\mathcal{E}_\infty$ to its semi-axes.

\begin{lemma}\label{lem:firstellipinfnorm}
Let $\{\mu_n\}_{n \in \mathbb{N}^*}$ be a sequence of positive real numbers and 
let $\mathcal{E}_\infty$ be the $\infty$-ellipsoid with semi-axes $\{\mu_n\}_{n\in \mathbb{N}^*}$.
Then, 
\begin{equation}\label{eq:skqnvjkefnvfzjksgbejkrzs}
H\left(\varepsilon\semcol \mathcal{E}_\infty, \|\cdot\|_\infty\right)
= \sum_{n=1}^\infty \log\left(\left\lceil \frac{\mu_n}{\varepsilon} \right\rceil\right), \quad \text{ for all } \varepsilon > 0.
\end{equation}
\end{lemma}

\vspace*{-3mm}

\begin{proof}[Proof.]
See Appendix \ref{sec:prooflemfirstellipinfnorm}.
\end{proof}

\noindent
The right-hand side of (\ref{eq:skqnvjkefnvfzjksgbejkrzs}) can be expanded further through the following lemma.

%\textcolor{blue}{
\begin{lemma}\label{lem:firstellipinfnorm2}
  Let $\{\mu_n\}_{n \in \mathbb{N}^*}$ be a sequence of positive real numbers with associated counting function $M_k(\cdot)$ according to (\ref{eq:hjqjhjhbsfvbfbvsjkhbvjhbsfvbsjksjfd}).
  Then, for all $\varepsilon >0$, we have 
\begin{equation}\label{eq:nzapfpspdjcvklskpppplazdnfkzv}
\sum_{n=1}^\infty \log \left(\left\lceil \frac{\mu_n}{\varepsilon} \right\rceil\right)
= \sum_{k=2}^\infty \log(k) 
\left(M_{k-1}(\varepsilon)-M_k(\varepsilon)\right).
\end{equation}
\end{lemma}%}

\vspace*{-3mm}

\begin{proof}[Proof.]
See Appendix \ref{sec:prooflemfirstellipinfnorm2}.
\end{proof}

%\textcolor{blue}{
\noindent Setting
\begin{equation*}
d \coloneqq \sup_{n \in \mathbb{N}^*}\left\lceil\frac{\mu_n}{\varepsilon}\right\rceil,
\end{equation*}
we next note that, for all $\varepsilon > 0$,
\begin{equation}\label{eq:1330}
  M_{k-1}(\varepsilon)
= M_{k}(\varepsilon)
=0,
\quad \text{for all } k \geq d +1.
\end{equation}
This allows us to conclude that, in fact, the sum in (\ref{eq:nzapfpspdjcvklskpppplazdnfkzv}) ranges from $2$ to $d$ only, i.e.,
\begin{equation} \label{eq:1231}
\sum_{k=2}^\infty \log(k) 
\left(M_{k-1}(\varepsilon)-M_k(\varepsilon)\right)
= \sum_{k=2}^d \log(k) 
\left(M_{k-1}(\varepsilon)-M_k(\varepsilon)\right).
\end{equation}
Next, we develop (\ref{eq:1231}) further according to
\begin{align}
\sum_{k=2}^d \log(k) \left(M_{k-1}(\varepsilon)-M_k(\varepsilon)\right)
 &= \sum_{k=2}^d \log(k) M_{k-1}(\varepsilon)
- \sum_{k=2}^d \log(k) M_{k}(\varepsilon) \label{eq:1230}\\
& = \sum_{k=1}^{d-1} \log(k+1) M_{k}(\varepsilon)
- \sum_{k=1}^{d-1} \log(k) M_{k}(\varepsilon) \nonumber\\
& = \sum_{k=1}^{d-1} \log\left(1+\frac{1}{k} \right) M_{k}(\varepsilon),\label{eq:1229}
\end{align}
where, in the second step, we used $M_{d}(\varepsilon)=0$.
Finally, with (\ref{eq:1330}), we obtain
\begin{equation}\label{eq:1228}
\sum_{k=1}^{d-1}  \log\left(1+\frac{1}{k} \right) M_{k}(\varepsilon)
= \sum_{k=1}^{\infty}  \log\left(1+\frac{1}{k} \right) M_{k}(\varepsilon).
\end{equation}
Combining Lemmata \ref{lem:firstellipinfnorm} and \ref{lem:firstellipinfnorm2} with (\ref{eq:1231}), (\ref{eq:1230})--(\ref{eq:1229}), and (\ref{eq:1228}) finalizes the proof.

\subsection{Proof of Corollary \ref{cor:infellpinfnorm}}\label{sec:proofcorinfellpinfnorm}

Particularizing (\ref{eq:okijhugffcdsvhcdz}) to
\begin{equation*}
\phi(n)=\frac{c}{n^b},
\quad \text{for } n\in \mathbb{N}^*,
\end{equation*}
yields 
\begin{align}
H \left(\varepsilon\semcol \mathcal{E}_\infty, \|\cdot\|_\infty\right)
&= \sum_{k=1}^\infty \log \left(1+\frac{1}{k} \right) \left(\left\lceil \left(\frac{c}{k\varepsilon}\right)^{1/b} \right\rceil -1\right), \quad \text{for all } \varepsilon > 0. \label{kkkdjriizufouuuu2}
\end{align}
This readily gives
\begin{equation}\label{ubincorinfellnorminf}
 H \left(\varepsilon\semcol \mathcal{E}_\infty, \|\cdot\|_\infty\right)
 \leq \sum_{k=1}^\infty \log \left(1+\frac{1}{k} \right) \left(\left\lfloor \left(\frac{c}{k\varepsilon}\right)^{1/b} \right\rfloor\right)
 \leq \frac{c^{1/b}}{\varepsilon^{1/b}}\sum_{k=1}^\infty \frac{\log \left(1+\frac{1}{k} \right)}{k^{1/b}},
\end{equation}
for all $\varepsilon>0$.
The following lemma provides a matching lower bound.

\begin{lemma}\label{lem:lbincorinfellnorminf}
Let $b>0$ and $c>0$. We have
\begin{equation*}
\sum_{k=1}^\infty \log \left(1+\frac{1}{k} \right) \left(\left\lceil \left(\frac{c}{k\varepsilon}\right)^{1/b} \right\rceil -1\right)
 \geq \frac{c^{1/b}}{\varepsilon^{1/b}}\sum_{k=1}^\infty \frac{\log \left(1+\frac{1}{k} \right)}{k^{1/b}} + O_{\varepsilon \to 0}\left( \log\left(\varepsilon^{-1}\right)\right).
\end{equation*}
\end{lemma}

\begin{proof}[Proof.]
See Appendix \ref{sec:prooflbincorinfellnorminf}.
\end{proof}

\noindent
Combining Lemma~\ref{lem:lbincorinfellnorminf}, (\ref{kkkdjriizufouuuu2}), and (\ref{ubincorinfellnorminf}), yields
\begin{equation}\label{eq:13330}
 H \left(\varepsilon\semcol \mathcal{E}_\infty, \|\cdot\|_\infty\right)
 = \frac{c^{1/b}}{\varepsilon^{1/b}}\sum_{k=1}^\infty \frac{\log \left(1+\frac{1}{k} \right)}{k^{1/b}} + O_{\varepsilon \to 0}\left( \log\left(\varepsilon^{-1}\right)\right).
\end{equation}
The desired lower bound (\ref{eq:lbforoptimalityinfinf}) now follows directly from (\ref{eq:13330}).
To get (\ref{eq:lbforoptimalityinfinfmainres}), we further observe that
\begin{align}
\sum_{k=1}^\infty \frac{\log \left(1+\frac{1}{k} \right)}{k^{1/b}}
&= \frac{1}{\ln(2)}\sum_{k=1}^\infty  \sum_{\ell=1}^\infty \frac{(-1)^{\ell+1}}{\ell \, k^{\ell+ 1/b}} \label{eq:13331}\\
&= \frac{1}{\ln(2)}\sum_{\ell=1}^\infty  \sum_{k=1}^\infty \frac{(-1)^{\ell+1}}{\ell \, k^{\ell+ 1/b}}\nonumber\\
&= \frac{1}{\ln(2)} \sum_{\ell=1}^\infty \frac{(-1)^{\ell+1}}{\ell}\zeta\left(\ell+\frac{1}{b}\right),\label{eq:13332}
\end{align}
where we used Fubini's theorem and $\zeta(\cdot)$ denotes the Riemann zeta function.
Using (\ref{eq:13331})--(\ref{eq:13332}) in (\ref{eq:13330}) establishes (\ref{eq:lbforoptimalityinfinfmainres}).

\subsection{Proof of Theorem~\ref{thm:depomdomainn}}\label{sec:proofthm199}

Taking $\eta\to 0$ in the inclusion relation \cite[Chapter 2.3.2, Proposition 2(ii)]{triebelTheoryFunctionSpaces1983} 
\begin{equation*}
  B_{p_1, p_1}^{s+\eta}(\Omega) 
  \subset B_{p_1, p_2}^s(\Omega) 
  \subset B_{p_1, p_1}^{s-\eta}(\Omega),
\end{equation*}
shows that it suffices to prove the statement for the diagonal case $p\coloneqq p_1=p_2$. We employ the
orthonormal wavelet basis construction described in \cite{grohsPhaseTransitionsRate2021} and accordingly start by defining the sets
\begin{equation}\label{eq:nkfezvjbjhbfrhjzbr}
\begin{cases}
J\coloneqq [(\{0\}\times T_0)\cup (\mathbb{N}^*\times T)]\times \mathbb{Z}^d, \\[.2cm]
J_j\coloneqq \left\{(t, m) \in \{F, M\}^d \times \mathbb{Z}^d  \mid (j, t,m)\in J \right\}, \quad &\text{for } j \in \mathbb{N},
\end{cases}
\end{equation}
with 
\begin{equation*}
T_0 \coloneqq \{F\}^d, \quad
T \coloneqq \{F, M\}^d\setminus T_0,
\end{equation*}
where $M$ and $F$ label the mother and father wavelet, respectively, both of which are compactly supported.
The associated wavelet family is denoted by
$\{\psi_{j, t, m}\}_{(j, t, m)\in J}$, where $j$ designates the scaling parameter and $m$ is the translation parameter.
For details on the construction of $\{\psi_{j, t, m}\}_{(j, t, m)\in J}$, including the tensorization aspect (recall that $\Omega \subset \mathbb{R}^{d}$) and
regularity conditions on the mother and father wavelets, the reader is referred to \cite{grohsPhaseTransitionsRate2021}.
Here, we shall only use that
(i) the support of $\psi_{j, t, m}$ is a $d$-dimensional cube of sidelength proportional to $2^{-j}$, for all $(j, t, m)\in J$, and
(ii) $\{\psi_{j, t, m}\}_{(j, t, m)\in J}$ is an orthonormal basis for $L^2(\mathbb{R}^d)$, so that every function $f\in B_{p, p}^s(\Omega) \hookrightarrow L^2(\mathbb{R}^d)$ can be uniquely represented by its wavelet coefficients $\lambda = \{\lambda_{j, t, m}\}_{(j, t, m) \in J }$, with $\lambda_{j, t, m} \coloneqq \langle f , \psi_{j, t, m}\rangle$. 
The norm 
\begin{equation}\label{eq:jkvzekbskjdlvq}
\left\| \lambda \right\|_{b^s_{p,p}}
= \left(\sum_{j=0}^\infty 2^{j\left(s-d\left(\frac{1}{p}-\frac{1}{2}\right)\right)p} \sum_{(t, m) \in J_j} \left\lvert \lambda_{j, t, m} \right\rvert^{p}  \right)^{1/p}
\end{equation}
then defines the space $b^s_{p,p}$ of Besov coefficients.
Next, with
\begin{equation*}
\begin{cases}
J^+ \coloneqq \left\{(j, t, m)\in J \mid \Omega \cap \text{supp} (\psi_{j, t, m}) \neq \emptyset  \right\},
\\[.2cm]
J^- \coloneqq \left\{(j, t, m)\in J \mid \text{supp} (\psi_{j, t, m}) \subseteq \Omega \right\},
\end{cases}
\end{equation*}
let $b_{p, p}^{s, +} $ and $b_{p, p}^{s, -} $ be the spaces of Besov coefficients corresponding to the index sets $J^+$ and $J^-$, respectively, i.e., 
\begin{equation}\label{eq:ivkzejntkjvsnkq}
b_{p, p}^{s, \pm}
\coloneqq \left\{\lambda \in b_{p, p}^{s} \mid \lambda_{j,t,m} = 0, \ \text{ if } (j, t, m ) \in J \setminus J^{\pm} \right\} .
\end{equation}
In particular, for $\lambda \in b_{p, p}^{s, \pm}$ we can rewrite (\ref{eq:jkvzekbskjdlvq}) according to
\begin{equation}\label{eq:nkzaenlkcbajz}
\left\| \lambda \right\|_{b^s_{p,p}}
= \left(\sum_{j=0}^\infty 2^{j\left( s-d\left(\frac{1}{p}-\frac{1}{2}\right)\right)p}  \sum_{(t, m) \in J_j^{\pm}}  \left\lvert \lambda_{j, t, m} \right\rvert^{p} \right)^{1/p},
\end{equation}
with
\begin{equation*}
    J_j^{\pm} \coloneqq \left\{(t, m) \in \{F, M\}^d \times \mathbb{Z}^d  \mid (j, t,m)\in J^{\pm} \right\}, \quad \text{for } j \in \mathbb{N}.
\end{equation*}
The double sum on the right-hand-side of (\ref{eq:nkzaenlkcbajz}) may now be reindexed by summing over $j$ and $k$ according to
\begin{equation*}
   n^{\pm}_{j,k} = \sum_{\tau=0}^{j-1} \#J_\tau^{\pm} +k, 
   \quad \text{for } j\in \mathbb{N} \text{ and } k\in \left\{1, \dots, \#J_j^{\pm}\right\}.
 \end{equation*} 
Next, invoking the boundedness of $\Omega$ and recalling that the mother and father wavelets are compactly supported, it follows that 
there exists $j^* \in \mathbb{N}^*$ such that 
\begin{equation*}
  c \,  \vol{\left(\Omega\right)} 2^{dj} \leq  \# J_j^- \leq \# J_j^+ \leq C \,  \vol{\left(\Omega\right)} 2^{dj},
   \quad \text{for all } j\geq j^*,
\end{equation*}
with $c, C>0$ independent of $j$ and $\Omega$.
In particular, we have 
\begin{equation*}
 \kappa^\pm 2^{dj}\vol{\left(\Omega\right)}
 \leq  n^{\pm}_{j,k} 
 \leq K^\pm 2^{dj}\vol{\left(\Omega\right)}, \quad \text{for all } j\geq j^* \text{ and }  k\in \left\{1, \dots, \#J_j^{\pm}\right\},
\end{equation*}
with $\kappa^\pm, K^\pm > 0$ independent of $j,k$, and $\Omega$.
The reindexed sum in (\ref{eq:nkzaenlkcbajz}) can then be bounded according to
\begin{equation}\label{eq:kjzdkqncjefq}
\mathfrak{c}^\pm \left(\sum_{{n^\pm}=1}^\infty \left \lvert \frac{\lambda_{n^\pm}}{\mu_{n^\pm}} \right\rvert^{p}\right)^{1/p}
  \leq \left\| \lambda \right\|_{b^s_{p,p}}
  \leq \mathfrak{C}^\pm \left(\sum_{n^\pm=1}^\infty \left \lvert \frac{\lambda_{n^\pm}}{\mu_{n^\pm}} \right\rvert^{p}\right)^{1/p}, 
  \quad \text{for all } \lambda \in b_{p, p}^{s, \pm},
\end{equation}
with $\mathfrak{c}^\pm>0$ and $\mathfrak{C}^\pm>0$ independent of $j$ and $\Omega$ and 
\begin{equation*}
  \mu_{n^\pm} \coloneqq \left(\frac{\vol{\left(\Omega\right)}}{n^\pm}\right)^{\left(\frac{s}{d}-\left(\frac{1}{p}-\frac{1}{2}\right)\right)},
  \quad \text{for } {n^\pm}\in \mathbb{N}^*.
\end{equation*}

Through (\ref{eq:kjzdkqncjefq}) we have equipped $b^{s, \pm}_{p,p}$ with ellipsoid structure and can thus apply case (ii) of Corollary \ref{cor:oadcpzrjefnkznrs}, with $b=s/d +1/2-1/p $ and $q=2$, to get the following result on the metric entropy of the unit ball ${\tilde b}^{s, \pm}_{p,p}$ in the coefficient space:
there exist $\varepsilon^*>0$, $\mathfrak{k}^\pm>0$, and $\mathfrak{K}^\pm>0$, all independent of $\Omega$, such that, for all $\varepsilon \in (0, \varepsilon^*)$, we have
\begin{equation}\label{eq:hhhtgfgtgfgt}
\mathfrak{k}^\pm \,  \vol{\left(\Omega\right)}^{1-\frac{d}{s}(\frac{1}{p}-\frac{1}{2})}\,
\varepsilon^{-d/s}
\leq  H \left( \varepsilon \semcol {\tilde b}^{s, \pm}_{p,p}, \|\cdot\|_{\ell^2(J)}\right)
\leq \mathfrak{K}^\pm \,  \vol{\left(\Omega\right)}^{1-\frac{d}{s}(\frac{1}{p}-\frac{1}{2})}\,
\varepsilon^{-d/s}.
\end{equation}%
It remains to relate the metric entropy of ${\tilde b}^{s, \pm}_{p,p}$ to that of
$\mathcal{B}^s_{p, p}(\Omega)$.
This can be effected through the following result.

\begin{lemma}[{\cite[Lemma 11]{grohsPhaseTransitionsRate2021}}]\label{eq:aznkjejfz}
Let $d\in \mathbb{N}^*$, let $\Omega \subset \mathbb{R}^d$ be a nonempty, bounded, and open set, let $s>0$, and let $p_1,p_2\in[1,\infty]$.
Then, there exist continuous linear operators 
\begin{equation*}
T_{+} \colon b_{p_1, p_2}^{s, +} \longrightarrow B^s_{p_1, p_2}(\Omega) 
\quad \text{and} \quad 
T_{-} \colon b_{p_1, p_2}^{s, -} \longrightarrow B^s_{p_1, p_2}(\Omega) 
\end{equation*}
with the following properties:
\begin{itemize}
\item
there exists $\gamma_+>0$, not depending on 
$\Omega$, such that $\|T_+ \lambda \|_{L^2(\Omega)}\leq \gamma_+\|\lambda \|_{\ell^2(J)}$, for all $\lambda \in b_{p_1, p_2}^{s, +} $;

\item 
there exists $\gamma_->0$, not depending on 
$\Omega$, such that $\|T_- \lambda \|_{L^2(\Omega)}=\gamma_-\|\lambda \|_{\ell^2(J)}$, for all $\lambda \in \ell^2(J) \cap b_{p_1, p_2}^{s, -} $;

\item
we have the inclusions
\begin{equation*}
T_- \left({\tilde b}^{s, -}_{p_1,p_2}\right)
\subseteq \mathcal{B}_{p_1, p_2}^s(\Omega)
\subseteq T_+ \left({\tilde b}^{s,+}_{p_1,p_2}\right).
\end{equation*}
\end{itemize}
\end{lemma}

\noindent 
Application of Lemma~\ref{eq:aznkjejfz} with $p_1=p_2=p$ and $\gamma_-,\gamma_+$ as in Lemma~\ref{eq:aznkjejfz}, 
combined with (\ref{eq:hhhtgfgtgfgt}),
now ensures that there exist $\varepsilon^*>0$, $\mathfrak{k}^\pm>0$, and $\mathfrak{K}^\pm>0$, all independent of $\Omega$, such that, for all $\varepsilon \in (0, \varepsilon^*)$, we have
\begin{align*}
 {\mathfrak{k}^\pm}{\gamma_-^{d/s}} \,  \vol{\left(\Omega\right)}^{1-\frac{d}{s}(\frac{1}{p}-\frac{1}{2})}
\varepsilon^{-d/s}
&\leq  H \left( \varepsilon/\gamma_- \semcol {\tilde b}^{s, -}_{p,p}, \|\cdot\|_{\ell^2(J) }\right)\\
& =  H \left( \varepsilon \semcol T_- \left({\tilde b}^{s, -}_{p,p}\right), \|\cdot\|_{L^2(\Omega)} \right)\\
&\leq  H \left( \varepsilon \semcol\mathcal{B}_{p, p}^s(\Omega), \|\cdot\|_{L^2(\Omega)}\right)\\
&\leq  H \left( \varepsilon \semcol T_+ \left({\tilde b}^{s,+}_{p,p}\right), \|\cdot\|_{L^2(\Omega)} \right)\\
&\leq  H \left( \varepsilon/\gamma_+ \semcol {\tilde b}^{s, +}_{p,p}, \|\cdot\|_{\ell^2(J) }\right)
\leq \mathfrak{K}^\pm \gamma_+^{d/s} \,  \vol{\left(\Omega\right)}^{1-\frac{d}{s}(\frac{1}{p}-\frac{1}{2})}
\varepsilon^{-d/s},
\end{align*}
which is the desired result.

\bibliography{references}% common bib file

@article{firstpaper,
  author = {Allard, Thomas and Bölcskei, Helmut},
  title = {Ellipsoid methods for metric entropy computation}, 
  year = {2024},
  journal = {submitted to Constructive Approximation}
}

@article{thirdpaper,
  author = {Allard, Thomas and Bölcskei, Helmut},
  title = {Entropy of compact operators with applications to {Landau-Pollak-Slepian} theory and {Sobolev} spaces}, 
  year = {2025},
month ={June},
  journal = {Applied and Computational Harmonic Analysis},
  volume={77},
number = {101762},
}

@article{Pinskerpaper,
  author = {Allard, Thomas},
  title = {Metric Entropy and Minimax Risk of Ellipsoids with an Application to {Pinsker}'s Theorem}, 
  journal={Available online: https://arxiv.org/abs/2510.22441},
  volume={},
  number={},
  pages={},
  year={2025},
  url={},
  publisher={}
}

@book{carlEntropyCompactnessApproximation1990,
  place={Cambridge}, 
  series={Cambridge Tracts in Mathematics}, 
  title={Entropy, Compactness and the Approximation of Operators}, 
  publisher={Cambridge University Press}, 
  author={Carl, Bernd and Stephani, Irmtraud}, 
  year={1990}, 
  collection={Cambridge Tracts in Mathematics}}

@book{donohoCountingBitsKolmogorov2000,
  title={Counting bits with Kolmogorov and Shannon},
  author={Donoho, David Leigh},
  year={2000},
  publisher={Department of Statistics, Stanford University}
}

@book{edmundsFunctionSpacesEntropy1996,
  place={Cambridge}, 
  series={Cambridge Tracts in Mathematics}, 
  title={Function Spaces, Entropy Numbers, Differential Operators}, 
  publisher={Cambridge University Press}, 
  author={Edmunds, D. E. and Triebel, H.}, 
  year={1996}, 
  collection={Cambridge Tracts in Mathematics}}

@article{elbrachterDeepNeuralNetwork2021,
  title = {Deep {neural} {network} {approximation} {theory}},
  author = {Elbrächter, Dennis and Perekrestenko, Dmytro and Grohs, Philipp and Bölcskei, Helmut},
  year = {2021},
  journal = {IEEE Transactions on Information Theory},
  volume = {67},
  number = {5},
  pages = {2581--2623},
  publisher={IEEE}}

@article{grafSharpAsymptoticsMetric2004,
  title={Sharp asymptotics of the metric entropy for ellipsoids},
  author={Graf, Siegfried and Luschgy, Harald},
  journal={Journal of Complexity},
  volume={20},
  number={6},
  pages={876--882},
  year={2004},
  publisher={Elsevier}
}

@book{lorentzConstructiveApproximationAdvanced1996,
  title = {Constructive {Approximation}: {Advanced} {Problems}},
  author = {Lorentz, G. G. and von Golitschek, Manfred and Makovoz, Yuly},
  year = {1996},
  series = {Grundlehren der mathematischen Wissenschaften},
  publisher = {Springer Berlin, Heidelberg}
  }

@article{lorentzMetricEntropyApproximation1966,
  title = {Metric {entropy} and {approximation}},
  author = {Lorentz, G. G.},
  year = {1966},
  journal = {Bulletin of the American Mathematical Society},
  volume = {72},
  number = {6},
  pages = {903--937}}

@article{luschgySharpAsymptoticsKolmogorov2004,
  title={Sharp asymptotics of the {K}olmogorov entropy for {G}aussian measures},
  author={Luschgy, Harald and Pag{\`e}s, Gilles},
  journal={Journal of Functional Analysis},
  volume={212},
  number={1},
  pages={89--120},
  year={2004},
  publisher={Elsevier}
}

@book{wainwrightHighDimensionalStatistics2019,
  place={Cambridge}, 
  series={Cambridge Series in Statistical and Probabilistic Mathematics}, 
  title={High-Dimensional Statistics: A Non-Asymptotic Viewpoint}, 
  publisher={Cambridge University Press}, 
  author={Wainwright, Martin J.}, 
  year={2019}, 
  collection={Cambridge Series in Statistical and Probabilistic Mathematics}}

@book{binghamRegularVariation1987,
  title = {Regular {{Variation}}},
  author = {Bingham, N. H. and Goldie, C. M. and Teugels, J. L.},
  year = {1987},
  edition = {1},
  publisher = {Cambridge University Press}
}

@article{kuhnENTROPYNUMBERSDIAGONAL2001,
  title = {Entropy numbers of diagonal operators between vector-valued sequence spaces},
  author = {Kühn, Thomas and Schonbek, Tomas P.},
  year = {2001},
  journal = {Journal of the London Mathematical Society},
  volume = {64},
  number = {3},
  pages = {739--754}
}

@article{kuhnEntropyNumbersEmbeddings2005,
  title = {Entropy numbers of embeddings of weighted {B}esov spaces},
  author = {Kühn, Thomas and Leopold, Hans-Gerd and Sickel, Winfried and Skrzypczak, Leszek},
  year = {2005},
  journal = {Constructive Approximation},
  volume = {23},
  number = {1},
  pages = {61--77}
}

@article{kuhnEntropyNumbersSequence2008,
  title = {Entropy Numbers in Sequence Spaces with an Application to Weighted Function Spaces},
  author = {Kühn, Thomas},
  year = {2008},
  journal = {Journal of Approximation Theory},
  volume = {153},
  number = {1},
  pages = {40--52}
}

@article{kuhn2005entropy,
  title={Entropy numbers of general diagonal operators.},
  author={K{\"u}hn, Thomas},
  journal={Revista Matem{\'a}tica Complutense},
  volume={18},
  number={2},
  pages={479--491},
  year={2005}
}

@article{rogersNoteCoverings1957,
  title = {A Note on Coverings},
  author = {Rogers, Claude A.},
  year = {1957},
  journal = {Mathematika},
  volume = {4},
  number = {1},
  pages = {1--6}
  }

@article{prosserEntropyCapacityCertain1966,
  title = {The ɛ-Entropy and ɛ-Capacity of Certain Time-Varying Channels},
  author = {Prosser, Reese T.},
  year = {1966},
  journal = {Journal of Mathematical Analysis and Applications},
  volume = {16},
  pages = {553--573}
}

@article{donohoDataCompressionHarmonic1998,
  title = {Data {{compression}} and {{harmonic analysis}}},
  author = {Donoho, David L and Vetterli, Martin and DeVore, R A and Daubechies, Ingrid},
  year = {1998},
  journal = {IEEE Transactions on Information Theory},
  volume = {44},
  number = {6},
pages = {2435-2476}
}

@book{hugLecturesConvexGeometry2020,
  title = {Lectures on {{Convex Geometry}}},
  author = {Hug, Daniel and Weil, Wolfgang},
  year = {2020},
  series = {Graduate {{Texts}} in {{Mathematics}}},
  volume = {286},
  publisher = {Springer International Publishing}
  }

@article{rogersCoveringSphereSpheres1963,
  title = {Covering a Sphere with Spheres},
  author = {Rogers, C. A.},
  year = {1963},
  journal = {Mathematika},
  volume = {10},
  number = {2},
  pages = {157--164}
  }

@article{leopoldEmbeddingsEntropyNumbers2000,
  title = {Embeddings and entropy numbers for general weighted sequence spaces: {T}he non-limiting case},
  shorttitle = {Embeddings and {{Entropy Numbers}} for {{General Weighted Sequence Spaces}}},
  author = {Leopold, Hans-Gerd},
  year = {2000},
  journal = {Georgian Mathematical Journal},
  volume = {7},
  number = {4},
  pages = {731--743}
  }

@article{luschgyFunctionalQuantizationGaussian2002,
  title = {Functional Quantization of {{Gaussian}} Processes},
  author = {Luschgy, Harald and Pagès, Gilles},
  year = {2002},
  journal = {Journal of Functional Analysis},
  volume = {196},
  number = {2},
  pages = {486--531}}

@book{rogersbook1964,
  title = {Packing and Covering},
  author = {Rogers, Claude A.},
  year={1964},
  series={Cambridge Tracts in Mathematics}, 
  volume = {54},
  publisher={Cambridge University Press},
  pages={viii+109}
}

@article{edmunds1998entropy,
  title={Entropy numbers of embeddings of {S}obolev spaces in {Z}ygmund spaces},
  author={Edmunds, DE and Netrusov, Yu},
  journal={Studia Mathematica},
  volume={128},
  number={1},
  pages={71--102},
  year={1998},
  publisher={Polska Akademia Nauk. Instytut Matematyczny PAN}
}

@article{marcus1974varepsilon,
  title={The $\varepsilon$-entropy of some compact subsets of $\ell_{p}$},
  author={Marcus, Michael B},
  journal={Journal of Approximation Theory},
  volume={10},
  number={4},
  pages={304--312},
  year={1974},
  publisher={Elsevier}
}

@article{edmunds1992entropy,
  title={Entropy numbers and approximation numbers in function spaces, {II}},
  author={Edmunds, David E and Triebel, Hans},
  journal={Proceedings of the London Mathematical Society},
  volume={3},
  number={1},
  pages={153--169},
  year={1992},
  publisher={Oxford University Press}
}

@book{triebel2006theofctspace,
  title = {Theory of Function Spaces III},
  author = {Triebel, Hans},
  year = {2006},
  series = {Monographs in Mathematics},
  edition = {1},
  pages = {xii+426}
}

@book{groechenigFoundationsTimeFrequencyAnalysis2001,
  title = {Foundations of Time-Frequency Analysis},
  author = {Gröchenig, Karlheinz},
  series = {Applied and Numerical Harmonic Analysis},
  year = {2001},
  publisher = {Birkhäuser Boston}
}

@inproceedings{merucci2006applications,
  title={Applications of interpolation with a function parameter to {L}orentz, {S}obolev and {B}esov spaces},
  author={Merucci, Claude},
  booktitle={Interpolation Spaces and Allied Topics in Analysis: Proceedings of the Conference held in Lund, Sweden, August 29--September 1, 1983},
  pages={183--201},
  year={2006},
  organization={Springer}
}

@article{cobos1987entropy,
  title={Entropy and {L}orentz-{M}arcinkiewicz operator ideals},
  author={Cobos, Fernando},
  journal={Arkiv f{\"o}r Matematik},
  volume={25},
  pages={211--219},
  year={1987},
  publisher={Kluwer Academic Publishers}
}

@article{vybiral2006function,
  title={Function spaces with dominating mixed smoothness},
  author={Vybiral, Jan},
  year={2006},
  journal={Dissertationes Mathematicae},
  volume={436},
  pages={1--73}
}

@book{pietsch1980ideals,
  title={Operator Ideals},
  author={Pietsch, Albrecht},
  year={1980},
  publisher={North-Holland Publishing Company},
  pages={451}}

@book{triebelTheoryFunctionSpaces1983,
  title = {Theory of Function Spaces},
  author = {Triebel, Hans},
  year = {1983},
  publisher = {{Birkhäuser Basel}}}

@article{birman1980quantitative,
  title={Quantitative analysis in {S}obolev imbedding theorems and applications to spectral theory},
  author={Birman, M. S. and Solomjak, M. Z.},
  journal={American Mathematical Society Translations},
  year={1980},
  series={2},
  volume={114}
}

@article{Birman_1967,
year = {1967},
volume = {2},
number = {3},
pages = {295},
author = {Birman, M. S. and Solomjak, M. Z.},
title = {Piecewise-Polynomial Approximations of Functions of the Classes ${W}_p^{\alpha}$},
journal = {Mathematics of the USSR-Sbornik}
}

@article{focm2024,
    author = {Ou, Weigutian and Bölcskei, Helmut},
    title = {Covering numbers for deep {R}e{LU} networks with applications to function approximation and nonparametric regression },
    journal = {submitted to Foundations of Computational Mathematics},
    status = {submitted},
    month = oct,
    year = 2024,
    url = {https://www.mins.ee.ethz.ch/pubs/p/focm2024}
}

@article{edmunds1979embeddings,
  title={Embeddings of {S}obolev spaces},
  author={Edmunds, David E},
  journal={Nonlinear Analysis, Function Spaces and Applications},
  pages={38--58},
  year={1979},
  publisher={BSB BG Teubner Verlagsgesellschaft}
}

@article{Carl_1981, title={Entropy numbers of embedding maps between {B}esov spaces with an application to eigenvalue problems}, volume={90}, journal={Proceedings of the Royal Society of Edinburgh: Section A Mathematics}, author={Carl, Bernd}, year={1981}, pages={63–70}}

@book{bookeigen1987pietsch,
  title={Eigenvalues and s-Numbers},
  publisher={Cambridge University Press},
  author={Pietsch, Albrecht},
  pages={360},
  year={1987}
  }

@book{konigEigenvalueDistributionCompact1986,
  title = {Eigenvalue {{Distribution}} of {{Compact Operators}}},
  author = {König, Hermann},
  editorb = {Gohberg, I.},
  editorbtype = {redactor},
  year = {1986},
  series = {Operator {{Theory}}: {{Advances}} and {{Applications}}},
  volume = {16},
  publisher = {{Birkhäuser Basel}}
}

@article{edmunds1989entropy,
  title={Entropy numbers and approximation numbers in function spaces},
  author={Edmunds, David E and Triebel, Hans},
  journal={Proceedings of the London Mathematical Society},
  volume={3},
  number={1},
  pages={137--152},
  year={1989}
}

@article{grohsPhaseTransitionsRate2021,
  title = {Phase {{transitions}} in {{rate distortion theory}} and {{deep learning}}},
  author = {Grohs, Philipp and Klotz, Andreas and Voigtlaender, Felix},
  year = {2021},
  journal = {Foundations of Computational Mathematics},
  volume ={23}, 
  pages = {329--392}
}

@incollection{feichtinger2020wiener,
  title={Wiener amalgams over {Euclidean} spaces and some of their applications},
  author={Feichtinger, Hans G},
  booktitle={Function spaces},
  pages={123--138},
  year={2020},
  publisher={CRC Press}
}

@book{pinkus1985,
  title = {n-Widths in Approximation Theory},
  author = {Pinkus, Allan},
  year = {1985},
  edition = {1},
  publisher = {Springer Berlin, Heidelberg},
  series={A Series of Modern Surveys in Mathematics}
}

@article{carlInequalitiesEigenvaluesEntropy1980,
  title = {Inequalities between eigenvalues, entropy numbers, and related quantities of compact operators in {B}anach spaces},
  author = {Carl, Bernd and Triebel, Hans},
  year = {1980},
  journal = {Mathematische Annalen},
  volume = {251},
  number = {2},
  pages = {129--133}
}

@article{carl1981entropy,
  title={Entropy numbers of diagonal operators with an application to eigenvalue problems},
  author={Carl, Bernd},
  journal={Journal of Approximation Theory},
  volume={32},
  number={2},
  pages={135--150},
  year={1981}
}

@article{KOSSACZKA2020319,
title = {Entropy numbers of finite-dimensional embeddings},
journal = {Expositiones Mathematicae},
volume = {38},
number = {3},
pages = {319-336},
year = {2020},
author = {Kossaczká, Marta and Vybíral, Jan}
}

\appendix

\section{Proofs of Auxiliary Results}\label{sec:proofaux}

\subsection{Proof of Lemma~\ref{lem: Mixed ellipsoid as a product of balls}}\label{sec:proofMixed ellipsoid as a product of balls}

Arbitrarily fix $x \in {\mathcal{E}}_{2, 2}$ and let
\begin{equation}\label{eq:1334}
\omega_j 
\coloneqq \jrdn^{-\gamma} \left\lceil \left\| \left\{\mu_j^{-1} \jrdn^{\, \gamma} x_{\bar d_{j-1} +i}  \right\}_{i=1}^{ d_j} \right\|_{2}^{2} \right \rceil^{1/2}, 
\quad \text{for } j=1, \dots, k,
\end{equation}
and 
\begin{equation}\label{eq:133422}
\omega_{k+1} 
\coloneqq 
\jrdn^{-\gamma} \left\lceil \jrdn^{\, 2 \gamma} \sum_{j=k+1}^\infty  \sum_{i=1}^{d_j} \frac{ x^2_{\sum_{m=1}^{j-1} d_{m} +i}}{\mu_j^2}   \right \rceil^{1/2},
\end{equation}
where the $\bar d_j$ are as defined in (\ref{eq:defdesbarres}).
It follows from (\ref{eq:1334}) that $\left \| \{x_{i} \}_{i=\bar d_{j-1} +1}^{\bar d_j} \right \|_{2} \leq \omega_j \mu_j$, for $j=1,\dots, k$, and from (\ref{eq:133422}) that $ \left \| \{x_{i} \}_{i=\sum_{m=1}^{j-1} d_{m} +1}^{\sum_{m=1}^{j} d_{m}} \right \|_{2} \leq \omega_{k+1} \mu_j$, for $j\geq k+1$,
hence
\begin{equation*}
    x \in \left(\prod_{j=1}^k \omega_j \,  \mathcal{B} (0 \semcol \mu_j)\right) \times \left(\omega_{k+1} \,\prod_{j=k+1}^\infty   \mathcal{B} (0 \semcol \mu_j)\right).
\end{equation*}
It remains to show that $\omega \coloneqq (\omega_1, \dots, \omega_{k+1})\in\Omega$.
To this end, first note that (\ref{eq:1334}) and \eqref{eq:133422} directly imply $\jrdn^{\, 2\gamma} \omega_j^2 \in \mathbb{N}$, \text{ for all } $j \in \{1, \dots, k+1\}$.
Furthermore, we have 
\begin{align*}
    \left\| \omega \right\|_{2} 
    &= \left(\sum_{j=1}^{k} \jrdn^{-2\gamma}  \left\lceil \left\| \left\{\mu_j^{-1} \jrdn^{\, \gamma} x_{i}  \right\}_{i=\bar d_{j-1}+ 1}^{\bar d_j} \right\|_{2}^{2}  \right \rceil + \jrdn^{-2\gamma} \left\lceil \jrdn^{\, 2 \gamma} \sum_{j=k+1}^\infty  \sum_{i=1}^{d_j} \frac{ x^2_{\sum_{m=1}^{j-1} d_{m} +i}}{\mu_j^2}   \right \rceil \right)^{1/2} \\
    &\leq  \left(\sum_{j=1}^{\infty} \sum_{i=1}^{d_j} \frac{ x^2_{\sum_{m=1}^{j-1} d_{m} +i}}{\mu_j^2} + (k+1) \jrdn^{-2\gamma} \right)^{1/2} \\
    &\leq \|x\|_{2,2, \mu} + \jrdn^{-\gamma} \sqrt{k+1}
    \leq 1 + \jrdn^{-\gamma} \sqrt{k+1},
\end{align*}
where we set $\sum_{m=1}^{0} d_{m} = 0$ and used $\|x\|_{2, 2, \mu}\leq 1$ as a consequence of  $x\in {\mathcal{E}}_{2, 2}$.
This establishes that $\omega \in \Omega$ and thereby concludes the proof.

\subsection{Proof of Lemma~\ref{lem:propertieseffdim}}\label{sec:prooflempropertieseffdim}

Introducing the quantity 
\begin{equation*}
  a_{k,s} 
  \coloneqq \frac{s^k \bar \mu_k^{k}}{\mu_k^{k}}
  = \frac{s^k\prod_{n=1}^{k} \mu_n }{\mu_k^{k}},
  \quad \text{for } k\in \mathbb{N}^*,
\end{equation*}
the upper and lower effective dimensions can be rewritten according to 
\begin{equation*}
\ueffdim_s
= \min \left\{ k \in \mathbb{N}^* \mid N \left(\varepsilon\semcol \mathcal{E}_p, \|\cdot\|_q\right) \leq a_{k,s}\right\},
\end{equation*}
and 
\begin{equation*}
\leffdim_s 
= \max \left\{ k \in \mathbb{N}^* \mid N \left(\varepsilon\semcol \mathcal{E}_p, \|\cdot\|_q\right) >  a_{k,s} \right\}.
\end{equation*}
By the maximality of $\leffdim_s$, we have that 
\begin{equation*}
  N \left(\varepsilon\semcol \mathcal{E}_p, \|\cdot\|_q\right) \leq a_{\leffdim_s+1, s}
\end{equation*}
from which we deduce, by minimality of $\ueffdim_s$, that 
\begin{equation}\label{eq:poijhvuuuiuopp1}
  \ueffdim_s \leq \leffdim_s +1.
\end{equation}
Next, we note that
\begin{equation*}
  a_{k+1, s}
  = \frac{s^{k+1}\prod_{n=1}^{k+1} \mu_n }{\mu_{k+1}^{k+1}}
  = \frac{s^{k+1}\prod_{n=1}^{k} \mu_n }{\mu_{k+1}^{k}}
  \geq \frac{s^k\prod_{n=1}^{k} \mu_n }{\mu_k^{k}} =a_{k, s},
  \quad \text{for all } k\in \mathbb{N}^*.
\end{equation*}
The sequence $\{a_{k,s}\}_{k\in\mathbb{N}^*}$ is therefore non-decreasing.
Combining this observation with 
\begin{equation*}
  a_{\leffdim_s, s}
  < N \left(\varepsilon\semcol \mathcal{E}_p, \|\cdot\|_q\right) 
  \leq a_{\ueffdim_s,s},
\end{equation*}
we obtain 
\begin{equation}\label{eq:poijhvuuuiuopp2}
  \leffdim_s + 1 \leq  \ueffdim_s.
\end{equation}
Combining (\ref{eq:poijhvuuuiuopp1}) and (\ref{eq:poijhvuuuiuopp2}), establishes that $\leffdim_s + 1 = \ueffdim_s$.
Finally, using 
\begin{equation*}
   N \left(\varepsilon\semcol \mathcal{E}_p, \|\cdot\|_q\right) 
  \leq a_{\ueffdim_s, s}
  \quad \text{and} \quad 
  \lim_{\varepsilon \to 0} N \left(\varepsilon\semcol \mathcal{E}_p, \|\cdot\|_q\right) = \infty,
\end{equation*}
it follows that
\begin{equation*}
  \lim_{\varepsilon \to 0} a_{\ueffdim_s, s} = \infty
  \quad \text{and therefore} \quad 
  \lim_{\varepsilon \to 0} {\ueffdim_s}= \infty.
\end{equation*}
%which is the desired result.

\subsection{Proof of Lemma~\ref{lem:1328}}\label{sec:proofoflemma1335}

Fix $\varepsilon >0$. We begin by taking the logarithm in the definition (\ref{eq: definition of underbar n pt1mg3}) of upper effective dimension to obtain
\begin{equation}\label{eq: estimate on logN logp pt1mg3U}
\frac{H \left(\varepsilon\semcol \mathcal{E}_p, \|\cdot\|_q \right)}{\ueffdim_s} 
\leq \log(s) + \frac{1}{\ueffdim_s }\sum_{n=1}^{\ueffdim_s} \log \left ( \frac{\mu_{n}}{\mu_{\ueffdim_s}} \right ).
\end{equation}
Likewise, based on the definition (\ref{eq: definition of bar n pt1mg3fst}) of lower effective dimension, we get
\begin{equation*}
\log(s) + \frac{1}{\leffdim_s}\sum_{n=1}^{\leffdim_s} \log \left ( \frac{\mu_{n}}{\mu_{\leffdim_s}} \right )
< \frac{H \left(\varepsilon\semcol \mathcal{E}_p, \|\cdot\|_q \right)}{\leffdim_s},
\end{equation*}
which, using $\ueffdim_s = \leffdim_s +1$ (see Lemma~\ref{lem:propertieseffdim}), leads to
\begin{equation}\label{eq:ubtrail2}
\log(s)+\frac{1}{\leffdim_s}\sum_{n=1}^{\leffdim_s} \log \left ( \frac{\mu_{n}}{\mu_{\leffdim_s}} \right )
\leq \frac{H \left(\varepsilon\semcol \mathcal{E}_p, \|\cdot\|_q \right)}{\ueffdim_s}\left(1+O_{\varepsilon \to 0} \left(\frac{1}{\ueffdim_s}\right)\right).
\end{equation}
We next analyze the right-hand side of (\ref{eq: estimate on logN logp pt1mg3U}) and the left-hand side of (\ref{eq:ubtrail2}) using the following lemma.

\begin{lemma}\label{lem: Convergence of Cesaro-mean; regularly varying case}
Let $b>0$ and let $\{\mu_n\}_{n\in \mathbb{N}^*}$ be a regularly varying sequence with index $-b$.
Then, we have 
\begin{equation*}
\lim_{N \to \infty}
\frac{1}{N}\sum_{n=1}^N \log \left(\frac{\mu_n}{\mu_N}\right)
=\frac{b}{\ln(2)}.
\end{equation*}
\end{lemma}

\begin{proof}[Proof.]
See Appendix \ref{sec:prooflemconvcesmeanregvar}.
\end{proof}

\noindent
Using Lemma~\ref{lem: Convergence of Cesaro-mean; regularly varying case} in (\ref{eq: estimate on logN logp pt1mg3U}) and (\ref{eq:ubtrail2}), in combination with Lemma~\ref{lem:propertieseffdim}, yields
\begin{equation}\label{eq:nkfndlkvnsfdhjbvjsk}
\lim_{\varepsilon \to 0} \left(\frac{H \left(\varepsilon\semcol \mathcal{E}_p, \|\cdot\|_q \right)}{\ueffdim_s} - \log(s)\right)
=\frac{b}{\ln(2)}.
\end{equation}
We now require the following result.

\begin{lemma}\label{lme:1335199}
Let $b \geq 0$ and let $\varphi$ be a non-increasing regularly varying function with index $-b$. 
Let the sequences $\{a_{k}\}_{k \in \mathbb{N}^*}$ and $\{b_{k}\}_{k \in \mathbb{N}^*}$ be such that 
\begin{equation*}
\begin{cases}
\lim_{k \to \infty} a_{k} = \infty, \\
0 < \liminf_{k \to \infty} b_{k} \leq \limsup_{k \to \infty} b_{k} < \infty.
\end{cases}
\end{equation*}
Then, we have 
\begin{equation*}
\liminf_{k \to \infty} \frac{\varphi(a_k b_k)}{\varphi(a_k)} 
=\left(\limsup_{k\to\infty} b_k\right)^{-b}
\quad \text{and} \quad 
\limsup_{k \to \infty} \frac{\varphi(a_k b_k)}{\varphi(a_k)} 
= \left(\liminf_{k \to \infty} b_{k}\right)^{-b}.
\end{equation*}
\end{lemma}

\begin{proof}[Proof.]
  See Appendix \ref{sec:prooflemgivenameplssssempai}.
\end{proof}

\noindent
Next, consider a sequence $\{\varepsilon_k\}_{k \in \mathbb{N}^*}$ satisfying $\varepsilon_k \to_{k\to\infty} 0$ and let $\{\ueffdim_s^{(k)}\}_{k \in \mathbb{N}^*}$ be the associated sequence of upper effective dimensions. 
Note that by \cite[Theorem 1.9.5 case (ii)]{binghamRegularVariation1987}, the step function $\phi$ is regularly varying with index $-b$.
Therefore, application of Lemma~\ref{lme:1335199} to the sequences 
\begin{equation*}
  a_{k} \coloneqq \ueffdim_s^{(k)}
  \quad \text{and} \quad
  b_{k} \coloneqq \frac{H \left(\varepsilon_k \semcol \mathcal{E}_p, \|\cdot\|_q \right)}{\ueffdim_s^{(k)}}
\end{equation*}
yields, when combined with (\ref{eq:nkfndlkvnsfdhjbvjsk}), 
\begin{equation*}
\limsup_{k\to\infty} \frac{\phi \left(H \left(\varepsilon_k \semcol \mathcal{E}_p, \|\cdot\|_q \right)\right)}{\phi \left(\ueffdim_s^{(k)}\right)} 
= \left(\liminf_{k\to\infty} \frac{H \left(\varepsilon_k \semcol \mathcal{E}_p, \|\cdot\|_q \right)}{\ueffdim_s^{(k)}}\right)^{-b}
= \left(\frac{b}{\ln(2)} +\log(s)\right)^{-b}
\end{equation*}
and
\begin{equation*}
\liminf_{k\to\infty} \frac{\phi \left(H \left(\varepsilon_k \semcol \mathcal{E}_p, \|\cdot\|_q \right)\right)}{\phi \left(\ueffdim_s^{(k)}\right)} 
= \left(\limsup_{k\to\infty} \frac{H \left(\varepsilon_k \semcol \mathcal{E}_p, \|\cdot\|_q \right)}{\ueffdim_s^{(k)}}\right)^{-b}
= \left(\frac{b}{\ln(2)} +\log(s)\right)^{-b}.
\end{equation*}
In particular, the limit exists, and we simply write
\begin{equation}\label{eq:1336}
\lim_{k\to\infty} \frac{\phi \left(H \left(\varepsilon_k \semcol \mathcal{E}_p, \|\cdot\|_q \right)\right)}{\phi \left(\ueffdim_s^{(k)}\right)} 
= \left(\frac{b}{\ln(2)} +\log(s)\right)^{-b}.
\end{equation}
Now, combining $\phi(n)=\mu_{n+1}$ and \cite[Lemma~1.9.6]{binghamRegularVariation1987} with (\ref{eq:1336}), we obtain
\begin{align}
\lim_{k\to\infty} \frac{\mu_{\ueffdim_s^{(k)}}}{\phi\left(H \left(\varepsilon_k\semcol \mathcal{E}_p, \|\cdot\|_q \right)\right)} 
&= \lim_{k\to\infty} \frac{\mu_{\ueffdim_s^{(k)}}}{\phi \left(\ueffdim_s^{(k)} \right)} \cdot \frac{\phi \left(\ueffdim_s^{(k)}\right)}{\phi\left(H \left(\varepsilon_k\semcol \mathcal{E}_p, \|\cdot\|_q \right)\right)}\label{eq:1336201}\\
&= \lim_{k\to\infty} \frac{\mu_{\ueffdim_s^{(k)}}}{\mu_{\ueffdim_s^{(k)}+1}} \cdot \lim_{k\to\infty} \frac{\phi \left(\ueffdim_s^{(k)} \right)}{\phi\left(H \left(\varepsilon_k\semcol \mathcal{E}_p, \|\cdot\|_q \right)\right)} \nonumber\\
&= \left(\frac{b}{\ln(2)} +\log(s)\right)^{b}.\label{eq:1336201nnn}
\end{align}

As the sequence $\{\varepsilon_k\}_{k \in \mathbb{N}^*}$ satisfies $\varepsilon_k \to_{k\to\infty} 0 $ but is arbitrary otherwise,  (\ref{eq:1336201})--(\ref{eq:1336201nnn}) implies 
\begin{equation*}
\lim_{\varepsilon \to 0} \frac{\mu_{\ueffdim_s}}{\phi \left(H \left(\varepsilon\semcol \mathcal{E}_p, \|\cdot\|_q \right)\right)}
  =\left(\frac{b}{\ln(2)} +\log(s)\right)^b.
\end{equation*}
The limit 
\begin{equation*}
\lim_{\varepsilon \to 0} \frac{\mu_{\leffdim_s}}{\phi \left(H \left(\varepsilon\semcol \mathcal{E}_p, \|\cdot\|_q \right)\right)}
  =\left(\frac{b}{\ln(2)} +\log(s)\right)^b
\end{equation*}
can be established analogously.

\subsection{Proof of Lemma~\ref{lem:1327}}\label{sec:proofintermediary bound1337}

We start with the lower bound (\ref{eq:lkjbkjbkjbckjbkebzhzzzqdfguyuy}) in Theorem~\ref{thm:profinitethm} to get
\begin{equation}\label{eq:hidogzhvjcabzefjhvbjsh}
N \left(\varepsilon\semcol \mathcal{E}_p^{\ueffdim_s}, \|\cdot\|_q\right)
\geq \varepsilon^{-\ueffdim_s} \left(V_{p,q, \ueffdim_s}\right)^{\ueffdim_s} \bar \mu_{\ueffdim_s}^{\ueffdim_s}.
\end{equation}
Combining (\ref{eq:hidogzhvjcabzefjhvbjsh}) with %the standard inequality 
\begin{equation*}
N \left(\varepsilon\semcol \mathcal{E}_p, \|\cdot\|_q\right)
\geq N \left(\varepsilon\semcol \mathcal{E}_p^{\ueffdim_s}, \|\cdot\|_q \right),
\end{equation*}
we obtain
\begin{equation}\label{eq:lkjhgfghqdhjksdcq}
N \left(\varepsilon\semcol \mathcal{E}_p, \|\cdot\|_q \right)
\geq \varepsilon^{-\ueffdim_s} \left(V_{p,q, \ueffdim_s}\right)^{\ueffdim_s} \bar \mu_{\ueffdim_s}^{\ueffdim_s}.
\end{equation}
Moreover, from the definition of $\ueffdim_s$ in (\ref{eq: definition of underbar n pt1mg3}), one has 
\begin{equation*}
N \left(\varepsilon\semcol \mathcal{E}_p, \|\cdot\|_q \right) 
\leq \frac{s^{\ueffdim_s} \bar \mu_{\ueffdim_s}^{\ueffdim_s}}{\mu_{\ueffdim_s}^{\ueffdim_s}},
\end{equation*}
which, when combined with (\ref{eq:lkjhgfghqdhjksdcq}), yields the desired result according to
\begin{equation*}
s \varepsilon 
\geq V_{p,q, \ueffdim_s} \mu_{\ueffdim_s}.
\end{equation*}

\subsection{Proof of Lemma~\ref{lem:klmjhgfddddfhqsdfg}}\label{sec:prooflemklmjhgfddddfhqsdfg}

\noindent
With the sequence $\{\alpha_n\}_{n\in \mathbb{N}^*}$ as defined in statement (ii) of  
Lemma~\ref{lem:lemboundtriangleineqlike} when $p/(pb+1) < q < p$, and $\alpha_n=\mu_{n+1}$, for all $n\in\mathbb{N}^*$, when $p=q$,
and general $s > 1$, 
let us
first consider the case
\begin{equation}\label{eq:rezalkjhugyfss31}
\alpha_{\leffdim_s} \left(1+\frac{2^{1/p}}{\leffdim_s}\right) < \varepsilon,
\end{equation}
as $\varepsilon \rightarrow 0$.
Next, with a view towards application of Lemma~\ref{lem:lemboundtriangleineqlike}, consider the finite lattice 
\begin{equation*}% \label{eq:defomegacasepeqq31}
  \Omega_{s} 
  \coloneqq 
  \begin{cases}
    {(1,1)}, \quad &\text{if } p=\infty,\\[.2cm]
    \left\{\omega \in \mathbb{R}_+^{2} \mid \| \omega \|_{p} \leq 1+ \frac{2^{1/p}}{\leffdim_s} \text{ and } (\leffdim_s \omega_j)^p \in \mathbb{N}, \text{ for } j \in \{1,2\}\right\},
     &\text{otherwise,}
  \end{cases}
\end{equation*}
and note that
\begin{equation*}% \label{eq:joojdosnojzenjvn31}
  \# \Omega_{s} 
  \begin{cases}
    = 1, & \text{ if } p = \infty, \\
    \leq K \leffdim_s^{2p},
    \quad &\text{for some constant } K>0, \text{ if } p \neq \infty.
  \end{cases}
\end{equation*}%}
Using arguments similar to those employed in the proof of Lemma~\ref{lem: Mixed ellipsoid as a product of balls} together with the Minkowski inequality, 
one can establish the following inclusion relation according to 
Definition \ref{def:blockdecomp}, with $k=1$, $\bar{d}_1= \leffdim_s$, and $\bar{d}_{2}=\infty$,
\begin{equation*}
\mathcal{E}_{p} 
\subseteq \bigcup_{(\omega_1, \, \omega_2) \in \Omega_{s}}  \left(\omega_1 \mathcal{E}_p^{\leffdim_s}\right)\times \left(\omega_2 \mathcal{E}_p^{[2]}\right).
\end{equation*}
Next, note that (\ref{eq:rezalkjhugyfss31}) guarantees the existence of a $\rho > 0$ such that 
\begin{equation}\label{eq:profiniteeqqaedff31}
\varepsilon = \max_{(\omega_1, \omega_2) \in \Omega_{s}} \,  \left[(\omega_1\, \rho)^q  + (\omega_{2}\, \alpha_{\leffdim_s})^q \right]^{1/q},
\end{equation}
with the usual modification when $q=\infty$.
Application of case (i) in Lemma~\ref{lem:lemboundtriangleineqlike} when $p=q$, respectively case (ii) when $p/(pb+1) < q < p$, now yields
\begin{equation}\label{eqetctuvreaaa31}
  N \left(\varepsilon \semcol \mathcal{E}_p, \|\cdot\|_q\right)
  \leq (\# \Omega_{s}) \, N \left(\rho\semcol \mathcal{E}_p^{\leffdim_s}, \|\cdot\|_q\right).
\end{equation}

From now on we shall work with an $s=s^*>1$ which satisfies
\begin{equation}\label{eq:kojhudsiqhde31}
  \rho \leq ({s^*-1}) \,  \leffdim_{s^*}^{(1/q-1/p)} \mu_{\leffdim_{s^*}},
\end{equation}
but is arbitrary otherwise.
Such an $s^*$ exists as the left-hand side in (\ref{eq:kojhudsiqhde31}) is upper-bounded by $\varepsilon \leffdim_{s^*}$, $\leffdim_{s^*}$ decreases for increasing $s^*$, and the right-hand side approaches infinity as $s^* \rightarrow \infty$.
In particular, (\ref{eq:kojhudsiqhde31}) ensures that the hypothesis (\ref{eq:extraassumptionub2}) of the case \ref{item:iifinitedim} of Theorem~\ref{thm:profinitethm} is satisfied with $\eta = s^* - 1$, $d=\leffdim_{s^*}$, and hence allows to conclude that
\begin{eqnarray}
  \rho 
& \leq & \kappa(\leffdim_{s^*})\,s^* \left(N \left(\rho\semcol \mathcal{E}_p^{\leffdim_{s^*}}, \|\cdot\|_q\right)-1\right)^{-1/\leffdim_{s^*}} \leffdim_{s^*}^{(1/q-1/p)} \bar \mu_{\leffdim_{s^*}} \label{eq:lkanekrjvenzkjfgnvjrD11} \\
& = & \bar \kappa(\leffdim_{s^*})\,s^* N \left(\rho\semcol \mathcal{E}_p^{\leffdim_{s^*}}, \|\cdot\|_q\right)^{-1/\leffdim_{s^*}} \leffdim_{s^*}^{(1/q-1/p)} \bar \mu_{\leffdim_{s^*}}, \label{eq:lkanekrjvenzkjfgnvjrD12}
\end{eqnarray}
with 
\begin{equation*}
  \bar \kappa(\leffdim_{s^*})
  = 1+ o_{\leffdim_{s^*}\to \infty}(1)
  =1+ o_{\varepsilon \to 0}(1).
\end{equation*}
Inserting (\ref{eqetctuvreaaa31}) into (\ref{eq:lkanekrjvenzkjfgnvjrD11})--(\ref{eq:lkanekrjvenzkjfgnvjrD12}), we get
\begin{equation*}
\rho 
\leq \frac{\bar \kappa(\leffdim_{s^*})\,s^* (\# \Omega_{s^*})^{1/\leffdim_{s^*}}  \leffdim_{s^*}^{(1/q-1/p)} \bar \mu_{\leffdim_{s^*}}}{N \left(\varepsilon\semcol \mathcal{E}_p, \|\cdot\|_q\right)^{1/\leffdim_{s^*}}},
\end{equation*}
which, using the definition of lower effective dimension in (\ref{eq: definition of bar n pt1mg3fst}), allows us to infer that 
\begin{equation}\label{eq:kojhudsiqhde231}
\rho 
\leq \tilde \kappa(\leffdim_{s^*}) \, \leffdim_{s^*}^{(1/q-1/p)} \mu_{\leffdim_{s^*}},
\end{equation}
with
\begin{equation}\label{eq:kappacorrectscaleeps31}
\tilde \kappa(\leffdim_{s^*}) 
\coloneqq \bar \kappa(\leffdim_{s^*}) (\# \Omega_{s^*})^{1/\leffdim_{s^*}}
= 1 + o_{\varepsilon \to 0}(1) .
\end{equation}
Next, comparing (\ref{eq:kojhudsiqhde31}) and (\ref{eq:kojhudsiqhde231}), we conclude that $s^*$ can be chosen to satisfy
\begin{equation*}
  ({s^*-1}) = \tilde\kappa(\leffdim_{s^*}),
\quad \text{or equivalently,} \quad
s^* = 1+\tilde\kappa(\leffdim_{s^*}) = 2 + o_{\varepsilon \to 0}(1) .
\end{equation*}
Combining (\ref{eq:profiniteeqqaedff31}) with (\ref{eq:kojhudsiqhde231}) yields
\begin{equation}\label{eq:eefdksjhbnjdhvbjnbnhbsdvffdfee}
\varepsilon 
\leq \max_{(\omega_1, \omega_2) \in \Omega_{s^*}} \,  \left[(\omega_1\, \tilde \kappa(\leffdim_{s^*}) \, \leffdim_{s^*}^{(1/q-1/p)} \, \mu_{\leffdim_{s^*}})^q  + (\omega_{2}\, \alpha_{\leffdim_{s^*}})^q \right]^{1/q}.
\end{equation}
In the case $p/(pb+1) < q < p$, we can use (\ref{eq:ijankjvbhjbvjqookknsd}) to get 
\begin{equation*}
\varepsilon\leq \max_{(\omega_1, \omega_2) \in \Omega_{s^*}} \, \left[\omega_1^q + \omega_2^q\, \left(\frac{b}{\frac{1}{q}-\frac{1}{p}}  -1 \right)^{\frac{q}{p}-1}\right]^{1/q} \, \leffdim_{s^*}^{(1/q-1/p)} \mu_{\leffdim_{s^*}} \left(1+ o_{\varepsilon \to 0}(1)\right),  
\end{equation*}
which, upon application of Hölder's inequality, with parameters $p/q$ and $(1-q/p)^{-1}$, yields
\begin{align*}
    \left[\omega_1^q + \omega_2^q\, \left(\frac{b}{\frac{1}{q}-\frac{1}{p}}  -1 \right)^{\frac{q}{p}-1}\right]^{1/q}
    &\leq \left[\omega_1^p + \omega_2^p\right]^{1/p} \, \left[1+\left(\frac{b}{\frac{1}{q}-\frac{1}{p}}  -1 \right)^{-1}\right]^{(\frac{1}{q}-\frac{1}{p})} \\
    &= \left[\omega_1^p + \omega_2^p\right]^{1/p} \, \left(\frac{b}{b+\frac{1}{p}-\frac{1}{q}}   \right)^{(\frac{1}{q}-\frac{1}{p})}.
\end{align*}
Recalling that $ \| \omega \|_{p} \leq 1+ o_{\varepsilon \to 0}(1)$, for all $\omega \in \Omega_{s^*}$, and
letting 
\begin{equation*}
\gamma_{p,q,b}
\coloneqq   \left(\frac{b}{b+\frac{1}{p}-\frac{1}{q}}   \right)^{(\frac{1}{q}-\frac{1}{p})},
\end{equation*}
we obtain
\begin{equation}\label{eq:jskkjjbdcbjhsjtfgbvjdsbqjjhbqf}
    \varepsilon 
    \leq \gamma_{p,q,b} \, \leffdim_{s^*}^{(1/q-1/p)} \mu_{\leffdim_{s^*}} \left(1+ o_{\varepsilon \to 0}(1)\right).
\end{equation}
In the case $p=q$, the bound (\ref{eq:jskkjjbdcbjhsjtfgbvjdsbqjjhbqf}) follows directly from (\ref{eq:eefdksjhbnjdhvbjnbnhbsdvffdfee}) by noting that $\alpha_{\leffdim_{s^*}}=\mu_{\leffdim_{s^*+1}}$, $\mu_{\leffdim_{s^*+1}}=\mu_{\leffdim_{s^*}}(1+o_{\varepsilon \to 0}(1))$ thanks to \cite[Lemma 1.9.6] {binghamRegularVariation1987}, $\gamma_{p,q,b}=1$, $\leffdim_{s^*}^{(1/q-1/p)}=1$, and $ \| \omega \|_{q} = \| \omega \|_{p} \leq 1+ o_{\varepsilon \to 0}(1)$, for all $\omega \in \Omega_{s^*}$. Putting everything together, we have established the existence of $s^* > 1$ satisfying
\begin{equation*}
  s^* = 2 + o_{\varepsilon \to 0}(1)  \quad
  \text{ and such that } \quad
  \varepsilon \leq \gamma_{p,q, b} \,  \leffdim_{s^*}^{(1/q-1/p)}\, \mu_{\leffdim_{s^*}}\left(1+ o_{\varepsilon \to 0}(1)\right),
\end{equation*}
thereby concluding the proof in the case
%\textcolor{blue}{
$$
\varepsilon > \alpha_{\leffdim_{s^*}}\left(1+\frac{2^{1/p}}{\leffdim_{s^*}}\right), \quad \text{as } \varepsilon \rightarrow 0.
$$
Finally, for
$$
\varepsilon \leq \alpha_{\leffdim_{s^*}}\left(1+\frac{2^{1/p}}{\leffdim_{s^*}}\right), \quad \text{as } \varepsilon \rightarrow 0,
$$%}
the proof, in the case $p/(pb+1)<q<p$, follows directly upon recalling the definition (\ref{eq:ijankjvbhjbvjqookknsd}) of $\{\alpha_n\}_{n\in \mathbb{N}^*}$ and noting that
\begin{equation*}
  \left(\frac{b}{\frac{1}{q}-\frac{1}{p}}  -1 \right)^{(\frac{1}{p}-\frac{1}{q})} 
  = \left(\frac{\frac{1}{q}-\frac{1}{p}}{b+\frac{1}{p}-\frac{1}{q}} \right)^{(\frac{1}{q}-\frac{1}{p})} 
  < \left(\frac{b}{b+\frac{1}{p}-\frac{1}{q}} \right)^{(\frac{1}{q}-\frac{1}{p})} = \gamma_{p,q, b}.
\end{equation*}
For $p=q$, we again use $\gamma_{p,q,b}=1$, $\leffdim_{s^*}^{(1/q-1/p)}=1$, $\alpha_{\leffdim_{s^*}}=\mu_{\leffdim_{s^*+1}}$, and $\mu_{\leffdim_{s^*+1}}=\mu_{\leffdim_{s^*}}(1+o_{\varepsilon \to 0}(1))$ thanks to \cite[Lemma 1.9.6]{binghamRegularVariation1987}.

\subsection{Proof of Lemma~\ref{lem:klmjhgfddddfhqsdfgEEE}}\label{sec:prooflemklmjhgfddddfhqsdfgEEE}

The proof strategy essentially follows that of the proof of Lemma~\ref{lem:klmjhgfddddfhqsdfg}. We provide the detailed arguments as the required modifications are not obvious.

With the sequence $\{\alpha_n\}_{n\in \mathbb{N}^*}$ as defined in statement (iii) of  
Lemma~\ref{lem:lemboundtriangleineqlike} and general $s > 1$, let us assume, with a view towards contradiction, that
\begin{equation}\label{eq:rezalkjhugyfss31EEE}
\alpha_{\leffdim_s} \left(1+\frac{2^{1/p}}{\leffdim_s}\right) < \varepsilon,
\end{equation}%}
as $\varepsilon \rightarrow 0$.
Next, with the application of Lemma~\ref{lem:lemboundtriangleineqlike} in mind, we define the finite lattice 
\begin{equation}\label{eq:defomegacasepeqq31EEE}
  \Omega_{s} 
  \coloneqq 
  \begin{cases}
    {(1,1)}, \quad &\text{if } p=\infty,\\[.2cm]
    \left\{\omega \in \mathbb{R}_+^{2} \mid \| \omega \|_{p} \leq 1+ \frac{2^{1/p}}{\leffdim_s} \text{ and } (\leffdim_s \omega_j)^p \in \mathbb{N}, \text{ for } j \in \{1,2\}\right\},
     &\text{otherwise},
  \end{cases}
\end{equation}
and note that
\begin{equation*}
  \# \Omega_{s} 
  \begin{cases}
    = 1, & \text{ if } p = \infty, \\
    \leq K \leffdim_s^{2p},
    \quad &\text{for some constant } K>0, \text{ if } p \neq \infty.
  \end{cases}
\end{equation*}%}
Using arguments similar to those employed in the proof of Lemma~\ref{lem: Mixed ellipsoid as a product of balls} together with the Minkowski inequality, 
one can establish the following inclusion relation according to 
Definition \ref{def:blockdecomp}, with $k=1$, $\bar{d}_1= \leffdim_s$, and $\bar{d}_{2}=\infty$,
\begin{equation*}
\mathcal{E}_{p} 
\subseteq \bigcup_{(\omega_1, \, \omega_2) \in \Omega_{s}}  \left(\omega_1 \mathcal{E}_p^{\leffdim_s}\right)\times \left(\omega_2 \mathcal{E}_p^{[2]}\right).
\end{equation*}
Next, note that the assumption (\ref{eq:rezalkjhugyfss31EEE}) guarantees the existence of a $\rho > 0$ such that 
\begin{equation}\label{eq:profiniteeqqaedff31EEE}
\varepsilon = \max_{(\omega_1, \omega_2) \in \Omega_{s}} \,  \left[(\omega_1\, \rho)^q  + (\omega_{2}\, \alpha_{\leffdim_s})^q \right]^{1/q}.
\end{equation}
Application of Lemma~\ref{lem:lemboundtriangleineqlike} (iii) now yields
\begin{equation}\label{eqetctuvreaaa31EEE}
  N \left(\varepsilon \semcol \mathcal{E}_p, \|\cdot\|_q\right)
  \leq (\# \Omega_{s}) \, N \left(\rho\semcol \mathcal{E}_p^{\leffdim_s}, \|\cdot\|_q\right).
\end{equation}

From here on we shall work with an $s=s^*>1$ which satisfies
\begin{equation}\label{eq:kojhudsiqhde31EEE}
  \rho \leq ({s^*-1}) \,  \leffdim_{s^*}^{(1/q-1/p)} \mu_{\leffdim_{s^*}},
\end{equation}
but is arbitrary otherwise.
Such an $s^*$ exists as the left-hand side in (\ref{eq:kojhudsiqhde31EEE}) is upper-bounded by $\varepsilon \leffdim_{s^*}$, $\leffdim_{s^*}$ decreases for increasing $s^*$, and the right-hand side approaches infinity as $s^* \rightarrow \infty$.
In particular, (\ref{eq:kojhudsiqhde31EEE}) ensures that the hypothesis (\ref{eq:extraassumptionub2}) of the case \ref{item:iifinitedim} of Theorem~\ref{thm:profinitethm} is satisfied with $\eta = s^* - 1$, $d=\leffdim_{s^*}$, and hence allows to conclude that
\begin{eqnarray}
  \rho 
& \leq & \kappa(\leffdim_{s^*})\,s^* \left(N \left(\rho\semcol \mathcal{E}_p^{\leffdim_{s^*}}, \|\cdot\|_q\right)-1\right)^{-1/\leffdim_{s^*}} \leffdim_{s^*}^{(1/q-1/p)} \bar \mu_{\leffdim_{s^*}} \label{eq:lkanekrjvenzkjfgnvjrD11EEE} \\
& = & \bar \kappa(\leffdim_{s^*})\,s^* N \left(\rho\semcol \mathcal{E}_p^{\leffdim_{s^*}}, \|\cdot\|_q\right)^{-1/\leffdim_{s^*}} \leffdim_{s^*}^{(1/q-1/p)} \bar \mu_{\leffdim_{s^*}}, \label{eq:lkanekrjvenzkjfgnvjrD12EEE}
\end{eqnarray}
where 
\begin{equation*}
  \bar \kappa(\leffdim_{s^*})
  = 1+ o_{\leffdim_{s^*}\to \infty}(1)
  =1+ o_{\varepsilon \to 0}(1).
\end{equation*}
Inserting (\ref{eqetctuvreaaa31EEE}) into (\ref{eq:lkanekrjvenzkjfgnvjrD11EEE})--(\ref{eq:lkanekrjvenzkjfgnvjrD12EEE}), we get
\begin{equation*}
\rho 
\leq \frac{\bar \kappa(\leffdim_{s^*})\,s^* (\# \Omega_{s^*})^{1/\leffdim_{s^*}}  \leffdim_{s^*}^{(1/q-1/p)} \bar \mu_{\leffdim_{s^*}}}{N \left(\varepsilon\semcol \mathcal{E}_p, \|\cdot\|_q\right)^{1/\leffdim_{s^*}}},
\end{equation*}
which, using the definition of lower effective dimension in (\ref{eq: definition of bar n pt1mg3fst}), allows us to conclude that 
\begin{equation}\label{eq:kojhudsiqhde231EEE}
\rho 
\leq \tilde \kappa(\leffdim_{s^*}) \, \leffdim_{s^*}^{(1/q-1/p)} \mu_{\leffdim_{s^*}},
\end{equation}
with
\begin{equation}\label{eq:kappacorrectscaleeps31EEE}
\tilde \kappa(\leffdim_{s^*}) 
\coloneqq \bar \kappa(\leffdim_{s^*}) (\# \Omega_{s^*})^{1/\leffdim_{s^*}}
= 1 + o_{\varepsilon \to 0}(1) .
\end{equation}
Next, comparing (\ref{eq:kojhudsiqhde31EEE}) and (\ref{eq:kojhudsiqhde231EEE}), we conclude that $s^*$ can be chosen to satisfy
\begin{equation}\label{eq:kqlvflmqlm,fksvendfnvEEE}
  ({s^*-1}) = \tilde\kappa(\leffdim_{s^*}),
\quad \text{or equivalently,} \quad
s^* = 1+\tilde\kappa(\leffdim_{s^*}) = 2 + o_{\varepsilon \to 0}(1) .
\end{equation}
Combining (\ref{eq:profiniteeqqaedff31EEE}) with (\ref{eq:kojhudsiqhde231EEE}) and using $1/q-1/p=b$ owing to $q=p/(pb+1)$, yields
\begin{align*}
\varepsilon 
\leq \max_{(\omega_1, \omega_2) \in \Omega_{s^*}} \,  \left[(\omega_1\, \tilde \kappa(\leffdim_{s^*}) \, \leffdim_{s^*}^{b} \, \mu_{\leffdim_{s^*}})^q  + (\omega_{2}\, \alpha_{\leffdim_{s^*}})^q \right]^{1/q}.
\end{align*}
Now, we know, by Karamata's theorem (see \cite[Proposition 1.5.9b]{binghamRegularVariation1987}), that 
\begin{equation*}
 d \, \mu^{1/b}_{d}
    = o_{d\to\infty} \left(\alpha^{1/b}_d \right) \quad \text{and hence}   \quad 
    d^{b} \, \mu_{d}
    = o_{d\to\infty} \left(\alpha_d \right).
\end{equation*}
Consequently, in the regime $\varepsilon\to 0$, 
\begin{equation*}
    \varepsilon
    \leq \max_{(\omega_1, \omega_2) \in \Omega_{s^*}}[\omega_{2}\, \alpha_{\leffdim_{s^*}}]
    \leq \left(1+\frac{2^{1/p}}{\leffdim_{s^*}}\right)\alpha_{\leffdim_{s^*}},
\end{equation*}
which stands in contradiction to (\ref{eq:rezalkjhugyfss31EEE}), whenever $s^*$ is as in (\ref{eq:kqlvflmqlm,fksvendfnvEEE}).
This concludes the proof.

\subsection{Proof of Lemma~\ref{lem:relatpqtwomuneps}}\label{sec:proofkljhgfdwghfdfudzzz}

Let $\gamma \geq 2$, but otherwise arbitrary. Fix $\varepsilon > 0$, let $\tilde \varepsilon \coloneqq \varepsilon (1+\varepsilon^{(\gamma -1)/b})^{-1}$, define
$\nu \in \mathbb{N}^*$ to be the integer satisfying 
\begin{equation}\label{eq:rfhabhgrebzksbhdqvjoaret}
  \mu_{\nu+1} \leq \tilde \varepsilon < \mu_\nu,
\end{equation}
and take $\varepsilon$ small enough for $\nu \geq 81$ to hold.
With a view towards application of Theorem~\ref{thm:bound on mixed norm}, let us further set 
 $k \coloneqq \lfloor \sqrt{\nu} \rfloor$,   
 $d_n \coloneqq \lfloor \sqrt{\nu} \rfloor$, for $n \in \mathbb{N}^* \setminus \{k\}$, and $d_k \coloneqq \nu - \bar d_{k-1} \geq \lfloor \sqrt{\nu} \rfloor$, where $\bar{d}_{j}, j \in \{0,\dots,k+1\}$, is as defined in \eqref{eq:defdesbarres}.
In particular, we observe that $\bar d_k = \nu $ and $d_n \geq 9$, for all $n\in\mathbb{N}^*$.
It will also be convenient to define the semi-axes $\tilde \mu_n \coloneqq \mu_{(\sum_{j=1}^{n-1} d_{j})+1}$, for $n \in \mathbb{N}^*$, so that, owing to
$\{ \mu_n\}_{n\in\mathbb{N}^*}$ being non-increasing, we get 
\begin{equation}\label{eq:bhijavzohbeozzze}
  \mu_{(\sum_{j=1}^{n} d_{j})+i_1} \leq \tilde \mu_{n+1} \leq \mu_{(\sum_{j=1}^{n-1} d_{j})+i_2},
\end{equation}
for all $n \in \mathbb{N}^*$,  with $i_1\in\{1, \dots, d_{n+1}\}$ and $i_2\in\{1, \dots, d_{n}\}$.
We now consider the mixed ellipsoid $\tilde{\mathcal{E}}_{2, 2}$ with dimensions $\{ d_n\}_{n\in\mathbb{N}^*}$ and semi-axes $\{\tilde \mu_n\}_{n\in\mathbb{N}^*}$.
It follows from $\tilde{\mu}_{1}=\mu_{1}$ and \eqref{eq:bhijavzohbeozzze}, that ${\mathcal{E}}_{2} \subseteq \tilde{\mathcal{E}}_{2, 2}$, which, in turn, allows us to conclude that
\begin{equation}\label{eq:bjazrbghaiebhvjbqhjzeb1}
  N \left(\rho\semcol {\mathcal{E}}_{2}, \|\cdot\|_2 \right) 
  \leq N \left(\rho\semcol \tilde{\mathcal{E}}_{2, 2}, \|\cdot\|_2 \right),
  \quad \text{for all } \rho >0.
\end{equation}
Next, we note that
\begin{equation*}
  \tilde \mu_{k+1} =  \mu_{\bar d_k+1} = \mu_{\nu +1}
  \leq \tilde\varepsilon < \mu_{\nu} = \mu_{\bar d_k} \leq \tilde \mu_{k}.
\end{equation*}
Application of Theorem~\ref{thm:bound on mixed norm} now yields
\begin{equation}\label{eq:lkjhgffazfaghkfrehjkzvhv111}
  N \left(\tilde\varepsilon_\gamma\semcol \tilde{\mathcal{E}}_{2, 2}, \|\cdot\|_2 \right)^{1/\bar d_k}
  \leq  \frac{\kappa_k \prod_{j=1}^k \tilde{\mu}_j^{d_j/\bar d_k}}{\tilde \varepsilon_\gamma},
\end{equation}
with
\begin{equation}\label{eq:zsdtdssjhkgzqogeazjv}
  \tilde\varepsilon_\gamma
  \coloneqq \tilde\varepsilon \left({1+\bar d_k^{-\gamma}\sqrt{k+1}}\right).
\end{equation}
We observe that
\begin{equation*}
  \tilde \varepsilon_\gamma 
  \leq \varepsilon \, \frac{1+\sqrt{2} \, \nu^{\frac{1}{4}-\gamma}}{1+\varepsilon^{(\gamma -1)/b}}
  \leq  \varepsilon \, \frac{1+C\varepsilon^{(\gamma-\frac{1}{4})/b}}{1+\varepsilon^{(\gamma -1)/b}}
  \leq \varepsilon,
\end{equation*}
for some constant $C>0$ and $\varepsilon$ small enough. Hence,
\begin{equation}\label{eq:bjazrbghaiebhvjbqhjzeb2}
   N \left(\varepsilon\semcol \tilde{\mathcal{E}}_{2, 2}, \|\cdot\|_2 \right)
   \leq  N \left(\tilde\varepsilon_\gamma\semcol \tilde{\mathcal{E}}_{2, 2}, \|\cdot\|_2 \right).
\end{equation}
Combining (\ref{eq:bjazrbghaiebhvjbqhjzeb1}) and (\ref{eq:bjazrbghaiebhvjbqhjzeb2}), followed by application of
(\ref{eq:lkjhgffazfaghkfrehjkzvhv111}), yields
\begin{equation}\label{eq:lkjhgffazfaghkfrehjkzvhv11}
  N \left(\varepsilon\semcol {\mathcal{E}}_{2}, \|\cdot\|_2 \right)^{1/\bar d_k}
  \leq \frac{\kappa_k \prod_{j=1}^k \tilde{\mu}_j^{d_j/\bar d_k}}{\tilde \varepsilon_\gamma}.
\end{equation}
Starting from (\ref{eq:zsdtdssjhkgzqogeazjv}) and using the left inequality in (\ref{eq:rfhabhgrebzksbhdqvjoaret}), yields
\begin{equation}\label{eq:lkjhgffazfaghkfrehjkzvhv12}
 \tilde \varepsilon_\gamma^{-1}
 \leq  
 \mu_{\nu+1}^{-1}.
\end{equation}
Moreover, it holds that 
\begin{equation}\label{eq:lkjhgffazfaghkfrehjkzvhv13}
   \prod_{j=1}^k {\tilde{\mu}_j}^{d_j}
   = {\tilde{\mu}_1}^{d_1} \prod_{j=2}^k {\tilde{\mu}_j}^{d_j}
   \leq {{\mu}_1}^{d_1} \prod_{j=2}^k \prod_{i=1}^{d_{j-1}}  \mu_{\bar{d}_{j-2}+i}
   \leq \frac{{{\mu}_1}^{d_1}}{\mu_{\nu}^{d_k}} \prod_{j=1}^k \prod_{i=1}^{d_{j}}  \mu_{\bar{d}_{j-1}+i}
   = \frac{{{\mu}_1}^{d_1}}{\mu_{\nu}^{d_k}} \bar \mu^{\nu}_{\nu},
\end{equation}
where we used (\ref{eq:bhijavzohbeozzze}) and $\bar d_k = \nu$.
We proceed by noting that
the fraction on the right-hand side of (\ref{eq:lkjhgffazfaghkfrehjkzvhv13}) satisfies
\begin{equation}\label{eq:lkjhgffazfaghkfrehjkzvhv15}
  \left(\frac{{{\mu}_1}^{d_1}}{\mu_{\nu}^{d_k}}\right)^{1/\nu}
  = 1+O_{\nu \to \infty} \left(\frac{\log (\nu)}{\sqrt{\nu}}\right), 
\end{equation}
which is a consequence of $d_1 \simeq d_k \simeq \sqrt{\nu}$ and the regularly varying nature of the semi-axes.
Finally, it follows directly from (\ref{eq:jsqhkqoioiloio}) combined with the definitions of $k$ and $\jrdn$ that 
\begin{equation}\label{eq:lkjhgffazfaghkfrehjkzvhv14}
  \log\left(\kappa_k\right)
  = O_{\nu \to \infty} \left(\frac{\log (\nu)}{\sqrt{\nu}}\right),
  \quad \text{and, therefore,} \quad 
  \kappa_k
  = 1+O_{\nu \to \infty} \left(\frac{\log (\nu)}{\sqrt{\nu}}\right).
\end{equation}
Using (\ref{eq:lkjhgffazfaghkfrehjkzvhv12}), (\ref{eq:lkjhgffazfaghkfrehjkzvhv13}), (\ref{eq:lkjhgffazfaghkfrehjkzvhv15}), and (\ref{eq:lkjhgffazfaghkfrehjkzvhv14}) in (\ref{eq:lkjhgffazfaghkfrehjkzvhv11}), yields 
\begin{equation}\label{eq:kjnbvqkbefjzbvf}
 N \left(\varepsilon\semcol {\mathcal{E}}_{2}, \|\cdot\|_2 \right)^{1/\nu} \leq \frac{\bar \mu_\nu}{\mu_{\nu+1}}\left(1+O_{\nu \to \infty} \left(\frac{\log(\nu)}{\sqrt{\nu}}\right)\right).
\end{equation}
Next, we note that
$$
  \frac{\bar \mu_\nu^{\nu}}{\mu_{\nu+1}^{\nu}}
  = \frac{\bar \mu_{\nu+1}^{\nu+1}}{\mu_{\nu+1}^{\nu+1}},
$$  
which, when used in \eqref{eq:kjnbvqkbefjzbvf}, leads to
$$
 N \left(\varepsilon\semcol {\mathcal{E}}_{2}, \|\cdot\|_2 \right)^{1/(\nu+1)} \leq \frac{\bar \mu_{\nu+1}}{\mu_{\nu+1}}\left(1+O_{\nu \to \infty} \left(\frac{\log(\nu)}{\sqrt{\nu}}\right)\right).
$$
This can be reformulated as follows: There exist $K>0$ and $\nu^*$ such that, for all $\nu \geq \nu^*$,
\begin{equation}\label{eq:jkgozafjevzhsjkrtpzg2}
  N \left(\varepsilon\semcol {\mathcal{E}}_{2}, \|\cdot\|_2 \right)^{1/(\nu+1)} \leq \frac{\bar \mu_{\nu+1}}{\mu_{\nu+1}} \left(1+K\frac{\log(\nu)}{\sqrt{\nu}}\right)
  \eqqcolon \frac{\bar \mu_{\nu+1}}{\mu_{\nu+1}} s_{{\nu}}.
\end{equation}
In particular, we get, as a direct consequence of the definition (\ref{eq: definition of underbar n pt1mg3}) of upper effective dimension, 
that $\nu+1 \geq \ueffdim_{s_\nu}$. Since the semi-axes are non-increasing, we have 
\begin{equation*}
  \tilde \varepsilon < \mu_\nu \leq \mu_{\ueffdim_{s_\nu}-1}=\mu_{\leffdim_{s_\nu}},
  \quad \text{and hence } \quad
  \varepsilon \leq \mu_{\leffdim_{s_\nu}}\left(1+\varepsilon^{(\gamma-1)/b} \right).
\end{equation*}
To see that $s_\nu$ satisfies the desired asymptotic behavior, we combine (\ref{eq:rfhabhgrebzksbhdqvjoaret})
with 
the fact that the semi-axes are regularly varying with index $-b$. Finally, setting $\gamma =  \max\{1, b^2\}+1$ concludes the proof.

\subsection{Proof of Lemma~\ref{lem:kojihgfyfqsifgsquzehfv}}\label{sec:prooflemkojihgfyfqsifgsquzehfv}

We start by  writing 
 \begin{equation}\label{eq:kjhgfdfghjghgguyghybhbv3}
 \sum_{n=1}^{d} \log  \left(\frac{\mu_{n}}{\mu_{d}} \right)
 = \sum_{n=1}^{d} \log \left({\mu_{n}} \right) - d \, \log \left ({\mu_{d}} \right).
 \end{equation}
 Using (\ref{eq:jhkqdolfebzjbfqebtgbww}), it follows that
 \begin{equation}\label{eq:kjhgfdfghjghgguyghybhbv2}
\sum_{n=1}^{d} \log \left({\mu_{n}} \right)
= \sum_{n=1}^{d} \log \left(\frac{c_1}{n^{\alpha_1}} \right)
+ \sum_{n=1}^{d} \log \left(1+\frac{c_2}{c_1}n^{\alpha_1-\alpha_2} + o_{n \to \infty} \left(n^{\alpha_1-\alpha_2}\right) \right).
 \end{equation}
 We now treat each of the terms on the right-hand side of (\ref{eq:kjhgfdfghjghgguyghybhbv2}) separately.
 First, series-integral comparisons yield the upper bound
 \begin{align*}
 \sum_{n=1}^{d} \log \left(\frac{c_1}{n^{\alpha_1}} \right)
 &= d \, \log(c_1) - \alpha_1   \sum_{n=1}^{d} \log(n) \\
 &\leq d \, \log(c_1) - \alpha_1 \int_{1}^d \log(t)\, dt \\
 &= d \, \log(c_1) - \alpha_1 \left(d \log(d) - \frac{d-1}{\ln(2)}\right),
 \end{align*}
 and the lower bound
 \begin{align*}
 \sum_{n=1}^{d} \log \left(\frac{c_1}{n^{\alpha_1}} \right)
 &= d \, \log(c_1) - \alpha_1   \sum_{n=1}^{d} \log(n) \\
 &\geq d \, \log(c_1) - \alpha_1 \int_{1}^{d+1} \log(t)\, dt \\
 &= d \, \log(c_1) - \alpha_1 \left((d+1) \log(d+1) - \frac{d}{\ln(2)}\right).
 \end{align*}
 Combining these two bounds, we get 
 \begin{equation}\label{eq:kjhgfdfghjghgguyghybhbv1}
 \sum_{n=1}^{d} \log \left(\frac{c_1}{n^{\alpha_1}} \right)
 = - \alpha_1\, d \, \log(d) + \log\left({c_1}{e^{\alpha_1}}\right)\, d +O_{d\to\infty}\left(\log(d)\right).
 \end{equation}
 Next, using 
 \begin{equation*}
 \ln(1+x) = x + o_{x \to 0}(x),
 \end{equation*}
 we obtain
 \begin{equation*}
\sum_{n=1}^{d} \log \left(1+\frac{c_2}{c_1}n^{\alpha_1-\alpha_2} + o_{n \to \infty} \left(n^{\alpha_1-\alpha_2}\right) \right)
 =\sum_{n=1}^{d} \left(\frac{c_2}{c_1\ln(2)}n^{\alpha_1-\alpha_2} + o_{n \to \infty} \left(n^{\alpha_1-\alpha_2}\right) \right).
 \end{equation*}
Upon application of Lemma~\ref{lem:nksdcnjkdnqbvbbrytdioz}, we then get 
 \begin{equation*}
\sum_{n=1}^{d} \log \left(1+\frac{c_2}{c_1}n^{\alpha_1-\alpha_2} + o_{n \to \infty} \left(n^{\alpha_1-\alpha_2}\right) \right)
 =\sum_{n=1}^{d} \left(\frac{c_2}{c_1\ln(2)}n^{\alpha_1-\alpha_2} \right) + o_{d \to \infty} \left(d^{\alpha_1-\alpha_2+1}\right).
 \end{equation*}
 A series-integral comparison further yields
 \begin{equation}\label{eq:kjhgfdfghjghgguyghybhbv}
 \sum_{n=1}^{d} \log \left(1+\frac{c_2}{c_1}n^{\alpha_1-\alpha_2} + o_{n \to \infty} \left(n^{\alpha_1-\alpha_2}\right) \right)
 = \, \frac{c_2}{c_1\, \mathfrak{a}\, \ln(2)}d^{\mathfrak{a}} + o_{d \to \infty} \left(d^{\mathfrak{a}}\right),
 \end{equation}
  where we recall that $\mathfrak{a} = \alpha_1-\alpha_2+1$.
 Finally, combining
 \begin{align*}
 d \, \log \left ({\mu_{d}} \right)
 &= d\, \log \left(\frac{c_1}{d^{\alpha_1}} \right)
+d\, \log \left(1+\frac{c_2}{c_1}d^{\alpha_1-\alpha_2} + o_{d \to \infty} \left(d^{\alpha_1-\alpha_2}\right) \right)\\
&= d\, \log(c_1) - \alpha_1 \, d \, \log(d)+\frac{c_2}{c_1\ln(2) }d^{\mathfrak{a}} + o_{d \to \infty} \left(d^{\mathfrak{a}} \right), 
 \end{align*}
with (\ref{eq:kjhgfdfghjghgguyghybhbv3}), (\ref{eq:kjhgfdfghjghgguyghybhbv2}), (\ref{eq:kjhgfdfghjghgguyghybhbv1}), and (\ref{eq:kjhgfdfghjghgguyghybhbv}), the desired result follows.

\subsection{Proof of Lemma~\ref{lem:oijhbhjavqiaxvsregvfd}}\label{sec:prooflemmaoijhbhjavqiaxvsregvfd}

Starting from (\ref{eq:mljhgfhghgjgqsuk}), we can express $u$ according to
\begin{align}
u
&= \left(\frac{g(u) }{c_1} - \frac{c_2}{c_1{u}^{\alpha_2}} + o_{u \to \infty}\left(\frac{1}{{u}^{\alpha_2}}\right)\right)^{-\frac{1}{\alpha_1}}\label{eq:ajpnrvnzjknpkjdcap1}\\
&={c_1}^{1/\alpha_1}{g(u) }^{-1/\alpha_1}\left(1 - \frac{c_2}{g(u) {u}^{\alpha_2}} + o_{u \to \infty}\left(\frac{1}{g(u){u}^{\alpha_2}}\right)\right)^{-\frac{1}{\alpha_1}}.\label{eq:ajpnrvnzjknpkjdcap2}
\end{align}
In particular, we have 
\begin{equation}\label{eq:ajpnrvnzjknpkjdcap3}
u
={c_1}^{1/\alpha_1}{g(u) }^{-1/\alpha_1}\left(1 + o_{u \to \infty}\left(1\right)\right).
\end{equation}
Replacing $u$ in the terms $u^{\alpha_{2}}$ on the right-hand side of (\ref{eq:ajpnrvnzjknpkjdcap2}) by the right-hand side of
(\ref{eq:ajpnrvnzjknpkjdcap3}), it follows that 
\begin{align*}
u
&= {c_1}^{1/\alpha_1}{g(u)}^{-1/\alpha_1}\left(1 - \frac{c_2}{{c_1}^{\alpha_2/\alpha_1}g(u)^{1-\alpha_2/\alpha_1}} + o_{u \to \infty}\left(\frac{1}{g(u)^{1-\alpha_2/\alpha_1}}\right)\right)^{-\frac{1}{\alpha_1}}\\
&={c_1}^{1/\alpha_1}{g(u)}^{-1/\alpha_1}\left(1 + \frac{c_2}{\alpha_1{c_1}^{\alpha_2/\alpha_1}}g(u)^{\frac{\alpha_2-\alpha_1}{\alpha_1}} + o_{u \to \infty}\left(g(u)^{\frac{\alpha_2-\alpha_1}{\alpha_1}}\right)\right)\\
&=
{c_1}^{1/\alpha_1}{g(u)}^{-1/\alpha_1} 
+ \frac{c_2}{\alpha_1{c_1}^{\frac{\alpha_2-1}{\alpha_1}}}g(u)^{\frac{\alpha_2-\alpha_1-1}{\alpha_1}} 
+ o_{u \to \infty}\left(g(u)^{\frac{\alpha_2-\alpha_1-1}{\alpha_1}}\right),
\end{align*}
thereby concluding the proof.

\subsection{Proof of Lemma~\ref{lem:firstellipinfnorm}}\label{sec:prooflemfirstellipinfnorm}

An explicit $\varepsilon$-covering of $\mathcal{E}_\infty(\{\mu_n\}_{n \in \mathbb{N}^*})$ is obtained by covering each dimension $n\in \mathbb{N}^*$ individually with 
\begin{equation*}
\left\lceil \frac{2\mu_n}{2\varepsilon} \right\rceil
=\left\lceil \frac{\mu_n}{\varepsilon} \right\rceil
\end{equation*}
regularly spaced points.
This immediately yields 
\begin{equation}\label{eq:lkiuyftqedfezfre1}
N\left(\varepsilon\semcol \mathcal{E}_\infty(\{\mu_n\}_{n \in \mathbb{N}^*}), \|\cdot\|_\infty\right)
\leq \prod_{n=1}^\infty \left\lceil \frac{\mu_n}{\varepsilon} \right\rceil,
\end{equation}
for all $\varepsilon >0$.
Conversely, a covering in the sup$_{n \in \mathbb{N}^*}$-norm requires at least
\begin{equation*}
\left\lceil \frac{2\mu_n}{2\varepsilon} \right\rceil
=\left\lceil \frac{\mu_n}{\varepsilon} \right\rceil
\end{equation*}
points in each dimension,
which, in turn, allows us to conclude that
\begin{equation}\label{eq:lkiuyftqedfezfre2}
N\left(\varepsilon\semcol \mathcal{E}_\infty(\{\mu_n\}_{n \in \mathbb{N}^*}), \|\cdot\|_\infty\right)
\geq \prod_{n=1}^\infty \left\lceil \frac{\mu_n}{\varepsilon} \right\rceil.
\end{equation}
Taking the logarithm in (\ref{eq:lkiuyftqedfezfre1}) and (\ref{eq:lkiuyftqedfezfre2}), finalizes the proof.

\subsection{Proof of Lemma~\ref{lem:firstellipinfnorm2}}\label{sec:prooflemfirstellipinfnorm2}

%\textcolor{blue}{
We let $k_n \coloneqq \left\lceil \frac{\mu_n}{\varepsilon} \right\rceil$ and note that, owing to
\begin{equation}\label{eq:1339}
\sum_{n=1}^\infty \log \left(\left\lceil \frac{\mu_n}{\varepsilon} \right\rceil\right)
= \sum_{n=1}^\infty \log \left(k_n\right) 
= \sum_{n=1}^\infty \mathbbm{1}_{\{k_n\geq 2\}}\log \left(k_n\right),
\end{equation}
only values of $n \in \mathbb{N}^*$ with $k_n \geq 2$ need to be considered.
Given an integer $k\geq 2$,
the set of values $n \in \mathbb{N}^*$ for which $k_n = k$ can be characterized through 
%or equivalently 
\begin{equation*}
(k -1)\varepsilon < \mu_n \leq k \varepsilon,
\end{equation*}
or equivalently as
\begin{equation*}
  \left\{n\in \mathbb{N}^* |\, \mu_n > (k-1)  \varepsilon \right\} \setminus \left\{n\in \mathbb{N}^* |\, \mu_n > k  \varepsilon \right\}.
\end{equation*}
We can hence conclude that there are
$M_{k-1}(\varepsilon)- M_k(\varepsilon)$
integers $n$ for which $k_n=k$.
Using this observation in (\ref{eq:1339}), we obtain
\begin{equation*}
  \sum_{n=1}^\infty \log \left(\left\lceil \frac{\mu_n}{\varepsilon} \right\rceil\right)
  = \sum_{n=1}^\infty \mathbbm{1}_{\{k_n\geq 2\}}\log \left(k_n\right) 
  =\sum_{k=2}^\infty \log(k) \left(M_{k-1}(\varepsilon)- M_k(\varepsilon)\right),
\end{equation*}
which is the desired result.%}

\subsection{Proof of Lemma~\ref{lem:lbincorinfellnorminf}}\label{sec:prooflbincorinfellnorminf}

With
\begin{equation*}%\label{eq:defeffectivedimbutdontsayitshh}
d \coloneqq \left \lfloor \frac{c}{\varepsilon} \right \rfloor,
\end{equation*}
we have
\begin{equation}\label{eq:log-decomp}
\sum_{k=1}^\infty \log \left(1+\frac{1}{k} \right) \left(\left\lceil \left(\frac{c}{k\varepsilon}\right)^{1/b} \right\rceil -1\right) 
= 
\sum_{k=1}^d \log \left(1+\frac{1}{k} \right) \left\lceil \left(\frac{c}{k\varepsilon}\right)^{1/b} \right\rceil - \sum_{k=1}^d \log \left(1+\frac{1}{k} \right) .
\end{equation}
Next, we note that
\begin{equation}\label{eq:log-decomp2}
\sum_{k=1}^d \log \left(1+\frac{1}{k} \right) 
= \sum_{k=1}^d \left(\log \left(k+1 \right) -\log \left(k \right)\right) =
\log \left(d+1 \right).
\end{equation}
Now, observe that
\begin{equation}\label{eq:1139}
\sum_{k=1}^d \log \left(1+\frac{1}{k} \right) \left\lceil \left(\frac{c}{k\varepsilon}\right)^{1/b} \right\rceil 
\geq  \frac{c^{1/b}}{\varepsilon^{1/b}}\sum_{k=1}^d \frac{\log \left(1+\frac{1}{k} \right)}{k^{1/b}},
\end{equation}
where the right-hand side of (\ref{eq:1139}) can be decomposed according to
\begin{equation}\label{eq:decomprhscorinfell}
\sum_{k=1}^d \frac{\log \left(1+\frac{1}{k} \right)}{k^{1/b}}
= \sum_{k=1}^\infty \frac{\log \left(1+\frac{1}{k} \right)}{k^{1/b}} - \sum_{k=d+1}^\infty \frac{\log \left(1+\frac{1}{k} \right)}{k^{1/b}}.
\end{equation}
Employing the bound $\ln(1+x)\leq x$, for $x \geq 0$, we obtain
\begin{equation}\label{eq:decomprhscorinfell2}
\sum_{k=d+1}^\infty \frac{\log \left(1+\frac{1}{k} \right)}{k^{1/b}}
\leq \frac{1}{\ln(2)}\sum_{k=d+1}^\infty \frac{1}{k^{1/b+1}}.
\end{equation}
A sum-integral comparison yields
\begin{equation}\label{eq:decomprhscorinfell3}
\sum_{k=d+1}^\infty \frac{1}{k^{1/b+1}}
\leq \int_d^\infty \frac{dt}{t^{1/b+1}}
= \frac{b}{ d^{1/b}}
= O_{\varepsilon \to 0}\left(\varepsilon^{1/b}\right).
\end{equation}
Combining (\ref{eq:log-decomp})--(\ref{eq:decomprhscorinfell3}) and noting that $\log(d+1)=O_{\varepsilon \to 0}(\log(\varepsilon^{-1}))$, delivers the desired result.

\subsection{Proof of Lemma~\ref{lem: Convergence of Cesaro-mean; regularly varying case}}\label{sec:prooflemconvcesmeanregvar}

We first observe that
\begin{equation}\label{eq: quaquall}
\frac{1}{N}\sum_{n=1}^N \log \left(\frac{\mu_n}{\mu_N}\right)
= \frac{1}{N}\sum_{n=1}^N \left( \log \left(\frac{\mu_n n^b}{\mu_N N^b}\right) + b \log \left(\frac{N}{n}\right)\right),
\quad \text{for all } N \in \mathbb{N}^*.
\end{equation}
A direct application of Stirling's formula yields 
  \begin{equation}\label{eq:1110}
\frac{1}{N}\sum_{n=1}^N \ln \left(\frac{N}{n}\right)
= \frac{1}{N} \ln \left(\frac{N^N}{N!}\right)
=1 + O_{N\to \infty} \left(\frac{\ln N}{N}\right).
\end{equation}
Combining (\ref{eq: quaquall}) and (\ref{eq:1110}), we obtain
\begin{equation}\label{eq:kljnbhgfytxytfzdeuzferf}
\frac{1}{N}\sum_{n=1}^N \log \left(\frac{\mu_n}{\mu_N}\right)
=\frac{b}{\ln(2)} +  \frac{1}{N}\sum_{n=1}^N  \log \left(\frac{\mu_n n^b}{\mu_N N^b}\right) + O_{N\to \infty} \left(\frac{\log N}{N}\right).
\end{equation}
We next require the following lemma.

\begin{lemma}\label{lem: Convergence of Cesaro-mean; slowly varying case}
Let $\{\nu_n\}_{n\in \mathbb{N}^*}$ be a slowly varying sequence.
Then, we have
\begin{equation*}
\lim_{N \to \infty}
\frac{1}{N}\sum_{n=1}^N \log \left(\frac{\nu_n}{\nu_N}\right)
=0.
\end{equation*}
\end{lemma}

\begin{proof}[Proof.]
See Appendix \ref{sec:proofoflemmaconvcesaromeanslowly}.
\end{proof}

\noindent
Application of Lemma~\ref{lem: Convergence of Cesaro-mean; slowly varying case} to the sequence $\{\mu_n n^b\}_{n \in \mathbb{N}^*}$ in (\ref{eq:kljnbhgfytxytfzdeuzferf}) then finishes the proof.

\subsection{Proof of Lemma~\ref{lme:1335199}}\label{sec:prooflemgivenameplssssempai}

For convenience, let 
\begin{equation*}
b_{(-)} \coloneqq \liminf_{k \to \infty} b_k
\quad \text{and} \quad
b_{(+)} \coloneqq \limsup_{k \to \infty} b_k.
\end{equation*}
Next, choose a sequence $\{k_n\}_{n\in\mathbb{N}^*}$ satisfying 
\begin{equation*}
k_n \to_{n \to \infty} \infty
\quad \text{and} \quad  
\lim_{n\to\infty} \frac{\varphi(a_{k_n} b_{k_n})}{\varphi(a_{k_n})}
=\liminf_{k \to \infty} \frac{\varphi(a_{k} b_{k})}{\varphi(a_{k})} .
\end{equation*}
Such a sequence must exist by definition of the lim inf. 
Let $\eta \in (0,b_{(+)}) $ and choose $N_\eta \in \mathbb{N}^*$ such that
\begin{equation}\label{eq:dsjkofbsbkjvgfsffthklp}
b_{(+)} + \eta > b_{k_n},
\end{equation}
for all $n \geq N_\eta$. 
Such an $N_\eta$ exists by definition of the lim sup. 
Now, using that $k_n \to_{n\to\infty} \infty$ and $a_{n} \to_{n\to\infty} \infty$, we get $a_{k_n} \to_{n\to\infty} \infty$.
In particular, there exists $N_\eta' \in \mathbb{N}^*$ such that $a_{k_n}>0$, for all $n \geq N_\eta'$.
As $\varphi$ is non-increasing by assumption, it follows from (\ref{eq:dsjkofbsbkjvgfsffthklp}) that 
\begin{equation}\label{eq:jkfdsbjhbshbsiubqqqazeer}
\frac{\varphi(a_{k_n} b_{k_n})}{\varphi(a_{k_n})}
\geq \frac{\varphi\left(a_{k_n} \left(b_{(+)} + \eta\right)\right)}{\varphi(a_{k_n})},
\end{equation}
for all $n \geq \max\{N_\eta, N_\eta'\}$.
Moreover, since $\varphi$ is regularly varying with index $-b$, we have 
\begin{equation*}
\lim_{n\to\infty}\frac{\varphi\left(a_{k_n} \left(b_{(+)} + \eta\right)\right)}{\varphi(a_{k_n})}
= \left(b_{(+)} + \eta\right)^{-b}.
\end{equation*}
Therefore, taking the limit $n\to\infty$ in (\ref{eq:jkfdsbjhbshbsiubqqqazeer}), we obtain 
\begin{equation}\label{eq:134011}
\liminf_{k \to \infty} \frac{\varphi(a_{k} b_{k})}{\varphi(a_{k})} 
=\lim_{n\to\infty} \frac{\varphi(a_{k_n} b_{k_n})}{\varphi(a_{k_n})}
\geq \lim_{n\to\infty} \frac{\varphi\left(a_{k_n} \left(b_{(+)} + \eta\right)\right)}{\varphi(a_{k_n})}
=\left(b_{(+)} + \eta\right)^{-b}.
\end{equation}
Next, we choose a sequence $\{k'_n\}_{n\in\mathbb{N}^*}$ satisfying 
\begin{equation*}
k_n' \to_{n \to \infty} \infty
\quad \text{and} \quad  
\lim_{n\to\infty} b_{k_n'}
=b_{(+)} .
\end{equation*}
Such a sequence exists by definition of the lim sup.
In particular, there exists $M_\eta \in \mathbb{N}^*$ such that 
\begin{equation}\label{eq:dsjkofbsbkjvgfsffthklp2}
b_{(+)} - \eta < b_{k'_n},
\end{equation}
for all $n \geq M_\eta$.
Now, using that $k_n' \to_{n\to\infty} \infty$ and $a_{n} \to_{n\to\infty} \infty$, we get $a_{k'_n} \to_{n\to\infty} \infty$.
In particular, there exists $M_\eta' \in \mathbb{N}^*$ such that $a_{k'_n}>0$, for all $n \geq M_\eta'$.
As $\varphi$ is non-increasing by assumption, it follows from (\ref{eq:dsjkofbsbkjvgfsffthklp2}) that
\begin{equation}\label{eq:sjfkdlhbfjksnbkdjgnsgkbjsnfkjbn}
\frac{\varphi(a_{k'_n} b_{k'_n})}{\varphi(a_{k'_n})}
\leq \frac{\varphi\left(a_{k'_n} \left(b_{(+)} - \eta\right)\right)}{\varphi(a_{k'_n})},
\end{equation}
for all $n \geq \max\{M_\eta, M_\eta'\}$.
Moreover, since $\varphi$ is regularly varying with index $-b$ and $a_{k'_n} \to_{n\to\infty} \infty$, we have
\begin{equation*}
\lim_{n\to\infty} \frac{\varphi\left(a_{k'_n} \left(b_{(+)} - \eta\right)\right)}{\varphi(a_{k'_n})}
= \left(b_{(+)} - \eta\right)^{-b}.
\end{equation*}
Hence, taking the limit ${n\to\infty}$ in (\ref{eq:sjfkdlhbfjksnbkdjgnsgkbjsnfkjbn}), we obtain
\begin{equation}\label{eq:134011TT}
\lim_{n\to\infty} \frac{\varphi(a_{k'_n} b_{k'_n})}{\varphi(a_{k'_n})}
\leq\lim_{n\to\infty} \frac{\varphi\left(a_{k'_n} \left(b_{(+)} - \eta\right)\right)}{\varphi(a_{k'_n})}
= \left(b_{(+)} - \eta\right)^{-b}.
\end{equation}
Combining (\ref{eq:134011}) and (\ref{eq:134011TT}), yields
\begin{equation*}
\left(b_{(+)} + \eta\right)^{-b}
\leq \liminf_{k \to \infty} \frac{\varphi(a_{k} b_{k})}{\varphi(a_{k})} 
\leq \lim_{n\to\infty} \frac{\varphi(a_{k'_n} b_{k'_n})}{\varphi(a_{k'_n})}
\leq \left(b_{(+)} - \eta\right)^{-b},
\end{equation*}
for all $\eta \in (0,b_{(+)}) $. 
In particular, choosing $\eta$ arbitrarily small results in
\begin{equation*}
\liminf_{k \to \infty} \frac{\varphi(a_{k} b_{k})}{\varphi(a_{k})} 
=\left(\limsup_{k \to \infty} b_{k}\right)^{-b},
\end{equation*}
which is the first statement to be established. The proof of the second statement follows analogously and will hence not be detailed.

\subsection{Proof of Lemma~\ref{lem: Convergence of Cesaro-mean; slowly varying case}}\label{sec:proofoflemmaconvcesaromeanslowly}

%\textcolor{blue}{
Let $\{u_n\}_{\mathbb{N}^*}$ be a slowly varying sequence.
By inspection of the proof of \cite[Theorem~1.9.8]{binghamRegularVariation1987}, it then follows that 
there exists a sequence of positive real numbers $\{v_n\}_{\mathbb{N}^*}$ such that 
\begin{equation}\label{eq:gdgdkfjshbkjghhkjfgvdfrfd}
  u_n = v_n \left(1+o_{n\to\infty}\left(1\right)\right)
  \quad \text{and} \quad 
  \frac{v_{n+1}}{v_n}
  = 1+o_{n\to\infty}\left(\frac{1}{n}\right).
\end{equation}
In particular, setting $w_n=u_n/v_n - 1$, for all $n\in \mathbb{N}^*$, we have $w_n =o_{n\to\infty}(1)$ and therefore
\begin{align*}
  \limsup_{N \to \infty} \frac{1}{N}\sum_{n=1}^N \log \left(\frac{u_n}{u_N}\right)
   &\leq   \limsup_{N \to \infty} \frac{1}{N}\sum_{n=1}^N \log \left(\frac{v_n}{v_N}\right)
   + \limsup_{N \to \infty} \frac{1}{N}\sum_{n=1}^N \log \left(\frac{1+w_n}{1+w_N}\right)\\
   &\leq  \limsup_{N \to \infty} \frac{1}{N}\sum_{n=1}^N \log \left(\frac{v_n}{v_N}\right)
   + \limsup_{N \to \infty} \frac{1}{N}\sum_{n=1}^N w_n \\
   &=  \limsup_{N \to \infty} \frac{1}{N}\sum_{n=1}^N \log \left(\frac{v_n}{v_N}\right),
\end{align*}
where the last step is owing to 
Lemma~\ref{lem:nksdcnjkdnqbvbbrytdioz}.
Next, we prove that 
  \begin{equation}\label{eq:nsjkfnkjbnkjlkjqksndfnjhhqkd}
    \limsup_{N \to \infty} \frac{1}{N}\sum_{n=1}^N \log \left(\frac{v_n}{v_N}\right) \leq  0.
  \end{equation}
To this end, we need the following lemma.

%\textcolor{blue}{
\begin{lemma}\label{lem:1341}
Let $\{b_n\}_{n\in \mathbb{N}^*}$ be a sequence of positive real numbers.
Then, 
\begin{equation*}
\limsup_{n \to \infty} \frac{\left(\prod_{j=1}^n b_j \right)^{1/n}}{b_n}
\leq \limsup_{n \to \infty} \left(\frac{b_n}{b_{n+1}}\right)^{n}
\cdot \limsup_{n \to \infty} \frac{b_{n+1}}{b_n}.
\end{equation*}
\end{lemma}%}

\begin{proof}[Proof.]
  A similar result was stated without proof in \cite[Eq. (4.7)]{luschgyFunctionalQuantizationGaussian2002}. For completeness, we provide a proof in Appendix \ref{sec:prooflemmabgg1343}.
\end{proof}

\noindent
Direct application of Lemma~\ref{lem:1341} with $b_j=v_j$ now yields
\begin{equation}\label{eq:13412504}
\limsup_{N \to \infty} \frac{\left(\prod_{n=1}^N v_n\right)^{1/N}}{v_N} 
\leq \limsup_{N \to \infty} \left(\frac{v_N}{v_{N+1}}\right)^{N} \cdot
\limsup_{N \to \infty} \frac{v_{N+1}}{v_N}.
\end{equation}
%\textcolor{blue}{
With the second part of (\ref{eq:gdgdkfjshbkjghhkjfgvdfrfd}), we get 
\begin{equation}\label{eq:1342}
\lim_{N \to \infty}\left(\frac{v_N}{v_{N+1}}\right)^{N}
=1
\quad \text{and} \quad
\lim_{N \to \infty}\frac{v_{N+1}}{v_N}
=1.
\end{equation}%}
Taking logarithms in (\ref{eq:13412504}) hence yields (\ref{eq:nsjkfnkjbnkjlkjqksndfnjhhqkd}) according to
%\textcolor{blue}{
\begin{align*}
\limsup_{N \to \infty} \frac{1}{N}\sum_{n=1}^N \log \left(\frac{v_n}{v_N}\right)
&\leq 
\log \left(\limsup_{N \to \infty} \left(\frac{v_N}{v_{N+1}}\right)^{N}\right)
+\log \left(\limsup_{N \to \infty} \frac{v_{N+1}}{v_N}\right)\\
&\stackrel{(\ref{eq:1342})}{=} 0.
\end{align*}%}
Putting everything together, we have established that 
\begin{equation}\label{eq:nsjkfnkjbnkjlkjqksndfnjhhqkd77}
  \limsup_{N \to \infty} \frac{1}{N}\sum_{n=1}^N \log \left(\frac{u_n}{u_N}\right) \leq  0.
\end{equation}%}

\noindent
%\textcolor{blue}{
Next, application of the representation theorem \cite[Theorem~1.9.7]{binghamRegularVariation1987} allows
us to conclude that with $\nu_n$ slowly varying $1/{\nu_n}$ is also slowly varying. We can hence 
evaluate (\ref{eq:nsjkfnkjbnkjlkjqksndfnjhhqkd77}) for $u_n = \nu_n$ and $u_n = 1/\nu_n$, for all $n \in \mathbb{N}^*$, to get, respectively,
\begin{equation*}
  \limsup_{N \to \infty} \frac{1}{N}\sum_{n=1}^N \log \left(\frac{\nu_n}{\nu_N}\right)
\leq 0
\quad \text{and} \quad 
\liminf_{N \to \infty} \frac{1}{N}\sum_{n=1}^N \log \left(\frac{\nu_n}{\nu_N}\right)
\geq 0,
\end{equation*}%
thereby concluding the proof.
%}

\subsection{Proof of Lemma~\ref{lem:1341}}\label{sec:prooflemmabgg1343}

%\textcolor{blue}{
We first show that for an arbitrary sequence of positive real numbers $\{a_n\}_{n\in \mathbb{N}^*}$,
\begin{equation}\label{eq:kjfdkhsgjjhhhhbshvgfcfdgf}
\limsup_{n \to \infty} \left(\prod_{j=1}^n a_j \right)^{1/n}
\leq \limsup_{n \to \infty} a_n.
\end{equation}
To this end, fix $\eta > 0$ and let $k_\eta \in \mathbb{N}^*$ 
be such that, for all $k>k_\eta$, 
\begin{equation*}
a_{k} < \limsup_{n \to \infty} a_n + \eta.
\end{equation*}
Then, for all $k>k_\eta$, we have 
\begin{align*}
 \left(\prod_{j=1}^{k} a_j \right)^{1/{k}}
 &<\left(\left(\limsup_{n \to \infty} a_n + \eta\right)^{{k}-k_\eta}\prod_{j=1}^{k_\eta} a_j \right)^{1/{k}} \\
 &= \left(\limsup_{n \to \infty} a_n + \eta\right)\left(\frac{\prod_{j=1}^{k_\eta} a_j}{\left(\limsup_{n \to \infty} a_n + \eta\right)^{k_\eta}} \right)^{1/{k}}\\
&\xrightarrow[k\to \infty]{} \limsup_{n \to \infty} a_n + \eta.
\end{align*}
The inequality 
\begin{equation*}
 \limsup_{n \to \infty} \left(\prod_{j=1}^n a_j \right)^{1/n}
 < \limsup_{n \to \infty} a_n + \eta
\end{equation*}
hence holds for all $\eta>0$, which, in turn, implies (\ref{eq:kjfdkhsgjjhhhhbshvgfcfdgf}).
%}

Next, set $a_j = (b_j/b_{j+1})^j$, for $j\in \mathbb{N}^*$, and note that
\begin{equation}\label{eq:1343250041}
\frac{\left(\prod_{j=1}^n b_j \right)^{1/n}}{b_{n+1}}
= \left(\prod_{j=1}^{n} a_j\right)^{1/n}, \quad \text{for all } n\in \mathbb{N}^*.
\end{equation}
Application of (\ref{eq:kjfdkhsgjjhhhhbshvgfcfdgf}) now yields
\begin{equation}\label{eq:1343250042}
\limsup_{n \to \infty} \left(\prod_{j=1}^{n} a_j\right)^{1/n} 
\leq \limsup_{n \to \infty} \, a_n
= \limsup_{n \to \infty} \left(\frac{b_n}{b_{n+1}}\right)^n.
\end{equation}
Combining (\ref{eq:1343250041}) and (\ref{eq:1343250042}), we obtain
\begin{equation*}
\limsup_{n \to \infty} \frac{\left(\prod_{j=1}^n b_j \right)^{1/n}}{b_{n+1}}
\leq \limsup_{n \to \infty} \left(\frac{b_n}{b_{n+1}}\right)^n,
\end{equation*}
%\textcolor{blue}{
and therefore
\begin{align*}
    \limsup_{n \to \infty} \frac{\left(\prod_{j=1}^n b_j \right)^{1/n}}{b_{n}}
    &\leq \limsup_{n \to \infty} \frac{\left(\prod_{j=1}^n b_j \right)^{1/n}}{b_{n+1}}\cdot \limsup_{n \to \infty} \frac{b_{n+1}}{b_n}\\
&\leq \limsup_{n \to \infty} \left(\frac{b_n}{b_{n+1}}\right)^n\cdot \limsup_{n \to \infty} \frac{b_{n+1}}{b_n},
\end{align*}%}
%\textcolor{blue}{
which concludes the proof.
%}

\subsection{Statement and Proof of Lemma~\ref{lem:nksdcnjkdnqbvbbrytdioz}}

\begin{lemma}\label{lem:nksdcnjkdnqbvbbrytdioz}
Let $a>-1$ and let $\{u_i\}_{i\in\mathbb{N}^*}$ be such that $u_d = o_{d\to \infty}(d^{a})$.
Then, 
\begin{equation*}
\sum_{i=1}^d u_i = o_{d\to\infty}\left(d^{a+1}\right).
\end{equation*}
\end{lemma}
\noindent
\begin{proof}
Fix $\varepsilon > 0$ and take $m\in \mathbb{N}^*$ such that 
\begin{equation*}
u_i \leq \frac{a+1}{2} i^{a} \varepsilon, 
\quad \text{for all } i > m.
\end{equation*}
Such an $m$ is guaranteed to exist by the assumption $u_d = o_{d\to \infty}(d^{a})$.
Now, let $d\geq m$ and split up the sum under consideration according to 
\begin{equation}\label{eq:nvkdnaerkjqnvfksqcnjkn11}
\frac{1}{(d+1)^{a+1}} \sum_{i=1}^{d} u_i 
= \frac{1}{(d+1)^{a+1}} \sum_{i=1}^m u_i + \frac{1}{(d+1)^{a+1}} \sum_{i=m+1}^d u_i,
\end{equation}
with the convention that the second term on the right-hand side of \eqref{eq:nvkdnaerkjqnvfksqcnjkn11} is equal to zero for $d=m$.
We treat the two sums on the right-hand side of \eqref{eq:nvkdnaerkjqnvfksqcnjkn11} individually.
First, we note that there exists a $d^* \in \mathbb{N}^*$ such that 
\begin{equation}\label{eq:nvkdnaerkjqnvfksqcnjkn12}
 \frac{1}{(d+1)^{a+1}} \sum_{i=1}^m u_i \leq \frac{\varepsilon}{2},
 \quad \text{whenever } d\geq d^*.
 \end{equation} 
 Second, a series-integral comparison yields
 \begin{equation}\label{eq:nvkdnaerkjqnvfksqcnjkn13}
 \frac{1}{(d+1)^{a+1}} \sum_{i=m+1}^d u_i
 \leq  \frac{(a+1)\, \varepsilon}{2(d+1)^{a+1}}  \sum_{i=m+1}^d i^{a} 
 \leq  \frac{(a+1)\, \varepsilon}{2(d+1)^{a+1}}  \frac{(d+1)^{a+1}}{a+1} = \frac{\varepsilon}{2}.
 \end{equation}
 Using (\ref{eq:nvkdnaerkjqnvfksqcnjkn12}) and (\ref{eq:nvkdnaerkjqnvfksqcnjkn13}) in (\ref{eq:nvkdnaerkjqnvfksqcnjkn11}), establishes that for all $\varepsilon >0$ there exists a $d^* \in \mathbb{N}^*$ such that
\begin{equation*}
\sum_{i=1}^{d} u_i 
= \varepsilon (d+1)^{a+1} ,
 \quad \text{for all } d\geq d^*,
\end{equation*}
which yields the desired result upon noting that $\varepsilon$ can be taken arbitrarily small.
\end{proof}

\end{document}